\documentclass[preprint,12pt]{elsarticle}

\usepackage{amssymb}
\usepackage{amsmath}
\usepackage{amsthm}

\usepackage{url}
\usepackage{doi} 

\usepackage[capitalize]{cleveref}

\theoremstyle{theorem}
\newtheorem{theorem}{Theorem}[section]
\newtheorem{lemma}[theorem]{Lemma}
\newtheorem{proposition}[theorem]{Proposition}
\newtheorem{corollary}[theorem]{Corollary} 

\theoremstyle{definition}
\newtheorem{definition}[theorem]{Definition}
\newtheorem{remark}[theorem]{Remark} 
\newtheorem{example}[theorem]{Example}

\journal{JPAA}

\usepackage{etoolbox} 
\usepackage{xcolor} 
\usepackage{xspace}
\usepackage{amsmath}
\usepackage{bbold}
\usepackage[all]{xy}
\usepackage{mathtools}
\usepackage[mathscr]{euscript}

 \def\ie{{\slshape i.e.}\xspace}
\def\cf{cf.~} 
\newcommand{\refItem}[2]{\cref{#1}(\ref{#1:#2})}

\def\ple#1{\ensuremath{{\langle #1 \rangle }}} 

\def\vuoto{}
\newcommand{\FUN}[4]{\ensuremath{{#2}#1{#3}\rightarrow{#4}}}
\newcommand{\fun}[3]{\relax\def\testa{#1}\relax\ifx\testa\vuoto
  \relax\FUN{}{}{{#2}}{{#3}}\else\relax\FUN{:}{{#1}}{{#2}}{{#3}}\fi}

\newcommand{\id}{\mathsf{id}}
\newcommand{\PW}[2][]{\ensuremath{\mathop{\mathscr{P}_{#1}{#2}}}}
\newcommand{\pw}[1]{\relax\def\testa{#1}\relax\ifx\testa\vuoto
  \relax\PW{}\else\relax\PW{\left(#1\right)}\fi}
\newcommand{\fpw}[1]{\relax\def\testa{#1}\relax\ifx\testa\vuoto
  \relax\PW[\omega]{}\else\relax\PW[\omega]{\left(#1\right)}\fi}

\newcommand{\eqc}[2][]{[#2]_{#1}} 

\newcommand{\R}{\mathbb{R}} 
\newcommand{\RPos}{\R_{\ge 0}} 
\newcommand{\Bool}{\mathbb{B}}
\newcommand{\PP}{\pw{}} 
 
\newcommand{\Rel}{\mathsf{Rel}} 
\newcommand{\BRel}{\mathsf{BRel}}

\def\tt{\ensuremath{\textsf{t\kern-.3ex t}}}
\def\ff{\ensuremath{\textsf{f\kern-.3ex f}}}
\let\Land\wedge

\let\ForalL\forall \def\Forall#1.{\ForalL_{#1}}
\let\ExistS\exists \def\Exists#1.{\ExistS_{#1}}

\DeclareFontFamily{OT1}{pzc}{}
\DeclareFontShape{OT1}{pzc}{m}{it}{<->s*[1.30]pzcmi7t}{}
\DeclareMathAlphabet{\mathpzc}{OT1}{pzc}{m}{it}
\def\ct#1{\ensuremath{\mathpzc{#1}}}
\def\Ct#1{\ensuremath{\mathbf{#1}}}

\def\op{^{\mbox{\normalfont\scriptsize op}}}

\newcommand{\bop}[1]{(#1\times#1)\op} 
\def\blank{\mathchoice{\mbox{--}}{\mbox{--}}
{\mbox{\scriptsize--}}{\mbox{\tiny--}}}

\newcommand{\CC}{\ct{C}\xspace}

\newcommand{\D}{\ct{D}\xspace}

\newcommand{\oneAr}[3]{\ensuremath{{#1}:{#2}\rightarrow{#3}}}
\newcommand{\twoAr}[3]{\ensuremath{{#1}:{#2}\Rightarrow{#3}}}
\def\tdot{\textbf{.}}
\def\lsta{\vrule depth4pt width0pt}
\def\lstb{\vrule height5pt width0pt}
\def\arnta[#1]{\ar[#1]|-*=0[@]{\lsta\tdot}}
\def\arntb[#1]{\ar[#1]|-*=0[@]{\lstb\tdot}}
\newcommand{\nt}[3]{\ensuremath{{#1}:{#2} \stackrel{\makebox{\kern-.3ex\tdot}}\rightarrow{#3}}}
\newcommand{\lnt}[3]{\ensuremath{{#1}:{#2} \stackrel{\makebox{\kern-.3ex\tdot}}\rightarrow_l{#3}}}
\newcommand{\Hom}[3]{\ensuremath{{#1}({#2},{#3})}}
\newcommand{\Id}{\mathsf{Id}}
\newcommand{\ID}{\mathrm{Id}}

\newcommand{\Set}{\ct{Set}\xspace}
\newcommand{\BSet}{\ct{BSet}\xspace}

\newcommand{\Pos}{\ct{Pos}\xspace}
\newcommand{\Top}{\ct{Top}\xspace}

\newcommand{\SMet}[1]{\ensuremath{{#1}\text{-}\ct{Met}_{\mathsf{s}}}\xspace} 
\newcommand{\Dtn}{\Ct{Dtn}\xspace}
\newcommand{\EED}{\Ct{EED}\xspace}
\newcommand{\RDtn}{\Ct{RD}\xspace}

\newcommand{\RDtnl}{\Ct{RD_l}\xspace} 
 
\newcommand{\QRDtn}{\Ct{QRD}\xspace} 
 
\newcommand{\QRDtnl}{\Ct{QRD_l}\xspace} 
\newcommand{\ERDtn}{\Ct{ERD}\xspace} 
 
\newcommand{\EQRDtn}{\Ct{EQRD}\xspace} 
 
\newcommand{\ERUCRDtn}{\ensuremath{\Ct{ERD}_{\RUC}\xspace}}
\newcommand{\One}{\mathbf{1}}

\newcommand{\order}{\leq} 
\newcommand{\PDoc}{P}

\newcommand{\QDoc}{Q} 
\newcommand{\RDoc}{R}

\newcommand{\SDoc}{S} 
\newcommand{\reidx}[1]{_{#1}}

\newcommand{\fn}[1]{\widehat{#1}}
\newcommand{\lift}[1]{\overline{#1}}

\newcommand{\relr}{\alpha}
\newcommand{\rels}{\beta}
\newcommand{\relt}{\gamma} 
\newcommand{\eqrelr}{\rho}
\newcommand{\eqrels}{\sigma}
\newcommand{\eqrelt}{\tau} 
\newcommand{\rid}{\mathsf{d}}
\newcommand{\rcomp}{\mathop{\mathbf{;}}}
\newcommand{\rconv}{^{\bot}} 
\newcommand{\rdconv}{^{\bot\bot}} 
\newcommand{\gr}[1]{\Gamma_{#1}}

\newcommand{\exteq}{\approx} 
\newcommand{\RUC}{\textsc{(ruc)}\xspace}

\newcommand{\QC}[1]{\ensuremath{\ct{Q}_{#1}}\xspace} 
\newcommand{\QR}[1]{({#1})^q}
\newcommand{\Des}[2]{\ensuremath{\ct{Des}_{#1,#2}}\xspace} 
\newcommand{\QAr}[1]{E^{#1}} 
\newcommand{\QMAr}[1]{M^{#1}} 
\newcommand{\QRFun}{\mathrm{U_q}}
\newcommand{\RQFun}{\mathrm{Q}} 
\newcommand{\QMnd}{\mathrm{T_q}}

\newcommand{\EC}[1]{\ensuremath{\ct{E}_{#1}}\xspace}
\newcommand{\ER}[1]{(#1)^e} 
\newcommand{\EAr}[1]{C^{#1}} 
\newcommand{\ERFun}{\mathrm{U_e}}
\newcommand{\REFun}{\mathrm{E}} 
\newcommand{\EMnd}{\mathrm{T_e}} 
\newcommand{\EQQFun}{\ERFun'}
\newcommand{\EQEFun}{\QRFun'}
\newcommand{\QEQFun}{\REFun'} 
\newcommand{\EQMndQ}{\EMnd'}
\newcommand{\EQMnd}{\mathrm{T_{eq}}} 
\newcommand{\EQRFun}{\mathrm{U_{eq}}}
\newcommand{\REQFun}{\mathrm{EQ}}
\newcommand{\EQR}[1]{(#1)^{eq}} 
\newcommand{\EEDRFun}{\mathrm{\mathsf{Rel}}}

\newcommand{\VRel}[1]{{#1}\text{-}\mathsf{Rel}}

\newcommand{\Qtl}{V}

\newcommand{\Car}[1]{|#1|}
\newcommand{\qord}{\preceq}
\newcommand{\qmul}{\cdot}
\newcommand{\qone}{1} 
\newcommand{\qsup}{\bigvee}
\newcommand{\qinf}{\bigwedge} 

\newcommand{\RR}{R}

\newcommand{\rord}{\preceq}
\newcommand{\rsum}{+}
\newcommand{\rzero}{0}
\newcommand{\rmul}{\cdot}
\newcommand{\rone}{1}
\newcommand{\Mat}[1]{{#1}\text{-}\mathsf{Mat}}

\newcommand{\DRel}[1]{\mathsf{Rel}^{#1}} 
\newcommand{\RelD}[1]{\mathsf{Doc}^{#1}} 

\def\RB#1{\mathchoice
  {\rotatebox[origin=c]{180}{$#1$}}
  {\rotatebox[origin=c]{180}{$#1$}}
  {\rotatebox[origin=c]{180}{$\scriptstyle#1$}}
  {\rotatebox[origin=c]{180}{$\scriptscriptstyle#1$}}}
\def\Ex{\RB{E}\kern-.3ex}
\def\Al{\RB{A}\kern-.6ex}

\newcommand{\Span}[1]{\mathsf{Spn}^{#1}}
\newcommand{\JSpan}[1]{\mathsf{JSpn}^{#1}}
\newcommand{\spn}[5]{\ensuremath{#1 \xleftarrow{#2} #3 \xrightarrow{#4} #5}} 

\newcommand{\Map}[1]{\ct{Map}(#1)}
\newcommand{\RMap}[1]{\mathsf{Map}^{#1}} 
\newcommand{\Ord}[1]{\ct{O}_{#1}} 
\newcommand{\OCI}{\Ct{OCI}\xspace} 
\newcommand{\ORFun}{\mathrm{Map}}  
\newcommand{\ROFun}{\mathrm{\ct{O}}} 

\newcommand{\VecDoc}{\mathsf{Vec}} 
 
\newcommand{\vecx}{{\bf x}}
\newcommand{\vecy}{{\bf y}} 
\newcommand{\vecz}{{\bf z}}
\newcommand{\veczero}{{\bf 0}}

\def\exl{_{\textrm{\scriptsize ex/wlex}}}

\newcommand{\kzd}[1]{\textbf{KZ#1}\xspace} 

\newcommand{\psAlg}[1]{{#1}\text{-}\ct{Alg}^{\mathsf{ps}}\xspace} 

\newcommand{\RTop}{\mathsf{TRel}}

\newcommand{\Bisim}[1]{\mathsf{bisim}^{#1}}
\newcommand{\CoAlg}[1]{\ensuremath{\ct{CoAlg}(#1)}\xspace}

\newif\ifsubmit
\submitfalse

\ifsubmit
\def\FDComm#1{} 
\def\FPComm#1{} 
\def\PRComm#1{} 
\else
\def\lrcom#1{\marginpar{\parbox{6em}
 {\raggedright\scriptsize{#1}}}}
\def\FDComm#1{\smash{\textcolor{orange}\textbullet}\lrcom
 {\textcolor{orange}{#1}}}
\def\FPComm#1{\smash{\textcolor{blue}\textbullet}\lrcom
 {\textcolor{blue}{#1}}}
\fi

\newcommand{\mnd}{\mathbb{T}}

\newcommand{\Proj}[1]{\ct{P}_{#1}}
\newcommand{\RProj}[1]{\mathsf{Proj}^{#1}} 

\newcommand{\xpr}[2][]{\pi_{#2}\ifblank{#1}{}{^{#1}}} 
\newcommand{\fpr}[1][]{\xpr[#1]{1}} 
\newcommand{\spr}[1][]{\xpr[#1]{2}} 
\newcommand{\tpr}[1][]{\xpr[#1]{3}} 
\newcommand{\terar}{!} 
\newcommand{\diagar}{\Delta} 
\newcommand{\TerDoc}{T} 
\newcommand{\TerAr}{\mathbf{!}}
\newcommand{\DiagAr}{\mathbf{\Delta}} 
\newcommand{\reltimes}{\dot{\times}} 
\newcommand{\MCRDtn}{\Ct{MCRD}\xspace}

  \crefname{fig}{Figure}{Figures}
  \crefname{thm}{Theorem}{Theorems}
  \crefname{lem}{Lemma}{Lemmas}
  \crefname{def}{Definition}{Definitions}
  \crefname{prop}{Proposition}{Propositions}
  \crefname{cor}{Corollary}{Corollaries}
  \crefname{ex}{Example}{Examples}
  \crefname{rem}{Remark}{Remarks}
  \crefname{asm}{Assumption}{Assumptions}

\begin{document}

\begin{frontmatter}


\title{The Relational Quotient Completion} 

\author[fd]{Francesco Dagnino}
\ead{francesco.dagnino@unige.it}

\author[fp]{Fabio Pasquali}
\ead{pasquali@dima.unige.it} 

\affiliation[fd]{organization={DIBRIS - University of Genoa},
            country={Italy}}
\affiliation[fp]{organization={DIMA - University of Genoa},
            country={Italy}}

\begin{abstract}
Taking a quotient roughly means changing the notion of equality on a given object, set or type. 
In a quantitative setting, equality naturally generalises to a distance, measuring how much elements are similar instead of just stating their equivalence. 
Hence, quotients can be understood quantitatively as a change of distance.
In this paper, we show how, combining Lawvere's doctrines and the calculus of relations, one can unify quantitative and usual quotients in a common picture. 
More in detail, we introduce relational doctrines as a functorial description of (the core of) the calculus of relations. 
Then, we define quotients and a universal construction adding them to any relational doctrine, generalising the quotient completion of existential elementary doctrine and also recovering many quantitative examples. 
This construction deals with an intensional notion of quotient and breaks extensional equality of morphisms. 
Then, we describe another construction forcing extensionality, showing how it abstracts several notions of separation in metric and topological structures. 
Combining these two constructions, we get  the extensional quotient completion, whose essential image is characterized through the notion of projective cover. 
As an application, we show that, under suitable conditions, relational doctrines of algebras arise as the extensional quotient completion of free algebras. 
Finally, we compare relational doctrines to other categorical structures where one can model the calculus of relations. 
\end{abstract}


\end{frontmatter}


\section{Introduction}
\label{sect:intro}

Quotients are pervasive both in  mathematic and computer science, as they are crucial in carrying out many fundamental arguments. Quotients have been widely studied and several constructions have been refined to allow one to work with quotients even though they are not natively available in the setting in which one is reasoning (such as within a type theory, where usually quotients are not a primitive concept). 
The intuition behind these constructions is that
taking a quotient changes the notion of equality on an object to a given equivalence relation. 
Then, to work with (formal) quotients, one just endows  
each object (set, type, space, $\ldots$)  with an (abstract) equivalence relation and forces the object to ``believe'' that that equivalence relation is the equality.
This idea underlies the construction of setoids  in type theories \cite{BCP:2003,MR1485515}, which are the common solution to work with quotients  in that setting and 
underlies also the exact completion of a category with weak finite limits \cite{CarboniA:freecl,CarboniA:regec}, 
as well as the elementary quotient completion of an elementary doctrine \cite{MaiettiME:quofcm,MaiettiME:eleqc}. 

Following the approach pioneered by Lawvere \cite{Lawvere73}, 
we can regard distances as a quantitative counterpart of equivalence relations: 
they measure how much two elements are similar instead of just saying whether they are equivalent or not. 
Indeed, a distance on a set $X$ is just a function 
\fun{d}{X\times X}{[0,\infty]} taking values in the extended non-negative real numbers,  
which satisfies a form of reflexivity (every point is at distance $0$ from itself), 
symmetry (the distance from $x$ to $y$ is the same as the one from $y$ to $x$) and 
transitivity (which is given by the triangular inequality). 
In this way, a metric space can be viewed as a form of quantitative setoid and 
quotients can be understood quantitatively as a change of distance. 
In fact, this operation is often used when dealing with metric structures, see for instance the construction of monads associated with quantitative equational theories \cite{Adamek22,MardarePP16,MardarePP17}. 

A unified view of quotients covering both usual and quantitative settings is missing. The aim of this paper is to develop a notion of quotient, related concepts and constructions extending known results and  incorporating new quantitative examples.

Many mathematical tools have been adopted to study quotients. 
Among them, Lawvere's doctrines \cite{Lawvere69,Lawvere70} stand out as a simple and powerful framework capable to cope with a large variety of situations (see \cite{JacobsB:catltt, PittsCL, OostenJ:reaait} and references therein). Doctrines provide a functorial description of logical theories, 
abstracting the essential algebraic structure shared by both syntax and semantics of logics. 
In particular, Maietti and Rosolini \cite{MaiettiME:quofcm,MaiettiME:eleqc} identified doctrines modelling the conjunctive fragment of first order logic with equality as the minimal setting where to define equivalence relations and quotients. Then they defined a universal construction, named elementary quotient completion, that freely adds quotients to such doctrines, showing that it subsumes many others, such as setoids and the exact completion of a category with finite limits. 

In order to move this machinery to a quantitative setting, 
one may try to work with doctrines where the usual conjunction is replaced by its linear counterpart. 
In this way, equivalence relations becomes distances as transitivity becomes a triangular inequality. 
However smooth, this transition is less innocent than it appears. 
As shown in \cite{DagninoP22}, to properly deal with a quantitative notion of equality one needs a more sophisticated structure, which however fails to capture  important examples like the category of metric spaces and non-expansive maps. 
The main difficulty in working with Lawvere's doctrines is that doctrines, modelling usual predicate logic, take care of variables. This is problematic in a quantitative setting as the use of variables usually has an impact on the considered distances. 

For these reasons, in this paper we take a different approach: we work with doctrines abstracting the \emph{calculus of relations} \cite{Givant1,Peirce,Tarski41} which is a variable-free alternative to first order logic. Here one takes as primitive concept (binary) relations instead of (unary) predicates, together with some basic operations, such as relational identities, composition and the converse of a relation. 
Even though in general it is less expressive than first order logic,\footnote{The calculus of relations is equivalent to first order logic with three variables \cite{Givant06}.}
it is still quite expressive, for instance, one can axiomatise set theory in it \cite{tarski1988formalization}.
Moreover, being variable-free, it scales well to quantitative settings, as witnessed by the fruitful adoption of relational techniques to develop quantitative methods \cite{DalLagoG22,Gavazzo18,GavazzoF23}. 

Then, in this paper, we introduce \emph{relational doctrines}, 
as a functorial description of a core fragment of the calculus of relations. 
Relying on this structure, we define a notion of quotient capable to deal with also quantitative settings. We present a universal construction to add such quotients to any relational doctrine. The construction extends the one in \cite{MaiettiME:quofcm,MaiettiME:eleqc} and can also capture quantitative instances such as the category of metric spaces and non-expansive maps. 

Furthermore, related to quotients, we study the notion of extensional equality. 
Roughly, two functions or morphisms are extensionally equal if their outputs coincide on equal inputs. Even if quotients and extensionality are independent concepts, several known constructions that add quotients often force extensionality (see e.g., Bishop's sets, setoids over a type theory or the exact completion). Therefore the study of extensionality is essential to cover these well-known examples. We show that our quotient completion, changing the notion of equality on objects without affecting plain equality on arrows in the base category, may break this property. 
Thus, we define  another universal construction that forces extensionality and use it to obtain an extensional version of  our quotient completion. 
We show also how this logical principle captures many notions of separation in metric and topological structures.

These results are developed using the language of 2-categories \cite{Lack2010}. 
To this end, we organise relational doctrines in a suitable 2-category where morphisms  abstract the usual notion of relation lifting \cite{KurzV16}. 
The universality of our constructions is then expressed in terms of (lax) 2-adjunctions \cite{BettiP88}, thus describing their action not only on relational doctrines, but on their morphisms as well. 
We also prove that these constructions are 2-monadic \cite{BlackwellKP89}, showing that quotients and extensionality are algebraic concepts, that is, they can be described by (pseudo)algebra structures for certain 2-monads on a relational doctrine. 
Furthremore, we show that these algebra structures are essentially unique, proving that the associated 2-monads are (lax) idempotent \cite{Kock95,KellyL97}. 
This makes precise the fact that having quotients or being extensional is a property of a relational doctrine rather than a structure on it. 

Since many categorical concepts can be defined internally to any 2-category, we get them for free also for relational doctrines. 
For instance, following \cite{Street72}, we can consider monads on relational doctrines and the associated doctrines of algebras  where relations between two algebras are given by congruences, that is, relations closed under the operations of the algebras they relate. 
As an application of our construction, we  show that, under suitable hypotheses, these doctrines of algebras can be obtained as the extensional quotient completion of their restriction to free algebras, extending a similar result proved for the exact completion \cite{Vitale94}. 
To achieve this result, we rely on a characterization of relational doctrines obtained through the extensional quotient completion as those doctrines where every object can be presented as a quotient of a projective one, again generalizing  a similar result proved for the exact completion \cite{CarboniA:regec}. 

The paper is organised as follows. 
In \cref{sect:rel-doc} we introduce relational doctrines with their basic properties, presenting several examples. 
In \cref{sect:quotients} we define quotients and the intensional quotient completion, proving it is universal and 2-monadic, generating a lax idempotent 2-monad. 
In \cref{sect:ext-sep} we discuss extensionality, its connection with separation and the universal construction forcing it, proved again to be 2-monadic and to generate an idempotent 2-monad. 
In \cref{sect:eqc}, we combine these result to define the extensional quotient completion,  showing it is universal, 2-monadic and generating a lax idempotent 2-monad. 
\cref{sect:proj} provides the  characterization of doctrines obtained through the extensional quotient completion and apply it to doctrines of algebras. 
In \cref{sect:eed} we compare our approach with two important classes of examples: ordered categories with involution \cite{Lambek99}, which are a generalisation of both allegories and cartesian bicategories, and existential elementary doctrines \cite{MaiettiME:quofcm,MaiettiME:eleqc}, characterizing relational doctrines corresponding to them. 
Finally, \cref{sect:conclu} summarises our contributions and discusses directions for future work. 

\subsubsection*{Source of the material}
This paper is an extended version of \cite{DagninoP23}, which was presented at FSCD 2023. 
With respect to it,  here we include proofs of all our results, we discuss 2-monadicity of the presented constructions, we provide the characterization based on projective objects and the application to doctrines of algebras and we improve the comparison with existential elementary doctrines, highlighting the role of the modular law.


\section{Relational Doctrines: Definition and First Properties}
\label{sect:rel-doc}

Doctrines are a simple and powerful framework introduced by Lawvere \cite{Lawvere69,Lawvere70} to study several kinds of logics using categorical tools. 
A \emph{doctrine} $\PDoc$ on $\CC$ is a contravariant functor \fun{\PDoc}{\CC\op}{\Pos}, 
where  \Pos denotes  the category of posets and monotone functions. 
The category \CC is named the \emph{base} of the doctrine and, 
for $X$  in \CC, the poset $\PDoc(X)$ is called \emph{fibre over $X$}. For \fun{f}{X}{Y} an arrow in \CC, the monotone function \fun{\PDoc\reidx{f}}{\PDoc (Y)}{\PDoc (X)} is called \emph{reindexing along $f$}.
Roughly, the base category collects the objects one is interested in with their transformations, 
a fibre $\PDoc(X)$ collects predicates over the object $X$ ordered by logical entailment and 
reindexing allows to transport predicates between objects according to their transformations. 
An archetypal example of a doctrine  is the contravariant powerset functor \fun{\PP}{\Set\op}{\Pos}, where predicates are represented by subsets ordered by set inclusion. 

Doctrines capture the essence of predicate logic. 
In this section, we will introduce \emph{relational doctrines} as a functorial description of the essential structure of relational logics. 
To this end, since binary relations can be seen as predicates over a pair of objects, we will need to index posets over pairs of objects, that is, to consider functors \fun{\RDoc}{\bop\CC}{\Pos}, where each fibre $\RDoc(X,Y)$ collects relations from $X$ to $Y$. 
Here the reference example are set-theoretic relations: 
they can be organised into a functor \fun{\Rel}{\bop\Set}{\Pos} 
where $\Rel(X,Y)=\PP(X\times Y)$ and sending $f,g$ to the inverse image $(f\times g)^{-1}$.

We endow these functors with 
a structure modelling a core fragment of the calculus of relations given by 
relational identities, composition and converse \cite{Givant1,Peirce,Tarski41}. 
For set-theoretic relations, 
the identity relation on a set $X$ is the diagonal $\rid_X = \{\ple{x,x'}\in X \times X\mid x = x'\}$, 
the composition of $\relr\in \Rel(X, Y)$ with $\rels\in \Rel(Y, Z)$ is the set $\relr\rcomp\rels = \{\ple{x,z} \in X\times Z \mid \ple{x,y}\in \relr,\, \ple{y,z}\in \rels \text{ for some }y\in Y\}$, and 
the converse of $\relr\in\Rel(X,Y)$ is the set $\relr\rconv = \{\ple{y,x}\in Y\times X \mid \ple{x,y}\in \relr\}$. 
These operations interact with reindexing, \ie inverse images, by the following inclusions: 
$\rid_X \subseteq (f\times f)^{-1}(\rid_Y)$ and 
$(f\times g)^{-1}(\relr)\rcomp (g\times h)^{-1}(\rels)\subseteq (f\times h)^{-1}(\relr\rcomp\rels)$ and also
$((f\times g)^{-1}(\relr))\rconv \subseteq (g\times f)^{-1}(\relr\rconv)$. 
The first two inclusions are not equalities in general: 
the former is an equality when $f$ is injective, while the latter is an equality when $g$ is surjective. 
These observations lead us to the following definition.

\begin{definition}\label[def]{def:rel-doc}
A \emph{relational doctrine} consists of the following data: 
\begin{itemize}
\item a base category \CC, 
\item a functor \fun{\RDoc}{\bop\CC}{\Pos}, 
\item an element $\rid_X \in \RDoc(X,X)$, for every object $X$ in \CC, such that 
$\rid_X \order \RDoc\reidx{f,f}(\rid_Y)$, 
for every arrow \fun{f}{X}{Y} in \CC, 
\item a monotone function \fun{\blank\rcomp\blank}{\RDoc(X,Y)\times\RDoc(Y,Z)}{\RDoc(X,Z)}, for every triple of objects $X,Y,Z$ in \CC, such that 
$\RDoc\reidx{f,g}(\relr)\rcomp\RDoc\reidx{g,h}(\rels) \order \RDoc\reidx{f,h}(\relr\rcomp\rels)$, 
for all $\relr\in\RDoc(A,B)$, $\rels\in\RDoc(B,C)$ and 
\fun{f}{X}{A}, \fun{g}{Y}{B} and \fun{h}{Z}{C} arrows in \CC, 
\item a monotone function \fun{(\blank)\rconv}{\RDoc(X,Y)}{\RDoc(Y,X)}, for every pair of objects $X,Y$ in \CC,  such that 
$(\RDoc\reidx{f,g}(\relr))\rconv \order \RDoc\reidx{g,f}(\relr\rconv)$, 
for all $\relr\in\RDoc(A,B)$ and \fun{f}{X}{A} and \fun{g}{Y}{B}, 
\end{itemize}
satisfying the following equations for all 
$\relr\in\RDoc(X,Y)$, $\rels\in\RDoc(Y,Z)$ and $\relt\in\RDoc(Z,W)$
\begin{align*} 
\relr\rcomp(\rels\rcomp\relt) &= (\relr\rcomp\rels)\rcomp\relt 
& 
\rid_X\rcomp\relr &= \relr 
&
\relr\rcomp\rid_Y &= \relr 
\\
(\relr\rcomp\rels)\rconv &= \rels\rconv \rcomp \relr\rconv 
&
\rid_X\rconv &= \rid_X
&
\relr\rdconv &= \relr 
\end{align*} 
\end{definition}
The element $\rid_X$ is the \emph{identity} or \emph{diagonal} relation on $X$, 
$\relr\rcomp\rels$  is the \emph{relational composition} of $\relr$ followed by $\rels$, and 
$\relr\rconv$ is the \emph{converse} of the relation $\relr$.  
Note that all relational operations are lax natural transformations, but 
the operation of taking the converse, being  an involution, is actually strictly natural. 
Indeed, we have 
\[
\RDoc\reidx{g,f}(\relr\rconv) 
  = ((\RDoc\reidx{g,f}(\relr\rconv))\rconv)\rconv 
  \order (\RDoc\reidx{f,g}((\relr\rconv)\rconv))\rconv 
  = (\RDoc\reidx{f,g}(\relr))\rconv 
\]
Also, each one of the two axioms stating that $\rid$ is the neutral element of the composition, together with the other axioms, implies the other.
For instance, assuming the left identity, we derive 
$\relr\rcomp\rid_Y = (\rid_Y\rconv\rcomp\relr\rconv)\rconv = (\rid_Y\rcomp\relr\rconv)\rconv = \relr\rdconv =\relr$.

\begin{remark} \label[rem]{rem:internal-cat}
There are many alternative ways of defining relational doctrines.
A possibility is to see them as certain internal dagger categories in a category of doctrines. 
Doctrines are the objects of a 2-category \Dtn where a 1-arrow  from $\fun{\PDoc}{\CC\op}{\Pos}$ to $\fun{\QDoc}{\D\op}{\Pos}$ is
a pair $F=\ple{\fn{F},\lift{F}}$ where 
$\fun{\fn{F}}{\CC}{\D}$ is a functor and  
$\lnt{\lift{F}}{\PDoc}{\QDoc \fn{F}\op}$ is a lax natural transformation.
that is, for every object $X$ in \CC, a monotone function \fun{\lift{F}_X}{\PDoc(X)}{\QDoc(\fn{F} X)}  such that
$\lift{F}_X\circ\PDoc\reidx{f} \order \QDoc\reidx{\fn{F}f}\circ\lift{F}_Y$ holds, for every arrow \fun{f}{X}{Y} in \CC. 
Given 1-arrows \oneAr{F,G}{\ple{\CC,\PDoc}}{\ple{\D,\QDoc}}, 
a 2-arrow \twoAr{\theta}{F}{G} is a natural transformation \nt{\theta}{F}{G} such that 
$\lift{F}_X\order_{\fn{F}X}\QDoc\reidx{\theta_X}\circ \lift{G}_X$ for every object $X$ in \CC.
Compositions and identities are defined in the expected way. 
Then, the data defining a relational doctrine \fun{\RDoc}{\bop\CC}{\Pos} can be organized into the following diagram in $\Dtn$, describing an \emph{internal dagger category}: 
\[\xymatrix@C=15ex{
\RDoc^2 \ar[r]^-{\ple{\ple{\pi_1,\pi_3},\blank\rcomp\blank}}  & 
\RDoc \ar@(ul,ur)^-{\ple{\ple{\pi_2,\pi_1},(\blank)\rconv}}
      \ar@(ur,ul)[r]^-{\ple{\pi_2,\zeta}}
      \ar@(dr,dl)[r]_-{\ple{\pi_1,\zeta}}  & 
\One_\CC \ar[l]_-{\ple{\Delta,\rid}} 
}\]
Here, \fun{\One_\CC}{\CC\op}{\Pos} is the trivial doctrine, mapping every object of \CC to the singleton poset, 
$\zeta$ is the natural transformation whose components are the unique maps into the singleton poset,  and 
$\RDoc^2$ is the pullback of \ple{\pi_1,\zeta} against \ple{\pi_2,\zeta}, that is, 
the functor \fun{\RDoc^2}{(\CC\times\CC\times\CC)\op}{\Pos} defined by 
$\RDoc^2(X,Y,Z) = \RDoc(X,Y)\times\RDoc(Y,Z)$ and $\RDoc^2\reidx{f,g,h} = \RDoc\reidx{f,g}\times\RDoc\reidx{g,h}$. 

Alternatively, they can be regarded as faithful framed bicategories \cite{Shulman08,Lambert22} (a.k.a. equipements)
which are double categories with the additional structure of a fibration, extended with an appropriate involution. 
In \cref{def:rel-doc}, we give a more explicit and elementary description of relational doctrines,
to keep things simple and to stay closer to the usual language of doctrines.
Extending all our results to the more general and proof-relevant setting of framed bicategories is an interesting problem we leave for future work.
\end{remark}

The following list of examples is meant to give a broad range of situations that can be described by relational doctrines. Order categories and existential elementary doctrines provide two large classes of examples which are intentionally omitted as, due to their relevance, they will be discussed separately in \cref{sect:eed}.

\begin{example} \label[ex]{ex:rel-doc} 
\begin{enumerate}
\item\label{ex:rel-doc:vrel}
Let $\Qtl = \ple{\Car\Qtl,\qord,\qmul,\qone}$ be a commutative quantale. 
A \emph{$\Qtl$-relation}  \cite{HofmannST14} between sets $X$ and $Y$ is a function \fun{\alpha}{X\times Y}{\Car\Qtl}, where 
$\alpha(x,y)\in\Car\Qtl$ intuitively  measures how much elements $x$ and $y$ are related by $\alpha$. 
Then, we consider 
the functor \fun{\VRel\Qtl}{\bop\Set}{\Pos} where 
$\VRel\Qtl(X,Y) = \Car\Qtl^{X\times Y}$ is the set of $\Qtl$-relations from $X$ to $Y$  with the pointwise order, 
$\VRel\Qtl\reidx{f,g}$ is precomposition with $f\times g$ and 
The identity relation, composition and converse are defined as follows: 
\[
\rid_X(x,x') = \begin{cases}
\qone & x = x' \\
\bot  & x \ne x' 
\end{cases}
\qquad 
(\relr\rcomp\rels)(x,z) = \qsup_{y\in Y} (\relr(x,y)\qmul\rels(y,z))
\qquad 
\relr\rconv(y,x) = \relr(x,y) 
\]
where $\relr\in\VRel\Qtl(X,Y)$ and $\rels\in\VRel\Qtl(Y,Z)$. 
Special cases of this doctrine are
\fun{\Rel}{\bop\Set}{\Pos} , when the quantale is $\Bool = \ple{\{0,1\},\leq,\land,1}$, and 
metric relations, when one considers the Lawvere's quantale $\RPos = \ple{[0,\infty],\ge,+,0}$ as in \cite{Lawvere73}.
\item\label{ex:rel-doc:mat}
Let $\RR = \ple{\Car\RR,\rord,\rsum,\rmul,\rzero,\rone}$ be a continuous semiring \cite{LairdMMP13,Ong17}, that is, an ordered semiring where \ple{\Car\RR,\rord} is a directed complete partial order (DCPO), $\rzero$ is the least element and $\rsum$ and $\rmul$ are Scott-continuous functions. 
In this setting, we can compute sums of arbitrary arity. 
For a function \fun{f}{X}{\Car\RR},  we can define its sum
$\sum f$, also denoted by 
$\sum_{x \in X} f(x)$, 
as
\[\sum f = \qsup_{I \in \fpw{X}} \sum_{i \in I} f(i) \]
where $\fpw{X}$ is the finite powerset of $X$. 
Consider  
\fun{\Mat\RR}{\bop\Set}{\Pos} where 
$\Mat\RR(X,Y)$ is the set of functions \fun{}{X\times Y}{\Car\RR} with the pointwise order, 
$\Mat\R\reidx{f,g}$ is precomposition with $f\times g$.
Elements in $\Mat\RR(X,Y)$ are a matrices with entries in $\Car\RR$ and indices for rows and columns taken from $X$ and $Y$.
The identity relation, composition and converse are 
given by the Kronecker's delta (\ie the identity matrix), matrix multiplication and transpose, 
defined as follows: 
\[
\rid_X(x,x') = \begin{cases}
\rone  & x = x'  \\ 
\rzero & x \ne x' 
\end{cases}
\qquad 
(\relr\rcomp\rels)(x,z) = \sum_{y \in Y} (\relr(x,y) \rmul \rels(y,z)) 
\qquad 
\relr\rconv(y,x) = \relr(x,y)
\]
where $\relr\in\Mat\RR(X,Y)$ and $\rels\in\Mat\RR(Y,Z)$. 
This relational doctrine generalises $\Qtl$-relations since 
any quantale is a continuous semiring (binary/arbitrary joins give addition/infinite sum). 
The paradigmatic example of a continuous semiring which is not a quantale is that of extended non-negative real numbers $[0,\infty]$, with the usual order, addition and multiplication. 
Restricting the base to finite sets
all sums become finite, hence  the definition works also for a plain ordered semiring.
\item\label{ex:rel-doc:span}
Let \CC be a category with weak pullbacks. 
Denote by $\Span\CC(X,Y)$ the poset reflection of the preorder whose objects are spans in \CC between $X$ and $Y$ and 
$\spn{X}{p_1}{A}{p_2}{Y} \order \spn{X}{q_1}{B}{q_2}{Y}$ iff there is an arrow \fun{f}{A}{B} such that 
$p_1 = q_1 \circ f $ and $p_2 = q_2 \circ f$. 
Given a span $\relr = \spn{X}{p_1}{A}{p_2}{Y}$ and arrows \fun{f}{X'}{X} and \fun{g}{Y'}{Y} in \CC, define 
$\Span\CC\reidx{f,g}(\relr) \in \Span\CC(X',Y')$ by one of the following equivalent diagrams: 
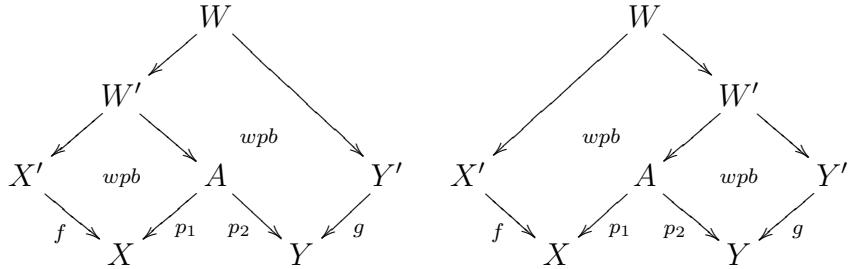
\[
\vcenter{\xymatrix@C=3ex@R=3ex{
&& W \ar[ld] \ar[rrdd] \ar@{}[rddd]|{wpb} && \\
& W'\ar[ld] \ar[rd] \ar@{}[dd]|{wpb} &&& \\  
X' \ar[rd]_-{f} && A \ar[ld]^-{p_1} \ar[rd]_-{p_2} && Y' \ar[ld]^-{g} \\ 
& X && Y 
}} 
\quad 
\vcenter{\xymatrix@C=3ex@R=3ex{
&& W \ar[rd] \ar[lldd] \ar@{}[lddd]|{wpb} && \\
&&& W'\ar[rd] \ar[ld] \ar@{}[dd]|{wpb} & \\  
X' \ar[rd]_-{f} && A \ar[ld]^-{p_1} \ar[rd]_-{p_2} && Y' \ar[ld]^-{g} \\ 
& X && Y 
}} 
\]
The functor \fun{\Span\CC}{\bop\CC}{\Pos} is a relational doctrine where, for 
$\relr = \spn{X}{p_1}{A}{p_2}{Y}$ and 
$\rels = \spn{X}{q_1}{B}{q_2}{Y}$ it is 
\[
\rid_X = \vcenter{\xymatrix@R=3ex@C=3ex{
& X \ar[ld]_-{\id_X} \ar[rd]^-{\id_X} & \\ 
X && X
}} \qquad 
\relr \rcomp \rels = \vcenter{\xymatrix@R=3ex@C=3ex{
&& W \ar[ld] \ar[rd] \ar@{}[dd]|{wpb} && \\ 
& A \ar[ld]^-{p_1}\ar[rd]_-{p_2} && B \ar[ld]^-{q_1} \ar[rd]_-{q_2} & \\ 
X && Y && Z 
}}
\qquad 
\relr\rconv = \vcenter{\xymatrix@C=3ex@R=3ex{
& A\ar[ld]_{p_2} \ar[rd]^{p_1} & \\ 
Y && X 
}}
\]
One can do a similar construction for jointly monic spans, provided that the category \CC has strong pullbacks and a proper factorisation system, as happens for instance in a locally regular category. 
In particular, the relational doctrine of jointly monic spans over $\ct{Set}$ is the relational doctrine $\Rel$ of set-based relations already mentioned in \cref{ex:rel-doc:vrel}.
\item\label{ex:rel-doc:vec}
Let $\ct{Vec}$ be the category of vector spaces over real numbers and linear maps. 
Write $\Car X$ for the underlying set of the vector space $X$ and $X\times Y$ for the cartesian product of vector spaces. 
A semi-norm on a vector space $X$ is a function \fun{\alpha}{\Car{X}}{[0,\infty]} which is 
lax monoidal, i.e., $0\geq\alpha(\veczero)$ and $\alpha(\vecx)+\alpha(\vecy) \geq \alpha(\vecx+\vecy)$, and 
homogeneous, i.e., $\alpha(a\cdot \vecx) = |a|\cdot \alpha(\vecx)$. 
The functor 
\fun{\VecDoc}{\bop{\ct{Vec}}}{\Pos} sends $X,Y$ to the suborder of $\VRel\RPos(\Car{X},\Car{Y})$ on semi-norms on $X\times Y$ and acts on linear maps by precomposition. 
The functor $\VecDoc$ is a relational doctrine where 
\[
\rid_X(\vecx,\vecx') = \begin{cases}
0 & \vecx = \vecx' \\
\infty & \vecx \ne \vecx' 
\end{cases}
\quad 
(\relr\rcomp\rels)(\vecx,\vecz) = \inf_{\vecy\in\Car Y} (\relr(\vecx,\vecy) + \rels(\vecy,\vecz)) 
\quad 
\relr\rconv(\vecy,\vecx) = \relr(\vecx,\vecy)
\]
\item\label{ex:rel-doc:cb}
Let \fun{\RDoc}{\bop\CC}{\Pos} be a relational doctrine and \fun{F}{\D}{\CC} a functor. 
The change-of-base of $\RDoc$ along $F$ is the relational doctrine 
\fun{F^\star\RDoc}{\bop\D}{\Pos} obtained precomposing $\RDoc$ with $\bop{F}$. 
The change of base allows to use relations of $\RDoc$ to reason about the category \D. 
For example the forgetful functor \fun{U}{\CC}{\Set} of a concrete category $\CC$ allows the use of set-theoretic relations to reason about \CC, considering the doctrine $U^\star\Rel$ which maps a pair of objects $X,Y$ in \CC to $\pw{UX\times UY}$. 
\end{enumerate}
\end{example}

Let  $\RDoc$ be  a relational doctrine on $\CC$. 
We identify some special classes of relations in $\RDoc(X,Y)$, generalising usual set-theoretic notions. 
A relation  $\relr\in\RDoc(X,Y)$ is said to be 
\begin{itemize} 
\item \emph{functional} if $\relr\rconv \rcomp \relr \order \rid_Y$, 
\item \emph{total} if $\rid_X \order \relr\rcomp\relr\rconv$, 
\item \emph{injective} if $\relr\rcomp\relr\rconv \order \rid_X$, and 
\item \emph{surjective} if $\rid_Y \order \relr\rconv \rcomp \relr$. 
\end{itemize} 
Finally, we say that $\relr$ is \emph{bijective} if it is both injective and surjective. 

The next proposition shows that 
 functional and total relations are discretely ordered.

\begin{proposition}\label[prop]{prop:fun-ord}
Let $\RDoc$ be a relational doctrine over \CC. 
For functional and total relations $\relr,\rels\in\RDoc(X,Y)$
if $\relr\order\rels$, then $\relr = \rels$. 
\end{proposition}
\begin{proof} 
If $\relr\order\rels$, then 
$\rels = \rid_X\rcomp\rels 
       \order \relr \rcomp \relr\rconv \rcomp \rels 
       \order \relr \rcomp \rels\rconv \rcomp \rels 
       \order \relr \rcomp \rid_Y 
       = \relr$. 
\end{proof}

Every arrow \fun{f}{X}{Y} defines a relation $\gr{f}=\RDoc\reidx{f,\id_Y}(\rid_Y)\in\RDoc(X,Y)$, called the \emph{graph} of $f$ whose converse is given by 
$\gr{f}\rconv = \RDoc\reidx{f,\id_Y}(\rid_Y)\rconv=\RDoc\reidx{\id_Y,f}(\rid_Y\rconv) =  \RDoc\reidx{\id_Y,f}(\rid_Y)$. 
Then, it is easy to see the next proposition holds. 

\begin{proposition}\label[prop]{prop:graph-fun} 
Let $\RDoc$ be a relational doctrine over \CC and 
\fun{f}{X}{Y} be an arrow in \CC. 
Then, $\gr{f}$ is functional and total. 
\end{proposition}
\begin{proof} 
Functionality follows from 
$
\gr{f}\rconv \rcomp \gr{f} 
  = \RDoc\reidx{\id_Y,f}(\rid_Y) \rcomp \RDoc\reidx{f,\id_Y}(\rid_Y) 
  \order \RDoc\reidx{\id_Y,\id_Y}(\rid_Y\rcomp \rid_Y) 
  = \rid_Y\rcomp\rid_Y = \rid_Y 
$. 
From $\rid_X\order\RDoc\reidx{f,f}(\rid_Y)$ and since $(f,f)=(f,\id_Y)\circ (\id_X,f)=(\id_X,f)\circ (f,\id_Y)$ in $\CC\times\CC$ 
we derive 
\begin{align*} 
\rid_X 
  &= \rid_X\rcomp\rid_X
   \order \RDoc\reidx{f,f}(\rid_Y) \rcomp \RDoc\reidx{f,f}(\rid_Y) 
   = \RDoc\reidx{\id_X,f}(\RDoc\reidx{f,\id_Y}(\rid_Y)) \rcomp \RDoc\reidx{f,\id_X}(\RDoc\reidx{\id_Y,f}(\rid_Y)) 
\\ 
  &\order \RDoc\reidx{\id_X,\id_X}(\RDoc\reidx{f,\id_Y}(\rid_Y)\rcomp \RDoc\reidx{\id_Y,f}(\rid_Y)) 
   = \gr{f}\rcomp \gr{f}\rconv 
\end{align*} 
that proves totality of $\gr{f}$.
\end{proof}

Overloading the terminology, 
we will saty that an arrow \fun{f}{X}{Y} is 
injective, surjective or bijective if so is its graph $\gr{f}$. 
Note that in general epimorphism and surjective arrows, as well as monomorphisms and injective arrows, are incomparable classes of arrows. 
Nevertheless, it is easy to see that split epimorphisms are always surjective, as well as split monomorphisms are always injective. 
Hence, isomorphisms are necessarily bijective. 
However, the converse of these implications does  not hold in general, thus we introduce the following definition. 

\begin{definition}\label[def]{def:balance}
Let \fun{\RDoc}{\bop\CC}{\Pos} be a relational doctrine. 
We say that $\RDoc$ is \emph{balanced} if every bijective arow in \CC is an isomorphism. 
\end{definition}

Relational composition allows us to express reindexing in relational terms and to show it has left adjoints, as proved below.  
Recall that in \Pos a left adjoint of a monotone function $\fun{g}{K}{H}$ is a monotone  function $\fun{f}{H}{K}$ such that for every $x$ in $K$ and $y$ in $H$, both $y\le gf(y)$ and $fg(x)\le x$ hold, or, equivalently, $y\le g(x)$ if and only if $f(y)\le x$. 

\begin{proposition}\label[cor]{prop:left-adj}
Let $\RDoc$ be a relational doctrine over \CC. 
For \fun{f}{A}{X} and \fun{g}{B}{Y}  in \CC the reindexing
\fun{\RDoc\reidx{f,g}}{\RDoc(X,Y)}{\RDoc(A,B)} has a left adjoint 
\fun{\Ex^\RDoc\reidx{f,g}}{\RDoc(A,B)}{\RDoc(X,Y)} and for $\relr\in\RDoc(X,Y)$ and $\rels\in\RDoc(A,B)$ we have 
\[
\RDoc\reidx{f,g}(\relr) = \gr{f} \rcomp \relr \rcomp \gr{g}\rconv 
\qquad 
\Ex^\RDoc\reidx{f,g}(\rels) = \gr{f}\rconv \rcomp \rels \rcomp \gr{g} 
\]
\end{proposition}
\begin{proof} 
First of all we prove that
$\RDoc\reidx{f,g}(\relr) = \gr{f} \rcomp \relr \rcomp \gr{g}\rconv$. 
We have that 
\begin{align*}
\RDoc\reidx{f,g}(\relr) 
  &
   \order \gr{f}\rcomp \gr{f}\rconv \rcomp \RDoc\reidx{f,g}(\relr) \rcomp \gr{g} \rcomp \gr{g}\rconv 
  = \gr{f}\rcomp \RDoc\reidx{\id_X,f}(\rid_X)\rcomp \RDoc\reidx{f,g}(\relr) \rcomp \RDoc\reidx{g,\id_Y}(\rid_Y)\rcomp \gr{g}\rconv 
  \\ 
  &\order \gr{f} \rcomp \RDoc\reidx{\id_X,\id_Y}(\rid_X\rcomp\relr\rcomp\rid_Y)\rcomp \gr{g}\rconv 
   = \gr{f}\rcomp \relr \rcomp \gr{g}\rconv 
  \\
\gr{f}\rcomp\relr\rcomp\gr{g}\rconv 
  &= \RDoc\reidx{f,\id_X}(\rid_X) \rcomp \RDoc\reidx{\id_X,\id_Y}(\relr) \rcomp \RDoc\reidx{\id_Y,g}(\rid_Y) 
   \order \RDoc\reidx{f,g}(\rid_X\rcomp\relr\rcomp\rid_Y) 
   = \RDoc\reidx{f,g}(\relr) 
\end{align*}
Let us set $\Ex^\RDoc\reidx{f,g}(\rels) = \gr{f}\rconv\rcomp\rels\rcomp\gr{g}$. 
It is immediate to see that $\Ex^\RDoc\reidx{f,g}$ is monotone. 
To check that it is the left adjoint of $\RDoc\reidx{f,g}$ it suffices to show that 
$\gr{f}\rconv\rcomp \gr{f}\rcomp \relr \rcomp \gr{g}\rconv\rcomp\gr{g}\order \relr$ and 
$\rels \order \gr{f}\rcomp\gr{f}\rconv\rcomp\rels\rcomp\gr{g}\rcomp\gr{g}\rconv$ hold. 
This follows by \cref{prop:graph-fun} and monotonicity of composition. 
\end{proof}

An easy consequence of this proposition is that the graph of an arrow can be defined by the left adjoints as well. 
Indeed, we have 
$\gr{f} = \Ex^\RDoc\reidx{\id_X,f}(\rid_X)$ and $\gr{f}\rconv = \Ex^\RDoc\reidx{f,\id_X}(\rid_X)$. 
Moreover, one can easily prove that  the graph of arrows respects composition and identities, that is, 
$\gr{g\circ f} = \gr{f}\rcomp\gr{g}$ and 
$\gr{\id_X} = \rid_X$.

We conclude the section describing the 2-categories of relational doctrines we will consider in the rest of the paper. 
The 2-category \RDtnl has relational doctrins as objects, 
while a 1-arrow \oneAr{F}{\RDoc}{\SDoc} between \fun{\RDoc}{\bop\CC}{\Pos} and \fun{\SDoc}{\bop\D}{\Pos} 
is a pair \ple{\fn{F},\lift{F}} consisting of 
a functor \fun{\fn{F}}{\CC}{\D} and a natural transformation 
\nt{\lift{F}}{\RDoc}{\SDoc\circ \bop{\fn{F}}} \emph{laxly} preserving relational identities, composition and converse, that is, satisfying 
$\rid_{\fn{F}X} \order \lift{F}_{X,X}(\rid_X)$ and 
$\lift{F}_{X,Y}(\relr)\rcomp \lift{F}_{Y,Z}(\rels) \order \lift{F}_{X,Z}(\relr\rcomp\rels)$ and 
$(\lift{F}_{X,Y}(\relr))\rconv \order \lift{F}_{Y,X}(\relr\rconv)$, 
for $\relr\in\RDoc(X,Y)$ and $\rels\in\RDoc(Y,Z)$. 
Finally, given 1-arrows \oneAr{F,G}{\RDoc}{\SDoc}, 
a 2-arrow \twoAr{\theta}{F}{G} is a natural transformation \nt{\theta}{\fn{F}}{\fn{G}} such that 
$\lift{F}_{X,Y} \order \SDoc\reidx{\theta_X,\theta_Y}\circ \lift{G}_{X,Y}$, for all objects $X,Y$ in the base of $\RDoc$. 
By \cref{prop:graph-fun,prop:left-adj} the condition of a 2-arrow \twoAr{\theta}{F}{G} is equivalent to both 
$\lift{F}_{X,Y}(\relr) \order \gr{\theta_X}\rcomp \lift{G}_{X,Y}(\relr) \rcomp \gr{\theta_Y}\rconv$ and 
$\lift{F}_{X,Y}(\relr) \rcomp \gr{\theta_Y} \order \gr{\theta_X}\rcomp \lift{G}_{X,Y}(\relr)$, for $\relr\in\RDoc(X,Y)$. 
 
It is easy to see that 
1-arrows actually strictly preserve the converse, since it is an involution, and 
laxly preserve graphs of arrows, that is, 
$\gr{\fn{F}f} \order \lift{F}_{X,Y}(\gr{f})$ and $\gr{\fn{F}f}\rconv \order \lift{F}_{Y,X}(\gr{f}\rconv)$, for every arrow \fun{f}{X}{Y} in the base of $\RDoc$. 
A 1-arrow is called \emph{strict} if it strictly preserves relational identities and composition. 
In this case, it also strictly preserves graphs of arrows. 
We denote by \RDtn the the 2-full 2-subcategory of \RDtnl where 1-arrows are strict. 
We will mainly work with the 2-category \RDtn, which in general is much better behaved than \RDtnl. 
However, we consider also the latter one since lax 1-arrows abstract a widely used notion in relational methods, as the next example shows.

\begin{example}[Relation lifting] \label[ex]{ex:rel-lift} 
A key notion used in relational methods is that of \emph{relation lifting} or \emph{lax extension} or \emph{relator} \cite{Barr70,KurzV16,Thijs96}. 
It can be used to formulate bisimulation for coalgebras or other notions of program equivalence. 
A (conversive) relation lifting of a functor \fun{F}{\Set}{\Set} is a family of monotonic maps \fun{\lift{F}_{X,Y}}{\Rel(X,Y)}{\Rel(FX,FY)}, indexed by sets $X$ and $Y$, such that 
$\lift{F}_{X;Y}(\relr)\rconv \subseteq \lift{F}_{Y,X}(\relr\rconv)$, $\lift{F}_{X,Y}(\relr)\rcomp\lift{F}_{Y,Z}(\rels) \subseteq \lift{F}_{X,Z}(\relr\rcomp\rels)$ and $Ff \subseteq \lift{F}_{X,Y}(f)$, where 
$\relr$ and $\rels$ are relations and \fun{f}{X}{Y} is a function. 
Note that in the last condition we are using the function to denote its graph, which is perfectly fine since set-theoretic functions coincide with their graph. 
It is easy to see that these requirements ensure that \oneAr{\ple{F,\lift{F}}}{\Rel}{\Rel} is a 1-arrow in \RDtnl. Conversely any 1-arrow \oneAr{G}{\Rel}{\Rel} is such that $\lift{G}$ is a relation lifting of $\fn{G}$, showing that 1-arrows between $\Rel$ and $\Rel$ are exactly the relation liftings. 
Hence, 1-arrows of the form \fun{F}{\RDoc}{\RDoc} in \RDtnl can be regarded as a generalisation of relation lifting to an arbitrary relational doctrine $\RDoc$. 

Finally, relying on the 2-categorical structure of \RDtnl, we get for free a notion of monad on a relational doctrine. 
A monad consists of a 1-arrow \oneAr{T}{\RDoc}{\RDoc} together with 2-arrows \twoAr{\eta}{\Id_\RDoc}{T} and \twoAr{\mu}{T\circ T}{T} satisfying usual diagrams: 
\[\vcenter{\xymatrix{
T \ar@{=>}[rd]_{\id} \ar@{=>}[r]^{\eta T} & T^2 \ar@{=>}[d]^{\mu} & T \ar@{=>}[l]_{T\eta} \ar@{=>}[ld]^{\id} \\ 
& T & 
}} \quad 
\vcenter{\xymatrix{
T^3 \ar@{=>}[r]^{T\mu} \ar@{=>}[d]_{\mu T} & T^2 \ar@{=>}[d]^{\mu} \\
T^2 \ar@{=>}[r]^{\mu} & T 
}}\]
Thanks to the conditions that 2-arrows in \RDtnl have to satisfy, such monads capture precisely the notion of monadic relation lifting used to reason about effectful programs \cite{Gavazzo18}. 
Similarly, comonads in \RDtnl abstracts comonadic relation liftings \cite{DalLagoG22}. 
\end{example} 

\begin{example}\label[ex]{ex:pw-1ar}
Recall \fun{\VRel\Qtl}{\bop\Set}{\Pos} the doctrine of $\Qtl$-relations from \refItem{ex:rel-doc}{vrel}. 
Consider the 1-arrow \oneAr{P}{\VRel\Qtl}{\VRel\Qtl} where $\fun{\fn{P}}{\Set}{\Set}$ is the covariant powerset functor  and 
$\fun{\lift{P}_{X,Y}}{\VRel\Qtl(X,Y)}{\VRel\Qtl(\fn{P}X,\fn{P}Y)}$ maps a $\Qtl$-relation $\relr$ to the function 
$ \lift{P}_{X,Y}(\relr)(A,B) = h_\relr(A,B) \land h_{\relr\rconv}(B,A)$ where $\land$ denotes the binary meet operation in $\Qtl$ and
for every \fun{\rels}{Z\times W}{\Car\Qtl}, we set 
\[ h_\rels(A,B) = \qinf_{x\in A} \qsup_{y \in B} \rels(x,y)  \qquad \qquad 
\text{for $A \subseteq Z$ and $B\subseteq W$} \] 
It is easy to check that this is indeed a 1-arrow. 
In particular, when considering the boolean quantale $\Bool$, given \fun{\relr}{X\times Y}{\{0,1\}} we have that  $\lift{P}_{X,Y}(\relr)$ relates $A$ and $B$ iff 
for all $x\in A$, there is $y \in B$ s.t. $\relr(x,y) = 1$ and viceversa; 
considering instead Lawvere's quantale $\RPos$, $\lift{P}_{X,Y}(\relr)$ is a generalisation to arbitrary $\RPos$-relations of the Hausdorff pseudometric on subsets of (pseudo)metric spaces. 
\end{example}

\begin{example}[Bisimulations]\label[ex]{ex:bisim}
We can express the notion of bisimulation for coalgebras in an arbitrary relational doctrine, thus covering both usual and quantitative versions of bisimulation. 
If \oneAr{F}{\RDoc}{\RDoc} is a 1-arrow in \RDtnl and \ple{X,c} and \ple{Y,d} two $\fn{F}$-coalgebras, then 
a relation $\relr\in\RDoc(X,Y)$ is a $F$-bisimulation from \ple{X,c} to \ple{Y,d}  if 
$\relr \order \gr{c}\rcomp\lift{F}_{X,Y}(\relr)\rcomp \gr{d}\rconv$ or, equivalently, $\relr\rcomp \gr{d} \order \gr{c}\rcomp \lift{F}_{X,Y}(\relr)$. 
This means that $\relr$ has to agree with the dynamics of the two coalgebras. 
Indeed, if $\RDoc$ is $\Rel$ (the doctrine of set-theoretic relations), 
this condition states that, if $x \in X$ is related to $y \in Y$ by $\relr$ and $y$ evolves to $B\in \fn{F}Y$ through $d$, then 
$x$ evolves to some $A \in \fn{F}X$ through $c$ and $A$ is related to $B$ by the lifted relation $\lift{F}_{X,Y}(\relr)$. 
This definition looks very much like that of simulation, but, since 1-arrows preserve the converse, it is easy to check that, if $\relr$ is a bisimulation, then $\relr\rconv$ is a bisimulation as well, thus justifying the name. 
Furthermore, one can easily check that $F$-bisimulations are closed under relational identities and composition. 
Then, the category of $\fn{F}$-coalgebras is the base of a relational doctrine $\Bisim{F}$ where relations in $\Bisim{F}(\ple{X,c},\ple{Y,d})$ are $F$-bisimulations between coalgebras \ple{X,c} and \ple{Y,d}. 

As a concrete example, let us consider the 1-arrow \oneAr{P}{\VRel\Qtl}{\VRel\Qtl} of \cref{ex:pw-1ar}. 
A $\fn{P}$-coalgebra is a usual (non-deterministic) transition system and a $P$-bisimulation from \ple{X,c} to \ple{Y,d} is a $\Qtl$-relation \fun{\relr}{X\times Y}{\Car\Qtl} such that 
$\relr(x,y) \qord h_\relr(c(x),d(y)) \land h_\relr(d(y),c(x))$, for all $x\in X$ and $y \in Y$. 
Roughly, this means that similar states reduce to similar states. 
When considering the boolean quantale $\Bool$, we get the usual notion of bisimulation, while considering Lawvere's quantale $\RPos$ we get a form of metric bisimulation. 
\end{example}

\begin{example}[Barr lifting] \label[ex]{ex:barr-lift} 
Let \fun{F}{\CC}{\D} be a weak pullback preserving functor between categories with weak pullbacks. 
It induces a strict  1-arrow \oneAr{\ple{F,\lift{F}}}{\Span\CC}{\Span\D} mapping a span $\spn{X}{p_1}{A}{p_2}{Y}$ to $\spn{FX}{Fp_1}{FA}{Fp_2}{FY}$. 
When $F$ is an endofunctor, this construction provides an abstract version of the well-known Barr lifting for set-theoretic relations \cite{Barr70}. 
It is easy to see that this construction extends to a 2-functor $\Span{\blank}$ from the 2-category of categories with weak pullbacks, functor preserving them and natural transformations to the 2-category \RDtn. 
Hence, every weak pullbacks preserving monad on a category \CC with weak pullbacks induces a monad on $\Span\CC$ in \RDtn. 
\end{example}


\section{The intensional quotient completion}
\label{sect:quotients} 

Here we show how one can deal with quotients in relational doctrines extending the quotient completion in \cite{MaiettiME:eleqc, MaiettiME:quofcm} which we used as inspiration for many  notions and constructions. 
We present instances having a quantitative flavour that usual doctrines do not cover, showing that quotients are the key structure characterizing  them.

In  a relational doctrine \fun{\RDoc}{\bop\CC}{\Pos} an \emph{$\RDoc$-equivalence relation} on an object $X$ in \CC is a relation 
$\eqrelr \in \RDoc(X,X)$  satisfying the following properties: 
\begin{center}
\begin{tabular}{lclcl}
reflexivity: $\rid_X\order \eqrelr$&\ &
symmetry: $\eqrelr\rconv \order \eqrelr$&\ &
transitivity: $\eqrelr\rcomp\eqrelr \order \eqrelr$
\end{tabular} 
\end{center}

\begin{example}\label[ex]{ex:eq-rel}
\begin{enumerate}
\item\label{ex:eq-rel:vrel} 
In the doctrine of $\Qtl$-relations $\VRel\Qtl$ (\cf \refItem{ex:rel-doc}{vrel}), an equivalence relation \fun{\eqrelr}{X\times X}{\Car\Qtl} on a set $X$  is a (symmetric) $\Qtl$-metric \cite{HofmannST14}: 
reflexivity is $\qone\qord \eqrelr(x,x)$, for all $x \in X$, 
symmetry is $\eqrelr(x,y)\qord\eqrelr(y,x)$, for all $x,y\in X$, and 
transitivity is $\qsup_{y\in X} \eqrelr(x,y)\qmul\eqrelr(y,z)\qord\eqrelr(x,z)$, which is equivalent to 
$\eqrelr(x,y)\qmul\eqrelr(y,z)\qord\eqrelr(x,z)$, for all $x,y,z\in X$, by properties of suprema. 
For the boolean quantale $\Bool$ these are usual equivalence relations, while 
for the Lawvere's quantale $\RPos$ these are the so-called pseudometrics as the transitivity property is exactly the triangular inequality. 
\item\label{ex:eq-rel:span} 
In the doctrine $\Span\CC$ (\cf \refItem{ex:rel-doc}{span})  of spans in a category with weak pullbacks, an equivalence relation 
on  $X$  is a pair of parallels arrows $\fun{r_1,r_2}{A}{X}$ such that there are  arrows 
$\fun{r}{X}{A}$ with $r_1r=r_2r=\id_X$ (reflexivity), 
$\fun{s}{A}{A}$ with $r_1s=r_2$ and $r_2s=r_1$ (symmetry), and 
$\fun{t}{W}{A}$ with $r_1t=r_1d_1$ and $r_2t=r_2d_2$ where 
\[\xymatrix@C=3ex@R=3ex{
& W\ar[ld]_-{d_1} \ar[rd]^-{d_2} \ar@{}[dd]|{wpb} & \\  
A\ar[rd]_-{r_2} && A \ar[ld]^-{r_1} \\ 
& X 
} \]
is a weak pullback. 
These spans are the pseudo-equivalence relations of \cite{CarboniA:freecl, CarboniA:regec}.
\item\label{ex:eq-rel:vec}
In the relational doctrine \fun{\VecDoc}{\bop{\ct{Vec}}}{\Pos} (\cf \refItem{ex:rel-doc}{vec}) an equivalence relation over a vector space $X$ is 
a semi-norm \fun{\eqrelr}{\Car{X}\times\Car{X}}{[0,\infty]} which is reflexive, that is, 
$0\geq\eqrelr(\vecx,\vecx)$, for all $\vecx\in\Car{X}$. 
Indeed, by reflexivity,  one can show that 
$\eqrelr(\vecx,\vecy) = \eqrelr(\veczero,\vecy-\vecx)$, noting that 
We have $
\eqrelr(\vecx, \vecy) \geq \eqrelr(\vecx, \vecy) + \eqrelr(-\vecx, -\vecx) \geq \eqrelr(0, \vecy-\vecx) \geq \eqrelr(0, \vecy-\vecx)+ \eqrelr(\vecx, \vecx) \geq  \eqrelr(\vecx, \vecy)
$. 
From this we can derive 
symmetry follows by 
$\eqrelr(\vecx,\vecy) = \eqrelr(\veczero,\vecy-\vecx) = |-1|\eqrelr(\veczero,\vecx-\vecy) = \eqrelr(\vecy,\vecx)$ and 
transitivity by 
$\eqrelr(\vecx,\vecy) + \eqrelr(\vecy,\vecz) = \eqrelr(\veczero,\vecy-\vecx) + \eqrelr(\veczero,\vecz-\vecy) \ge \eqrelr(\veczero + \veczero, (\vecy-\vecx) + (\vecz - \vecy)) = \eqrelr(\veczero,\vecz-\vecx) = \eqrelr(\vecx,\vecz)$. 
Hence a $\VecDoc$-equivalence on a vector space $X$ is a lax monoidal and homogeneous pseudometric on it. 
\end{enumerate}
\end{example}

Every arrow \fun{f}{X}{Y} in \CC induces a $\RDoc$-equivalence relation on $X$, dubbed \emph{kernel} of $f$, given by $\gr{f}\rcomp\gr{f}\rconv$. 
The fact that this is an equivalence follows immediately since $\gr{f}$ is a total and functional relation. 
Roughly, the kernel of $f$ relates those elements which are identified by $f$; 
indeed, for the relational doctrine \fun{\Rel}{\bop\Set}{\Pos} of set-theoretic relations it is defined exactly in this way. 
Kernels are crucial to talk about quotients as the following definition shows. 

\begin{definition}\label[def]{def:quotient}
Let \fun{\RDoc}{\bop\CC}{\Pos}  be a relational doctrine on $\CC$ and $\eqrelr$ a $\RDoc$-equivalence relation on an object $X$ in \CC. 
A \emph{quotient arrow} of $\eqrelr$ is an arrow \fun{q}{X}{W} in \CC such that 
$\eqrelr \order \gr{q}\rcomp\gr{q}\rconv$ and, 
for every arrow \fun{f}{X}{Z} with $\eqrelr\order\gr{f}\rcomp\gr{f}\rconv$, there is a unique arrow \fun{h}{W}{Z} such that $f = h\circ q$.  
The quotient arrow $q$ is \emph{effective} if  $\eqrelr = \gr{q}\rcomp\gr{q}\rconv$ and 
it is \emph{descent} if $\rid_W \order \gr{q}\rconv\rcomp\gr{q}$. 

We say that \emph{$\RDoc$ has quotients} if every $\RDoc$-equivalence relation admits an effective descent  quotient arrow. 
\end{definition} 

Intuitively, a quotient of $\eqrelr$ is the ``smallest'' arrow $q$ which transforms the equivalence $\eqrelr$ into the relational identity, that is, such that $\eqrelr$ is smaller than the kernel of $q$. 
The quotient $q$ is effective when its kernel $\gr{q}\rcomp\gr{q}\rconv$ coincides with the equivalence relation $\eqrelr$ and it is descent when  its graph is surjective. 
 
\begin{example} \label[ex]{ex:rel-quot} 
To exemplify the definition above, let us unfold it for the relational doctrine \fun{\Rel}{\bop\Set}{\Pos},  which has quotients. 
Recall from \refItem{ex:eq-rel}{vrel} that a $\Rel$-equivalence is just a usual equivalence relation. 
Here, a quotient arrow for an equivalence relation $\eqrelr$ on a set $X$ is a function \fun{q}{X}{W} 
which is universal among those functions $f$ whose kernel includes the equivalence $\eqrelr$, that is, such that 
$\eqrelr(x,x')$ implies $f(x) = f(x')$. 
Effectiveness requires the converse inclusion, \ie $q(x) = q(x')$ implies $\eqrelr(x,x')$.  
Finally, the descent condition amounts to requiring $q$ to be surjective in the usual sense.
A choice for such a function $q$ 
is the usual quotient projection from $X$ to the set $X/\eqrelr$ of $\eqrelr$-equivalence classes, which maps $x\in X$ to its equivalence class $[x]$. 
Indeed, by definition this function is surjective and $\eqrelr(x,x')$ holds iff $[x] = [x']$. 
Moreover, for every function $f$ such that $\eqrelr(x,x')$ implies $f(x) = f(x')$, 
the function $[x] \mapsto f(x)$ turns out to be well-defined, proving that the quotient projection is universal. 
\end{example}

\begin{example}\label[ex]{ex:eq-rel-quo}
Consider the relational doctrine $\VRel\RPos$ of $\RPos$-relations  where $\RPos$ is the Lawvere's quantale \ple{[0,\infty],\ge,+,0}  (\cf \refItem{ex:rel-doc}{vrel})  and 
suppose \fun{\eqrelr}{X\times X}{[0,\infty]} is a $\VRel\RPos$-equivalence relation , \ie a pseudometric on $X$ (\cf \refItem{ex:eq-rel}{vrel}). 
Define an equivalence relation on $X$ setting $x\sim_\eqrelr y$ whenever $\eqrelr(x,x')\ne\infty$, that is, when $x$ and $x'$ are connected. 
The canonical surjection $\fun{q}{X}{X/\sim}$ mapping $x$ to $q(x)=[x]$ is a quotient arrow for $\eqrelr$. 
It is immediate to see that $\eqrelr(x,x')\ge \rid_{X/\sim_\eqrelr}([x],[x'])$ as $\rid_X([x],[x'])$ is either $0$ or $\infty$ and 
$\rid_{X/\sim_\eqrelr}([x],[x']) = \infty$ precisely when $x$ and $x'$ are not connected, that is, when $\eqrelr(x,x') = \infty$. 
The universality of $q$ easily follows from its universal property as a quotient of $\sim_\eqrelr$ in $\Rel$ (\cf \cref{ex:rel-quot}). 
This shows that $\VRel\RPos$ has quotient arrows for all pseudometrics, which are descent: 
for \fun{q}{X}{X/\sim_\eqrelr} a quotient of $\eqrelr$, 
the descent condition becomes 
$ 
\rid_{X/\sim_\eqrelr}(y,y')\ge \inf_{x\in X}\left( \rid_{X/\sim_\eqrelr}(y,q(x))+\rid_{X/\sim_\eqrelr}(q(x),y')\right)
$, 
which trivially holds since $q$ is surjective and $\rid_{X/\sim_\eqrelr}$ is either $0$ or $\infty$. 
However, such quotient arrows cannot be effective. 
Indeed, if \fun{f}{X}{Y}  is a function, since the relational identity $\rid_Y$ is only either $0$ or $\infty$, 
the kernel of $f$ is given by $x,x' \mapsto \rid_Y(f(x),f(x'))$, thus it takes values in $\{0,\infty\}$. 
Hence, if a quotient arrow $q$ for a pseudometric $\eqrelr$ were effective, then $\eqrelr$ would be either $0$ or $\infty$, as it would coincide with the kernel of $q$ and clearly this is not the case in general. 
This shows that $\VRel\RPos$ has not quotients in the sense of \cref{def:quotient}. 
\end{example}

\begin{example}\label[ex]{ex:quot-coeq}
Let \CC be a regular category and consider the relational doctrine $\JSpan\CC$ of jointly monic spans in \CC. 
A $\JSpan\CC$-equivalence relation  on an object $X$ is a jointly monic span $\spn{X}{p_1}{A}{p_2}{X}$ satisfying reflexivity, symmetry and transitivity. 
Then, one can easily see that a quotient arrow for such a span is just a coequalizer of $p_1$ and $p_2$, hence it is a regular epimorphism. 
Moreover, every regular epimorphism is the coequalizer of its kernel pair, which is  a jointly monic span that can be easily proved to be a $\JSpan\CC$-equivalence relation. 
Therefore, quotient arrows in $\JSpan\CC$ are exactly regular epimorphisms. 
Finally, note that \CC is exact precisely when $\JSpan\CC$ has quotients. 
\end{example}

In a relational doctrine with quotients, a form of the homomorphism theorem holds. 
More precisely, quotient and injective arrows determine an orthogonal factorization system, as the following results show. 

\begin{proposition}\label[prop]{prop:qi-lift}
Let \fun\RDoc{\bop\CC}{\Pos} be a relational doctrine with quotients. 
For every commutative square in \CC 
\[\xymatrix{
  X \ar[r]^-{f} \ar[d]_-{q} 
& Y \ar[d]^-{i} \\
  W \ar[r]_-{g} \ar@{..>}[ru]^-{h} 
& Z 
}\]
where $q$ is a quotient arrow and $i$ is injective, 
there exists a unique arrow $h$ making the two triangles commute. 
\end{proposition}
\begin{proof}
Since the square commutes and $i$ is injective, we have the following inequalities:
\[ 
\gr{q}\rcomp\gr{q}\rconv 
  \order \gr{q}\rcomp\gr{g}\rcomp\gr{q}\rconv\rcomp\gr{q}\rconv 
  = \gr{f}\rcomp\gr{i}\rcomp\gr{i}\rconv\rcomp\gr{f}\rconv 
  \order \gr{f}\rcomp\gr{f}\rconv 
\]
Hence, since $q$ is a quotient arrow, we get a unique arrow \fun{h}{W}{Y} such that 
$f = h\circ q$. 
Moreover, we have $g = i\circ h$ because 
$g\circ q = i\circ f = i\circ h\circ q$ and $q$ is a quotient arrow. 
\end{proof}

\begin{proposition}\label[prop]{prop:qi-factor}
Let \fun\RDoc{\bop\CC}{\Pos} be a relational doctrine with quotients and \fun{f}{X}{Y} an arrow in \CC. 
Then, $f$ factors as $i\circ q$ where $q$ is a quotient arrow and $i$ is injective. 
\end{proposition}
\begin{proof}
Let \fun{q}{X}{W} be a quotient arrow for the $\RDoc$-equivalence $\eqrelr = \gr{f}\rcomp\gr{f}\rconv$. 
Hence, we get an arrow \fun{i}{W}{Y} such that $f = i\circ q$. 
Towards a proof that $i$ is injective, note that 
$\gr{f} = \gr{q}\rcomp\gr{i}$ implies $\gr{i} = \gr{q}\rconv\rcomp\gr{f}$ because $q$ is surjective. 
Hence, we conclude 
$\gr{i}\rcomp\gr{i}\rconv = \gr{q}\rconv\rcomp\gr{f}\rcomp\gr{f}\rconv\rcomp\gr{q} = \gr{q}\rconv\rcomp\gr{q}\rcomp\gr{q}\rconv\rcomp\gr{q} = \rid_W$, as needed. 
\end{proof}

\cref{prop:qi-factor} shows that any arrow factors as a surjective arow followed by an injective one. 
However, in general surjections and injections do not form a factorization system essentially because not all surjections are quotients and the latter property is crucial  in the proof of \cref{prop:qi-lift,prop:qi-factor}. 
Indeed, one can easly prove that, if \fun{s}{X}{Y} is a surjective arrow, then in the factorization $s = i\circ q$ of \cref{prop:qi-factor} the arrow $i$ is  bijective and, unless the doctrine is balanced (\cf\cref{def:balance}) there is no way to deduce that $i$ is an isomorphism. 
In fact, the following result holds. 

\begin{proposition}\label[prop]{prop:quot-surj-balance}
Let $\RDoc$ be a relational doctrine with quotients. 
Then, $\RDoc$ is balanced if and only if every surjective arrow in the base of $\RDoc$ is a quotient arrow. 
\end{proposition}
\begin{proof}
To prove the left-to-right implication, consider a surjective arrow \fun{s}{X}{Y} and its factorization $s = i\circ q$ as in \cref{prop:qi-factor}. 
Then, from $\gr{i} = \gr{q}\rconv\rcomp\gr{s}$ we deduce 
$ \gr{i}\rconv\rcomp\gr{i} 
  = \gr{s}\rconv\rcomp\gr{q}\rcomp\gr{q}\rconv\rcomp\gr{s} 
  = \gr{s}\rconv\rcomp\gr{s}\rcomp\gr{s}\rconv\rcomp\gr{s} 
  = \rid_Y$, 
hence $i$ is surjective. 
Then, since $\RDoc$ is balanced, we conclude that $i$ is an isomorphism and thus $s$ is a quotient arrow. 

Towards a proof of the right-to-left implication, consider a bijective arrow \fun{f}{X}{Y}. 
Then, $f$ is surjective and so it is a quotient arrow of $\gr{f}\rcomp\gr{f}\rconv = \rid_X$. 
Therefore, there is a unique arrow \fun{g}{Y}{X} such that $g\circ f = \id_X$, that is, $f$ is a split monomorphism. 
Finally, since $f$ is a quotient arrow, it is an epimorphism and so we conclude that $f$ is an isomorphism, as needed. 
\end{proof}

Therefore, in a balanced relational doctrine with quotients 
surjective and injective arrows form an orthogonal factorization system. 

\medskip 

\cref{ex:eq-rel-quo} shows that relational doctrines need not have quotients in general.
Hence, we now describe a free construction that takes a relational doctrine \fun{\RDoc}{\bop\CC}{\Pos} and builds a new one 
\fun{\QR\RDoc}{\bop{\QC\RDoc}}{\Pos} which has (effective descent) quotients for all equivalence relations. 
The construction is inspired by the quotient completion in \cite{MaiettiME:eleqc,MaiettiME:quofcm} and a comparison with it is delayed to \cref{sect:eed}.

The  category \QC\RDoc is defined as follows: 
\begin{itemize}
\item an object is a pair \ple{X,\eqrelr}, where $X$ is an object in \CC and $\eqrelr$ is a $\RDoc$-equivalence relation on $X$,
\item an arrow \fun{f}{\ple{X,\eqrelr}}{\ple{Y,\eqrels}} is an arrow \fun{f}{X}{Y} in \CC such that 
$\eqrelr \order \RDoc\reidx{f,f}(\eqrels)$, and 
\item composition and identities are those of \CC. 
\end{itemize} 
By \cref{prop:left-adj} the condition $\eqrelr \order \RDoc\reidx{f,f}(\eqrels)$ is equivalent to
both $\eqrelr\order \gr{f}\rcomp \eqrels\rcomp \gr{f}\rconv$ 
and $\gr{f}\rconv\rcomp \eqrelr\rcomp\gr{f} \order \eqrels$. 

Given $\RDoc$-equivalence relations $\eqrelr$ and $\eqrels$ over $X$ and $Y$, respectively, 
the suborder $\Des\eqrelr\eqrels(X,Y)$ of $\RDoc(X,Y)$ of \emph{descent data} with respect to $\eqrelr$ and $\eqrels$ is defined by 
\[\Des\eqrelr\eqrels(X,Y) = \{\relr\in\RDoc(X,Y) \mid \eqrelr\rconv \rcomp \relr \rcomp \eqrels \order \relr\}\] 
Roughly, a descent datum is a relation which is closed w.r.t. $\eqrelr$ on the left and $\eqrels$ on the right. 
For every arrow \fun{f}{\ple{X,\eqrelr}}{\ple{X',\eqrelr'}} and \fun{g}{\ple{Y,\eqrels}}{\ple{Y',\eqrels'}} in \QC\RDoc, 
the monotone function \fun{\RDoc\reidx{f,g}}{\RDoc(X',Y')}{\RDoc(X,Y)} applies $\Des{\eqrelr'}{\eqrels'}(X',Y')$ into $\Des\eqrelr\eqrels(X,Y)$ as
Indeed, for $\relr\in\Des{\eqrelr'}{\eqrels'}(X',Y')$, we have 
\[ \eqrelr\rconv \rcomp \RDoc\reidx{f,g}(\relr) \rcomp \eqrels 
  \order \RDoc\reidx{f,f}({\eqrelr'}\rconv) \rcomp \RDoc\reidx{f,g}(\relr) \rcomp \RDoc\reidx{g,g}(\eqrels') 
  \order \RDoc\reidx{f,g}({\eqrelr'}\rconv\rcomp\relr\rcomp\eqrels') 
  \order \RDoc\reidx{f,g}(\relr)\]
Therefore the assignments
$\QR\RDoc(\ple{X,\eqrelr},\ple{Y,\eqrels}) = \Des\eqrelr\eqrels(X,Y)$ and $\QR\RDoc\reidx{f,g} = \RDoc\reidx{f,g}$ determine a functor  \fun{\QR\RDoc}{\bop{\QC\RDoc}}{\Pos}.

\begin{proposition}\label[prop]{prop:qc-doc}
Let $\RDoc$ be a relational doctrine over \CC. 
The functor \fun{\QR\RDoc}{\bop{\QC\RDoc}}{\Pos} is a relational doctrine, 
where composition and converse are those of $\RDoc$ and $\rid_{\ple{X,\eqrelr}} = \eqrelr$. 
\end{proposition}
\begin{proof} 
We first check that composition and converse are well-defined. 
For  $\relr\in\Des\eqrelr\eqrels(X,Y)$ and $\rels\in\Des\eqrels\eqrelt(Y,Z)$, 
the properties of the converse lead to 
$
\eqrels\rconv\rcomp  \relr\rconv \rcomp \eqrelr 
  = (\eqrelr\rconv \rcomp \relr \rcomp \eqrels)\rconv 
  \order \relr\rconv 
$, proving that $\relr\rconv\in\Des\eqrels\eqrelr(Y,X)$. 
By reflexivity and symmetry of $\eqrels$, we get 
\[
\eqrelr\rconv\rcomp\relr\rcomp\rels\rcomp\eqrelt 
  = \eqrelr\rconv\rcomp\relr\rcomp\rid_Y\rcomp\rid_Y\rcomp\rels\rcomp\eqrelt 
  \order \eqrelr\rconv\rcomp\relr\rcomp\eqrels\rcomp\eqrels\rconv\rcomp\rels\rcomp\eqrelt 
  \order \relr\rcomp\rels
\]
showing that $\relr\rcomp\rels\in\Des\eqrelr\eqrelt(X,Z)$. 
Finally, by transitivity of $\eqrelr$, we have 
$
\eqrelr\rconv\rcomp\rid_{\ple{X,\eqrelr}}\rcomp\eqrelr 
  = \eqrelr\rcomp\eqrelr\rcomp\eqrelr 
  \order \eqrelr 
  = \rid_{\ple{X,\eqrelr}}
$, proving that $\rid_{\ple{X,\eqrelr}}\in\Des\eqrelr\eqrelr(X,X)$. 

Since composition and converse are those of $\RDoc$, the equational axioms of \cref{def:rel-doc} concerning  only composition and converse hold trivially. 
Hence, we have only to check those involving the relaitonal identity. 
By simmetry of $\eqrelr$, we have $\rid_{\ple{X,\eqrelr}}\rconv = \eqrelr\rconv = \eqrelr = \rid_{\ple{X,\eqrelr}}$. 
By the descent condition and reflexivity and symmetry of $\eqrelr$ and $\eqrels$, we have 
\[
\relr\rcomp\rid_{\ple{Y,\eqrels}} =\rid_X\rcomp\relr\rcomp\rid_{\ple{Y,\eqrels}}
  = \rid_X\rcomp \relr\rcomp\eqrels 
  \order \eqrelr \rcomp \relr \rcomp\eqrels   
  = \eqrelr\rconv \rcomp \relr \rcomp\eqrels 
  \order \relr 
  = \relr\rcomp\rid_Y 
  \order \relr \rcomp \eqrels 
  = \relr\rcomp\rid_{\ple{Y,\eqrels}} 
\]
proving that $\relr\rcomp\rid_{\ple{Y,\eqrels}} = \relr$. 
The equality $\rid_{\ple{X,\eqrelr}}\rcomp\relr = \relr$ is redundant.
\end{proof}

A $\QR\RDoc$-equivalence relation over an object \ple{X,\eqrelr} is a $\RDoc$-equivalence $\eqrels$ over $X$ such that $\eqrelr\order\eqrels$. 
Note that these conditions imply that $\eqrels$ is a descent datum in $\Des\eqrelr\eqrelr(X,X)$. 
Then, \ple{X,\eqrels} is an object of \QC\RDoc and 
\fun{\id_X}{\ple{X,\eqrelr}}{\ple{X,\eqrels}} 
is a well-defined arrow in \QC\RDoc, which turns out to be 
an effective descent quotient arrow for $\eqrels$. 
In this way we construct quotient arrows for all $\QR\RDoc$-equivalence relations, thus obtaining the following result. 

\begin{proposition}\label[prop]{prop:qc-quot-doc}
Let $\RDoc$ be a relational doctrine over $\CC$. 
The relational doctrine $\QR\RDoc$ over $\QC\RDoc$ has effective descent quotients. 
\end{proposition}
\begin{proof} 
Let \ple{X,\eqrelr} be an object in \QC\RDoc and $\eqrels$ a $\QR\RDoc$-equivalence over it. That is $\eqrels\order\eqrels\rconv$ and $\eqrels\rcomp\eqrels \order \eqrels$ and $\rid_{\ple{X,\eqrelr}} = \eqrelr \order \eqrels$. By reflexivity of $\eqrelr$, we get $\rid_X \order \eqrelr \order \eqrels$, proving that $\eqrels$ is a $\RDoc$-equivalence over $X$. 
Therefore, \ple{X,\eqrels} is an object in \QC\RDoc and 
\fun{\id_X}{\ple{X,\eqrelr}}{\ple{X,\eqrels}} is a well-defined arrow in \QC\RDoc, since $\eqrelr \order \eqrels = \RDoc\reidx{\id_X,\id_X}(\eqrels)$. 
To prove this is a quotient arrow for $\eqrels$, let us consider  an arrow \fun{f}{\ple{X,\eqrelr}}{\ple{Y,\eqrelt}} in \QC\RDoc such that $\eqrels \order \QR\RDoc\reidx{f,f}(\rid_{\ple{Y,\eqrelt}}) = \RDoc\reidx{f,f}(\eqrelt)$. 
So the arrow \fun{f}{\ple{X,\eqrels}}{\ple{Y,\eqrelt}} is well-defined in \QC\RDoc and the diagram 
\[\xymatrix{
\ple{X,\eqrelr} \ar[d]_-{\id_X} \ar[rd]^-{f} \\
\ple{X,\eqrels} \ar[r]^-f & \ple{Y,\eqrelt} 
}\]
trivially commutes. Uniqueness of $f$ is obvious. 
Effectiveness follows  from $\eqrels = \RDoc\reidx{\id_X,\id_X}(\eqrels)=\QR\RDoc\reidx{\id_X,\id_X}(\rid_{\ple{X,\eqrels}})$.
Finally, \fun{\id_X}{\ple{X,\rid_X}}{\ple{X,\eqrelr}} is descent because 
its graph in $\QR\RDoc$ is given by $\QR\RDoc\reidx{\id_X,\id_{\ple{X,\eqrelr}}}(\rid_{\ple{X,\eqrelr}}) = \eqrelr$, hence we have 
$\rid_{\ple{X,\rid_X}} = \rid_X = \rid_X\rcomp\rid_X \order \eqrelr\rconv\rcomp\eqrelr$, by reflexivity and symmetry of $\eqrelr$, which proves that the graph of $\id_X$ is surjective in $\QR\RDoc$. 
\end{proof}

\begin{example} \label[ex]{ex:qcat}
\begin{enumerate}
\item\label{ex:qcat:vrel} 
For the doctrine $\VRel\Qtl$ of $\Qtl$-relations, the category $\QC{\VRel\Qtl}$ is the category of $\Qtl$-metric spaces with non-expansive maps. 
By \refItem{ex:eq-rel}{vrel}, an object \ple{X,\eqrelr} is a $\Qtl$-metric space and \fun{f}{\ple{X,\eqrelr}}{\ple{Y,\eqrels}} has to satisfy $\eqrelr(x,x') \qord \eqrels(f(x),f(x'))$.
\item\label{ex:qcat:vec} 
For the relational doctrine $\VecDoc$ over the category of real vector spaces, 
 $\QC{\VecDoc}$ is the category of semi-normed vector spaces with short maps. 
An object \ple{X,\eqrelr} in \QC\VecDoc is a vector space with a lax monoidal and homogeneous pseudometric on it. Such a pseudometric satisfies $\eqrelr(\vecx,\vecy) = \eqrelr(\veczero,\vecy-\vecx)$ (see \refItem{ex:eq-rel}{vec}). 
Then, it is easy to see that the assignment $\| \vecx \|_\eqrelr = \eqrelr(\veczero,\vecx)$ defines a semi-norm on $X$, whose induced pseudometric is exactly $\eqrelr$. 
\end{enumerate}
\end{example}

Following Lawvere's structural approach to  logic, 
we can characterise
the property of having effective descent quotients by an adjunction in \RDtnl (hence, also in \RDtnl). 
First observe that 
the doctrine $\RDoc$ is embedded into $\QR\RDoc$ by the (strict) 1-arrow  \oneAr{\QAr\RDoc}{\RDoc}{\QR\RDoc} in \RDtnl defined as follows: 
the functor \fun{\fn{\QAr\RDoc}}{\CC}{\QC\RDoc} maps $\fun{f}{X}{Y}$ in \CC to \fun{f}{\ple{X,\rid_X}}{\ple{Y,\rid_Y}}; 
the natural transformation \nt{\lift{\QAr\RDoc}}{\RDoc}{\QR\RDoc\circ \bop{\fn{\QAr\RDoc}}} is the family of identities $\RDoc(X,Y) =\Des{\rid_X}{\rid_Y}(X,Y)$.
The 1-arrow $\QAr\RDoc$ shows that constructing $\QR\RDoc$ ``extends'' $\RDoc$ adding (effective descent) quotients for any equivalence relation. 

\begin{lemma}\label[lem]{lem:quot-adj}
A relational doctrine $\RDoc$ has effective descent quotients if and only if 
$\QAr\RDoc$ has a strict reflection left adjoint \oneAr{F}{\QR\RDoc}{\RDoc}. 
\end{lemma}
This means that the 1-arrow \oneAr{F}{\QR\RDoc}{\RDoc} is strict and it is a left adjoint of $\QAr\RDoc$ in \RDtnl and the counit of this adjunction is an isomorphism, hence $F\circ\QAr\RDoc \cong \Id_\RDoc$. 
Intuitively, the 1-arrow \oneAr{F}{\QR\RDoc}{\RDoc} computes quotients of $\RDoc$-equivalence relations: the object $\fn{F}\ple{X,\eqrelr}$ is the codomain of a quotient arrow obtained by applying $\fn{F}$ to \fun{\id_X}{\ple{X,\rid_X}}{\ple{X,\eqrelr}} which is the quotient arrow of $\eqrelr$ in $\QR\RDoc$ viewed as a $\QR\RDoc$-equivalence over \ple{X,\rid_X}. 

\begin{proof}[of \cref{lem:quot-adj}]
Let \fun{\RDoc}{\bop\CC}{\Pos} be a relational doctrine and assume that \oneAr{F}{\QR\RDoc}{\RDoc} is a reflection left adjoint of \oneAr{\QAr\RDoc}{\RDoc}{\QR\RDoc} in \RDtnl and denote by $\epsilon$ the counit of the adjunction $F\dashv\QAr\RDoc$, which is an isomorphism. 
We have to prove that $\RDoc$ has effective descent quotients.
Let us consider a $\RDoc$-equivalence $\eqrelr$ over $X$ in \CC. 
Denote by $q_\eqrelr$ the arrow 
\fun{\fn{F}(\id_X)\circ\epsilon_X^{-1}}{X}{\fn{F}\ple{X,\eqrelr}} in \CC, where 
\fun{\id_X}{\ple{X,\rid_X}}{\ple{X,\eqrelr}} is the quotient arrow of $\eqrelr$ over \ple{X,\rid_X} in $\QR\RDoc$. 
We show  this is a quotient arrow for $\eqrelr$ in $\RDoc$. 
Let \fun{f}{X}{Y} be an arrow in \CC such that $\eqrelr \order \gr{f}\rcomp\gr{f}\rconv$, then \fun{f}{\ple{X,\eqrelr}}{\ple{Y,\rid_Y}} is an arrow in \QC\RDoc and we denote by $\hat f$ its transpose along the adjunction, that is, 
\fun{\hat f = \epsilon_Y\circ \fn{F}(f)}{\fn{F}\ple{X,\eqrelr}}{Y}. 
Then we have 
\[
(\epsilon_Y\circ \fn{F}(f))\circ(\fn{F}(\id_X)\circ \epsilon_X^{-1}
  = \epsilon_Y\circ \fn{F}(\fn{\QAr\RDoc}(f)) \circ \epsilon_X^{-1}
  =f \circ \epsilon_X \circ \epsilon_X^{-1}
  = f 
\]
because in \QC\RDoc we have $f\circ\id_X = \fun{\QAr\RDoc}(f)$. 
Now, let \fun{g}{F\ple{X,\eqrelr}}{Y} be an arrow in \CC such that 
$g\circ q_\eqrelr = f$. 
Then, \fun{\fn{\QAr\RDoc}(g)\circ\eta_{\ple{X,\eqrelr}}}{\ple{X,\eqrelr}}{\ple{Y,\rid_Y}} satisfies 
\begin{align*} 
(\fn{\QAr\RDoc}(g)\circ\eta_{\ple{X,\eqrelr}})\circ \id_X 
  &= \fn{\QAr\RDoc}(g) \circ \fn{\QAr\RDoc}(F(\id_X))\circ \eta_{\ple{X,\rid_X}}
   = \fn{\QAr\RDoc}(g) \circ \fn{\QAr\RDoc}(q_\eqrelr\circ\epsilon_X) \circ \eta_{\ple{X,\rid_X}}
\\
  &= \fn{\QAr\RDoc}(f) \circ \fn{\QAr\RDoc}(\epsilon_X) \circ \eta_{\ple{X,\rid_X}} 
   = \fn{\QAr\RDoc}(f)
\end{align*} 
Therefore, since \fun{\id_X}{\ple{X,\rid_X}}{\ple{X,\eqrelr}} is a quotient arrow and both $\QAr\RDoc(g)\circ\eta_{\ple{X,\eqrelr}}$ and \fun{f}{\ple{X,\eqrelr}}{\ple{Y,\rid_Y}} are factorisation of $\QAr\RDoc(f)$ along $\id_X$, they must be equal. 
Hence, we get 
$g = \epsilon_Y\circ \fn{F}(f)$, 
proving that the factorisation of $f$ along $q_\eqrelr$ is unique. 

To check that $q_\eqrelr$ is efective, 
note that $\rid_{\fn{F}\ple{X,\eqrelr}} = \lift{F}_{\ple{X,\eqrelr},\ple{X,\eqrelr}}(\rid_{\ple{X,\eqrelr}})$, as $\lift{F}$ preserves identity relations. 
Hence we get 
\begin{align*} 
\gr{q_\eqrelr}\rcomp \rid_{\fn{F}\ple{X,\eqrelr}} \rcomp \gr{q\eqrelr}\rconv 
  &= \gr{\epsilon_X^{-1}}\rcomp \RDoc\reidx{\fn{F}(\id_X),\fn{F}(\id_X)}(\rid_{\fn{F}\ple{X,\eqrelr}}) \rcomp \gr{\epsilon_X^{-1}} 
\\
  &= \gr{\epsilon_X^{-1}} \rcomp \lift{F}_{\ple{X,\rid_X},\ple{X,\rid_X}}(\QR\RDoc\reidx{\id_X,\id_X}(\rid_{\ple{X,\eqrelr}})) \rcomp \gr{\epsilon_X^{-1}} 
\\
  &= \gr{\epsilon_X^{-1}} \rcomp \lift{F}_{\fn{\QAr\RDoc}X,\fn{\QAr\RDoc}X}(\eqrelr) \rcomp\gr{\epsilon_X^{-1}}
   = \eqrelr 
\end{align*} 
because $\lift{F}$ is natural and $\epsilon$ is an invertible 2-arrow in \RDtnl. 
To check that $q_\eqrelr$ is descent,  note that 
$ \gr{\fn{F}(\id_X)}\rconv\rcomp\gr{\fn{F}(\id_X)} 
    = \lift{F}_{\ple{X,\eqrelr},\ple{X,\eqrelr}}(\eqrelr)\rconv\rcomp \lift{F}_{\ple{X,\eqrelr},\ple{X,\eqrelr}}(\eqrelr) 
    = \lift{F}_{\ple{X,\eqrelr},\ple{X,\eqrelr}}(\eqrelr)$, 
because $\lift{F}$ preserve composition and graphs and $\eqrelr$ is reflexive, symmetric and transitive. 
Hence, we get 
\begin{align*} 
\rid_{\fn{F}\ple{X,\eqrelr}} 
  &= \lift{F}_{\ple{X,\eqrelr},\ple{X,\eqrelr}}(\rid_{\ple{X,\eqrelr}}) 
   = \lift{F}_{\ple{X,\eqrelr},\ple{X,\eqrelr}}(\eqrelr)  
\\
  &= \gr{\fn{F}(\id_X)}\rconv\rcomp \gr{\fn{F}(\id_X)} 
   = \gr{\fn{F}(\id_X)}\rconv\rcomp \gr{\epsilon_X^{-1}}\rconv\rcomp\gr{\epsilon_X^{-1}}\rcomp \gr{\fn{F}(\id_X)} 
   = \gr{q_\eqrelr}\rconv\rcomp\gr{q_\eqrelr}
\end{align*} 
because $\lift{F}$ preserves relational identities and $\epsilon$ is an isomorphism, hance both functional and sujective. 

\medskip

Let now \fun{\RDoc}{\bop\CC}{\Pos} be a relational doctrine with effective descent quotients. 
For every $\RDoc$-equivalence relation $\eqrelr$ over an object $X$, let us choose a quotient arrow 
\fun{q_{\ple{X,\eqrelr}}}{X}{X/\eqrelr}. 
For every arrow \fun{f}{\ple{X,\eqrelr}}{\ple{Y,\eqrels}}, using the universal property of $q_{\ple{X,\eqrelr}}$, there is a unique arrow \fun{\hat{f}}{X/\eqrelr}{Y/\eqrels} in \CC such that $\hat{f}\circ q_{\ple{X,\eqrelr}} = q_{\ple{Y,\eqrels}}\circ f$. 
Then, let us set 
$\fn{F}\ple{X,\eqrelr} = X/\eqrelr$ and $\fn{F}(f) = \hat{f}$, hence 
\fun{\fn{F}}{\QC\RDoc}{\CC} is a functor thanks to the fact that $\hat f$ is uniquely determined by $f$. 
Moreover, let us set $\lift{F}_{\ple{X,\eqrelr},\ple{Y,\eqrels}}(\relr) = \gr{q_{\ple{X,\eqrelr}}}\rconv \rcomp \relr\rcomp \gr{q_{\ple{Y,\eqrels}}}$, for every $\relr\in\QR\RDoc(\ple{X,\eqrelr},\ple{Y,\eqrels})$. 
In order to check that $\lift{F}$ is a natural transformation, let us note that given \fun{f}{\ple{X,\eqrelr}}{\ple{X',\eqrelr'}}, since $\hat{f}\circ q_{\ple{X,\eqrelr}} = q_{\ple{X',\eqrelr'}}\circ f$ and $\gr{q_{\ple{X',\eqrelr'}}}$ is surjective, we have 
$\gr{q_{\ple{X,\eqrelr}}}\rconv\rcomp\gr{f} = \gr{\hat f}\circ\gr{q_{\ple{X',\eqrelr'}}}\rconv$. 
Then, for \fun{f}{\ple{X,\eqrelr}}{\ple{X',\eqrelr'}} and \fun{g}{\ple{Y,\eqrels}}{\ple{Y',\eqrels'}} and $\relr\in\QR\RDoc(\ple{X',\eqrelr'},\ple{Y',\eqrels'})$, we  get 
\begin{align*}
\gr{q_{\ple{X,\eqrelr}}}\rconv \rcomp \QR\RDoc\reidx{f,g}(\relr) \rcomp\gr{q_{\ple{Y,\eqrels}}} 
  &= \gr{q_{\ple{X,\eqrelr}}}\rconv \rcomp \gr{f}\rcomp\eqrelr'\relr\rcomp {\eqrels'}\rconv \rcomp\gr{g}\rconv\rcomp\gr{q_{\ple{Y,\eqrels}}} 
\\
  &= \gr{\hat f} \rcomp \gr{q_{\ple{X',\eqrelr'}}}\rconv \rcomp \relr \rcomp \gr{q_{\ple{Y',\eqrels'}}} \rcomp \gr{\hat g}\rconv 
\\
  &= \RDoc\reidx{\hat f, \hat g}(\gr{q_{\ple{X',\eqrelr'}}}\rconv \rcomp \relr \rcomp \gr{q_{\ple{Y',\eqrels'}}} ) 
\end{align*}
which proves naturality. 
For every object \ple{X,\eqrelr} in \QC\RDoc, since $q_{\ple{X,\eqrelr}}$ is effective, we have 
$\eqrelr = \gr{q_{\ple{X,\eqrelr}}}\rcomp\gr{q_{\ple{X,\eqrelr}}}\rconv$, which implies 
$\lift{F}_{\ple{X,\eqrelr},\ple{X,\eqrelr}}(\rid_{\ple{X,\eqrelr}}) = \gr{q_{\ple{X,\eqrelr}}}\rconv\rcomp\eqrelr\rcomp\gr{q_{\ple{X,\eqrelr}}} \order \rid_{X/\eqrelr}$. 
On the other hand, since $q_{\ple{X,\eqrelr}}$ is descente, $\gr{q_{\ple{X,\eqrelr}}}$ is surjective, hence we have 
\begin{align*}
\rid_{X/\eqrelr} 
  &\order \gr{q_{\ple{X,\eqrelr}}}\rconv\rcomp\gr{q_{\ple{X,\eqrelr}}} 
   \order \gr{q_{\ple{X,\eqrelr}}}\rconv\rcomp\eqrelr\rcomp \gr{q_{\ple{X,\eqrelr}}} 
   = \lift{F}_{\ple{X,\eqrelr},\ple{X,\eqrelr}}(\rid_{\ple{X,\eqrelr}}) 
\end{align*}
thus proving that $\lift{F}$ preserves relational identities. 
For every $\relr\in\QR\RDoc(\ple{X,\eqrelr},\ple{Y,\eqrels})$ and $\rels\in\QR\RDoc(\ple{Y,\eqrels},\ple{Z,\eqrelt})$, since $q_{\ple{Y,\eqrels}}$ is effective and $\rels$ is a descent datum, we have 
\begin{align*}
\lift{F}_{\ple{X,\eqrelr},\ple{Y,\eqrels}}(\relr)\rcomp\lift{F}_{\ple{Y,\eqrels},\ple{Z,\eqrelt}}(\rels) 
  &= \gr{q_{\ple{X,\eqrelr}}}\rconv\rcomp\relr\rcomp\gr{q_{\ple{Y,\eqrels}}}\rcomp\gr{q_{\ple{Y,\eqrels}}}\rconv\rcomp\rels\rcomp\gr{q_{\ple{Z,\eqrelt}}} 
   = \gr{q_{\ple{X,\eqrelr}}}\rconv\rcomp\relr\rcomp \eqrels \rcomp\rels\rcomp\gr{q_{\ple{Z,\eqrelt}}} 
\\
  &\order \gr{q_{\ple{X,\eqrelr}}}\rconv\rcomp\relr\rcomp\rels\rcomp\gr{q_{\ple{Z,\eqrelt}}} 
   = \lift{F}_{\ple{X,\eqrelr},\ple{Z,\eqrelt}}(\relr\rcomp\rels)
\end{align*}
The other inequality follows just because $\gr{q_{\ple{Y,\eqrels}}}$ is total, thus proving that $\lift{F}$ preserves relational composition. 
Preservation of the converse is straightforward. 

What we have observed shows that \oneAr{F}{\QR\RDoc}{\RDoc} is a 1-arrow in \RDtnl. 
To prove it is a reflection left adjoint of $\QAr\RDoc$, we have to define the unit and the counit of the adjunction. 
Note that, for every $X$ in \CC, the arow \fun{q_{\ple{X,\rid_X}}}{X}{X/\rid_X} is an iso: 
by the universal property of $q_{\ple{X,\rid_X}}$ we gat an arrow $p_X$ such that $p_X \circ q_{\ple{X,\rid_X}} = \id_X$ and, 
since $q_{\ple{X,\rid_X}}\circ p_X\circ q_{\ple{X,\rid_X}}$ and $q_{\ple{X,\rid_X}}$ is an epi, we get $q_{\ple{X,\rid_X}}\circ p_X = \id_{X/\rid_X}$. 
Then, let us set 
\fun{\epsilon_X = q_{\ple{X,\rid_X}}^{-1}}{X/\rid_X}{X}  and 
\fun{\eta_{\ple{X,\eqrelr}} = q_{\ple{X,\eqrelr}}}{\ple{X,\eqrelr}}{\ple{X/\eqrelr,\rid_{X,\eqrelr}}}. 
These are both easily natural and
the latter is well-defined because $a_{\ple{X,\eqrelr}}$ is a quotient arrow. 
Finally, the inequalities of 2-arrows in \RDtnl and the triangular identities are easy to verify. 
\end{proof}

The construction of $\QR\RDoc$ is universal as it gives rise to a left biadjoint. 
More precisely, we will first show that is is part of a lax 2-adjunction \cite{BettiP88} involving \RDtnl, 
which restricts to a 2-mondic 2-adjunction when 1-arrows are strict, that is, when considering the 2-subcategory \RDtnl. 

To show this, we first introduce the 2-category \QRDtnl as the 2-full 2-subcategory of \RDtnl whose objects are relational doctrines with quotients and whose 1-arrows are those of \RDtnl that preserve quotient arrows, \ie  
1-arrows \oneAr{F}{\RDoc}{\SDoc} in \RDtnl mapping a quotient arrow for a $\RDoc$-equivalence $\eqrelr$ over $X$ to a quotient arrow for
$\lift{F}_{X,X}(\eqrelr)$, which  can be easily proved to be a $\SDoc$-equivalence over $\fn{F}X$.
There is an obvious inclusion 2-functor \fun{\QRFun}{\QRDtnl}{\RDtnl} which simply forgets quotients. 

The construction of the doctrine $\QR\RDoc$ determines a 2-functor \fun{\RQFun}{\RDtnl}{\QRDtnl}, defined as follows: 
For a relational doctrine $\RDoc$, we set $\RQFun(\RDoc) = \QR\RDoc$. 
for a 1-arrow \oneAr{F}{\RDoc}{\SDoc} in \RDtnl, the 1-arrow \oneAr{\RQFun(F) = \QR{F}}{\QR\RDoc}{\QR\SDoc} is given by 
$\fn{\QR{F}}\ple{X,\eqrelr} = \ple{\fn{F}X,\lift{F}_{X,X}(\eqrelr)}$ and $\fn{\QR{F}} f = \fn{F}f$ and $\lift{\QR{F}}_{\ple{X,\eqrelr},\ple{Y,\eqrels}}(\relr) = \lift{F}_{X,Y}(\relr)$, and 
for a 2-arrow \twoAr{\theta}{F}{G} in \RDtnl, the 2-arrow \twoAr{\RQFun(\theta) = \QR\theta}{\QR{F}}{\QR{G}} is given by 
$\QR\theta_{\ple{X\eqrelr}} = \theta_X$. 

\begin{proposition}\label[prop]{prop:qc-2fun}
$\RQFun$ is a well-defined 2-functor.
\end{proposition}
\begin{proof}
Just observe that 
\begin{itemize} 
\item if $\eqrelr$ is a $\RDoc$-equivalence on $X$ then $\lift{F}_{X,X}(\eqrelr)$ is a $\SDoc$-equivalence on $\fn{F}X$; 
\item if \fun{f}{\ple{X,\eqrelr}}{\ple{Y,\eqrels}} is an arrow in $\QC\RDoc$, then $\eqrelr \order \RDoc\reidx{f,f}(\eqrels)$ implies by naturality of $\lift{F}$ that $\lift{F}_{X,X}(\eqrelr)\order \SDoc\reidx{\fn{F}f,\fn{F}f}(\lift{F}_{Y,Y}(\eqrels))$; 
\item if $\relr\in\QR\RDoc(\ple{X,\eqrelr},\ple{Y,\eqrels})$, then we have 
$\lift{F}_{X,X}(\eqrelr)\rconv \rcomp \lift{F}_{X,Y}(\relr)\rcomp \lift{Y,Y}(\eqrels) 
   \order \lift{F}_{X,Y}(\eqrelr\rconv\rcomp\relr\rcomp\eqrels) 
   \order \lift{F}_{X,Y}(\relr)$, 
because $\relr$ is a descent datum, $\lift{F}$ is natural and laxly preserves relational composition and converse; 
\item if \oneAr{F}{\RDoc}{\SDoc} is a 1-arrow, then $\QR{F}$ preserves quotient arrows, because their underlying arrows are identities; 
\item if \twoAr{\theta}{F}{G} is a 2-arrow and \oneAr{F,G}{\RDoc}{\SDoc}, then $\lift{F}_{X,Y}(\relr) \order \SDoc\reidx{\theta_X,\theta_Y}(\lift{G}_{X,Y}(\relr))$  holds, hence in particular we have $\lift{F}_{X,X}(\eqrelr) \order \SDoc\reidx{\theta_X,\theta_X}(\lift{G}_{X,X}(\eqrels))$, which implies that \fun{\theta_X}{\ple{\fn{F}X,\lift{F}_{X,X}(\eqrelr)}}{\ple{\fn{G}X,\lift{G}_{X,X}(\eqrels)}} is an arrow in \QC\SDoc. 
\end{itemize}
The fact that $\RQFun$ preserves compositions and identities is straightforward. 
\end{proof}

We observe the following property of left adjoints in \RDtnl with respect to quotients, which is instrumental to the proof of our main result (\cref{thm:qc-univ}). 

\begin{lemma}\label[lem]{lem:ladj-quot}
Let \oneAr{F}{\RDoc}{\SDoc} be a left adjoint 1-arrow in \RDtnl. Then, $F$ preserves quotient arrows. 
\end{lemma}
\begin{proof}
Let \oneAr{G}{\SDoc}{\RDoc} be the right adjoint of $F$ in \RDtnl and \twoAr{\eta}{\id_\RDoc}{G\circ G} the unit of this adjunction. 
Let \fun{q}{X}{W} be a quotient arrow for the $\RDoc$-equivalence relation $\eqrelr$ on $X$. 
We have to prove that \fun{\fn{F}q}{\fn{F}X}{\fn{F}W} is a quotient arrow for the $\SDoc$-equivalence relation $\eqrels = \lift{F}_{X,X}(\eqrelr)$ on $\fn{F}X$. 
Consider an arrow \fun{f}{\fn{F}X}{Z} in the base of $\SDoc$ such that 
$\eqrels \order \gr{f}\rcomp\gr{f}\rconv$. 
Then, for every \fun{h}{\fn{F}W}{Z}, 
we have $f = h\circ \fn{F}Q$ if and only if $f^\sharp = h^\sharp \circ q$, where $f^\sharp$ and $h^\sharp$ are the transpose of $f$ and $h$, respectively, along the adjunction $\fn{F}\dashv \fn{G}$. 
By definition of transpose, we have $f^\sharp = \fn{G}f\circ\eta_X$, hence we get 
\begin{align*} 
\eqrelr
  & \order gr{\eta_X}\rcomp \lift{G}_{\fn{F}X,\fn{F}X}(\lift{F}_{X,X}(\eqrelr)) \rcomp \gr{\eta_X}\rconv 
    \order \gr{\eta_X}\rcomp \lift{G}_{\fn{F}X,\fn{F}X}(\gr{f}\rcomp\gr{f}\rconv)\rcomp\gr{\eta_X}\rconv 
\\
  & = \gr{\eta_X}\rcomp\gr{\fn{G}f}\rcomp\gr{\fn{G}f}\rconv\rcomp\gr{\eta_X}\rconv 
    = \gr{f^\sharp}\rcomp\gr{f^\sharp}\rconv 
\end{align*} 
Hence, since $q$ is a quotient arrow, we deduce that there is a unique arow \fun{g}{W}{\fn{G}Z} such that $f^\sharp = g\circ q$. 
Hence, we conclude that there is a unique \fun{h}{\fn{F}W}{Z} such that $f = h\circ\fn{F}q$, with $h = g^\sharp$, 
proving that $\fn{F}q$ is indeed a quotient arrow. 
\end{proof}

It is easy to see that the 1-arrow \oneAr{\QAr\RDoc}{\RDoc}{\QR\RDoc} in \RDtnl is the component of a lax natural transformation from the identity on \RDtnl to the composite $\QRFun\circ\RQFun$. 
The naturality square at \oneAr{F}{\RDoc}{\SDoc} is filled by a 2-arrow  \twoAr{\lambda}{\QAr\SDoc\circ F}{\QR{F}\circ\QAr\RDoc} where, for every object $X$ in the base of $\RDoc$, 
the component \fun{\lambda_X = \id_X}{\ple{\fn{F}X,\rid_{\fn{F}X}}}{\ple{\fn{F}X,\lift{F}_{X,X}(\rid_X)}} is a quotient arrow but not an identity because $F$ laxly preserves relational identities, hence we only have $\rid_{\fn{F}X}\order \lift{F}_{X,X}(\rid_X)$. 
Note that $\lambda$ is an identity, hence the naturality square strictly commutes, exactly when $F$ strictly preserves relational identities. 
Then, we can prove the following result. 

\begin{theorem}\label[thm]{thm:qc-univ}
The 2-functors $\RQFun$ and  $\QRFun$ are such that 
$\RQFun\dashv_l \QRFun$ is a lax 2-adjunction. 
That is, for every relational doctrine $\RDoc$ and every relational doctrine with quotients $\SDoc$, the functor 
\begin{equation}\label[eq]{eq:qc-adj}
 \fun{\QRFun(\blank)\circ\QAr\RDoc}{\Hom{\QRDtnl}{\QR\RDoc}{\SDoc}}{\Hom{\RDtnl}{\RDoc}{\QRFun(\SDoc)}} 
\end{equation}
has a reflection left adjoint. 
\end{theorem}
\begin{proof} 
Let $\RDoc$ be a relational doctrine and $\SDoc$ a relational doctrine with quotients. 
Let \oneAr{F}{\RDoc}{\SDoc} be a 1-arrow in \RDtnl and \oneAr{G}{\QR\RDoc}{\SDoc} a 1-arrow in \QRDtnl. 
We define a 1-arrow \oneAr{F'}{\QR\RDoc}{\SDoc} and a 2-arrow \twoAr{\eta}{F}{F'\circ\QAr\RDoc}. 
By \cref{lem:quot-adj}, there is a strict 1-arrow \oneAr{Q}{\QR\SDoc}{\SDoc} which is a reflection left adjoint of $\QAr\SDoc$ and such that, 
for every object \ple{X,\eqrelr} in \QC\SDoc, there is a quotient arrow \fun{\theta_{\ple{X,\eqrelr}}}{X}{Q\ple{X,\eqrelr}} such that, 
for every arrow \fun{f}{\ple{X,\eqrelr}}{\ple{Y,\eqrels}} in \QC\SDoc,  we have $\theta_{\ple{Y,\eqrels}}\circ f = \fn{Q}f \circ \theta_{\ple{X,\eqrelr}}$, and, 
for every $\relr\in\QR\SDoc(\ple{X,\eqrelr},\ple{Y,\eqrels})$, we have $\lift{Q}_{\ple{X,\eqrelr},\ple{Y,\eqrels}}(\relr) = \Ex^\SDoc\reidx{\theta_{\ple{X,\eqrelr}},\theta_{\ple{Y,\eqrels}}}(\relr)$. 
We set $F' = Q\circ \QR{F}$ and $\eta_X = \theta_{\QR{F}\ple{X,\rid_X}}$. 
It is not difficult to check that $F'$ and $\eta$ are well-defind, that is, $F'$ preserves quotient arrows (using \cref{lem:ladj-quot}) and $\eta$ is a natural transformation satisfying the condition of 2-arrows in \RDtnl. 
Consider a 2-arrow \twoAr{\phi}{F}{G\circ\QAr\RDoc} in \RDtnl. 
Then, we have to show there is a unique 2-arrow \twoAr{\psi}{F'}{G} in \QRDtnl such that $\phi = (\psi\QAr\RDoc)\eta$. 
For every object $X$ in the base of $\RDoc$, 
the component at $X$ of $\phi$ is an arrow \fun{\phi_X}{\fn{F}X}{\fn{G}\ple{X,\rid_X}} in the base of $\SDoc$ satisfying 
$\lift{F}_{X,X}(\relr) \order \SDoc\reidx{\phi_X,\phi_X}(\lift{G}_{\ple{X,\rid_X},\ple{X,\rid_X}}(\relr)$, for all $\relr\in\RDoc(X,X)$. 
We define \fun{\psi_\ple{X,\eqrelr}}{\fn{F'}\ple{X,\eqrelr}}{\fn{G}\ple{X,\eqrelr}} as the unique arrow making the following diagram commute: 
\[\xymatrix{
  \fn{F}X 
  \ar[d]_-{\phi_X}
  \ar[r]^-{\theta_\ple{X,\eqrelr}}
& \fn{F'}\ple{X,\eqrelr}
  \ar@{..>}[d]^-{\psi_\ple{X,\eqrelr}}
\\
  \fn{G}\ple{X,\rid_X}
  \ar[r]^-{\fn{G}\id_X}
& \fn{G}\ple{X,\eqrelr} 
}\]
where \fun{\id_X}{\ple{X,\rid_X}}{\ple{X,\eqrelr}} is a quotient arrow for $\eqrelr \in \QR\RDoc(\ple{X,\rid_X},\ple{X,\rid_X})$. 
Observe that $\fn{G}\id_X$ is a quotient arrow for $\lift{G}_{\ple{X,\rid_X},\ple{X,\rid_X}}(\eqrelr)$ as $G$ preserves quotients, hence we have 
$\lift{G}_{\ple{X,\rid_X},\ple{X,\rid_X}}(\eqrelr) = \SDoc\reidx{\fn{G}\id_X,\fn{G}\id_X}(\rid_{\fn{G}\ple{X,\eqrelr}}$, which implies 
$ \lift{F}_{X,X}(\eqrelr) 
  \order \SDoc\reidx{\phi_X,\phi_X}(\lift{G}_{\ple{X,\rid_X},\ple{X,\rid_X}}(\eqrelr)) 
  = \SDoc\reidx{\fn{G}\id_X\circ\phi_X,\fn{G}\id_X\circ\phi_X}(\rid_{\fn{G}\ple{X,\eqrelr}})$. 
Then, since $\theta_\ple{X,\eqrelr}$ is a quotient arrow for $\lift{F}_{X,X}(\eqrelr)$, we conclude that $\psi_X$ is well defined. 
Moreover, for every arrow \fun{f}{\ple{X,\eqrelr}}{\ple{Y,\eqrels}} in $\QC\RDoc$, the following cube commutes
\[\xymatrix{
& \fn{F}X
  \ar[rr]^-{f}
  \ar[dd]^-{\phi_X}
  \ar[ld]_-{\theta_\ple{X,\eqrelr}} 
& 
& \fn{F}Y 
  \ar[dd]^-{\phi_Y}
  \ar[ld]^-{\theta_\ple{Y,\eqrels}} 
\\
  \fn{F'}\ple{X,\eqrelr}
  \ar[rr]_-{\fn{F'}f} 
  \ar[dd]_-{\psi_\ple{X,\eqrelr}} 
& 
& \fn{F'}\ple{Y,\eqrels} 
  \ar[dd]^-{\psi_Y} 
\\ 
& \fn{G}\ple{X,\rid_X} 
  \ar[rr]_-{\fn{G}\QAr\RDoc f} 
  \ar[ld]^-{\fn{G}\id_X} 
&
& \fn{G}\ple{Y,\rid_Y}
  \ar[ld]^-{\fn{G}\id_Y} 
\\ 
  \fn{G}\ple{X,\eqrelr}
  \ar[rr]_-{\fn{G} f}
&
& \fn{G}\ple{Y,\eqrels}  
}\]
proving that \nt{\psi}{\fn{F'}}{\fn{G}} is a natural transformation. 
Furthermore, the defining diagram of $\psi_\ple{X,\eqrelr}$ ensures that, for every $\relr\in\QR\RDoc(\ple{X,\eqrelr},\ple{Y,\eqrels})$, we have 
\begin{align*} 
\lift{F'}_{\ple{X,\eqrelr},\ple{Y,\eqrels}}(\relr) 
  & = \Ex^\SDoc\reidx{\theta_\ple{X,\eqrelr},\theta_\ple{Y,\eqrels}}(\lift{F}_{X,Y}(\relr)) 
\\ 
  & \order \Ex^\SDoc\reidx{\theta_\ple{X,\eqrelr},\ple{Y,\eqrels}}(\SDoc\reidx{\phi_X,\phi_Y}(\lift{G}_{\ple{X,\rid_X},\ple{Y,\rid_Y}}(\relr))) 
\\
  & \order \SDoc\reidx{\psi_\ple{X,\eqrelr},\psi_\ple{Y,\eqrels}}(\Ex^\SDoc\reidx{\fn{G}\id_X,\fn{G}\id_Y}(\lift{G}_{\ple{X,\rid_X},\ple{Y,\rid_Y}}(\relr))) 
\\ 
  & \order \SDoc\reidx{\psi_\ple{X,\eqrelr},\ple{Y,\eqrels}}(\lift{G}_{\ple{X,\eqrelr},\ple{Y,\eqrels}}(\relr)) 
\end{align*}
showing that \twoAr{\psi}{F'}{G} is a well-defined 2-arrow in \QRDtnl. 

Finally, note that when $\eqrelr = \rid_X$ we have $\fn{G}\id_X = \id_{\fn{G}\ple{X,\rid_X}}$, by functoriality of $\fn{G}$,  and 
$\theta_\ple{X,\rid_X} = \eta_X$, thus proving that $\phi_X = \psi_\ple{X,\rid_X}\circ \eta_X$ as needed. 
In addition,  $\psi_X$ is unique with this property as $\eta_X$ is a quotient arrow.  

In order to check that this adjunction is a reflection, consider again the quotients preserving 1-arrow \oneAr{G}{\QR\RDoc}{\SDoc}. 
The counit at $G$ is a 2-arrow \twoAr{\epsilon^G}{Q\circ\QR{G\circ\QAr\RDoc}}{G}, whose component at \ple{X,\eqrelr} is the unique arrow making the following traingle commute: 
\[\xymatrix@C=6ex{
  \fn{G}\ple{X,\rid_X}
  \ar[r]^-{\theta_{\ple{\fn{G}\ple{X,\rid_X},\lift{G}_{\ple{X,\rid_X},\ple{X,\rid_X}}(\eqrelr)}}} 
  \ar[rd]_-{\fn{G}\id_X}
& Q\ple{\fn{G}\ple{X,\rid_X},\lift{G}_{\ple{X,\rid_X},\ple{X,\rid_X}}} 
  \ar[d]^-{\epsilon^G_\ple{X,\eqrelr}} 
\\
& \fn{G}\ple{X,\eqrelr} 
}\]
which is well defined by the universal property of the quotient arrow $\theta_\ple{\fn{G}\ple{X,\rid_X},\lift{G}_{\ple{X,\rid_X},\ple{X,\rid_X}}(\eqrelr)}$ for the $\SDoc$-equivalence $\lift{G}_{\ple{X,\rid_X},\ple{X,\rid_X}}(\eqrelr)$ on its domain. 
But, since $\fn{G}\id_X$ is a quotient arrow for the same equivalence relation, the arrow $\epsilon^G_\ple{X,\eqrelr}$ is actually an isomorphism, as needed. 
\end{proof}

\begin{remark} \label[rem]{rem:qc-counit-pseudo} 
Let $\SDoc$ be a relational doctrine with quotients. 
The component of the counit of the lax 2-adjunction at $\SDoc$ is obtained by transposing the identity 1-arrow on $\SDoc$. 
Hence, it is exactly the left adjoint \oneAr{Q}{\QR\SDoc}{\SDoc} of the 1-arrow \oneAr{\QAr\SDoc}{\SDoc}{\QR\SDoc} given by \cref{lem:quot-adj}. 
Moreover, since the adjunction between the hom-categories is a reflection, the counit of the lax 2-adjunction is actually a pseudo-natural transformation, that is, naturality squares commute up to  an invertible 2-arrow. 
This shows that the laxness of the adjunction  originates only from the laxness of the unit, which is due to the fact that 1-arrows in $\RDtnl$ laxly preserve relational identities. 
\end{remark}

\begin{remark} \label[rem]{rem:qc-counit-strict}
Let $\RDoc$ be a relational doctrine. 
The component of the counit at $\QR\RDoc$ is the 1-arrow 
\oneAr{\QMAr\RDoc}{\QR{\QR\RDoc}}{\QR\RDoc} where 
$\fn{\QMAr\RDoc}\ple{\ple{X,\eqrelr},\eqrels} = \ple{X,\eqrels}$ and $\fn{\QMAr\RDoc} f = f$ and 
$\lift{\QMAr\RDoc}$ is componentwise the identity. 
Then, it is easy to see that, for any 1-arrow \oneAr{F}{\RDoc}{\RDoc'} in \RDtnl, the naturality square at $\QR{F}$ strictly commutes, that is, 
the equality $\QR{F} \circ \QMAr\RDoc = \QMAr{\RDoc'}\circ \QR{\QR{F}}$ holds. 
\end{remark}

\begin{example}
Let $\RDoc$ be a relational doctrine with quotients and 
\oneAr{F}{\RDoc}{\RDoc} be a 1-arrow in \QRDtnl, that is, it preserves quotient arrows. 
Recall from \cref{ex:bisim} the doctrine $\Bisim{F}$ on the category $\CoAlg{\fn{F}}$ of $\fn{F}$-coalgebras, where relations between coalgebras are $F$-bisimulations.
It is easy to see that $\Bisim{F}$ has quotients. 
Indeed, a $\Bisim{F}$-equivalence relation $\eqrelr$ on a $\fn{F}$-coalgebra \ple{X,c} is an $F$-bisimulation which is also a $\RDoc$-equivalence relation on $X$. 
Since $\RDoc$ has quotients, $\eqrelr$ admits an effective descent quotient arrow \fun{q}{X}{W} in the base of $\RDoc$. 
To conclude, it suffices to endow $W$ with an $\fn{F}$-coalgebra structure, making $q$ an $\fn{F}$-coalgebra homomorphism. 
To this end, note that, since $\eqrelr$ is a $F$-bisimulation and $\fn{F}q$ is a quotient arrow for $\lift{F}_{X,X}(\eqrelr)$, we get 
$\eqrelr \order \gr{\fn{F}q\circ c}\rcomp \gr{\fn{F}q\circ c}\rconv$. 
Thus by the universal property of quotients, we get a unique arrow \fun{c_\eqrelr}{W}{\fn{F}W} making the following diagram commute:
\[\xymatrix{
X \ar[d]_-{c} \ar[r]^-{q} & W \ar@{..>}[d]^-{c_\eqrelr} \\
\fn{F}X \ar[r]^-{\fn{F}q} & \fn{F}W 
}\]
This shows that the doctrine of $F$ bisimulations inherits quotients, provided that $F$ preserves them. 
If however quotients are not available in $\RDoc$ and/or $F$ does not preserve them, 
we can use the intensional quotient completion to freely add them to $\Bisim{F}$. 
In this way, we get the doctrine $\QR{\Bisim{F}}$ whose base category has as objects triple \ple{X,c,\eqrelr} where \ple{X,c} is an $\fn{F}$-coalgebra and $\eqrelr$ is an $F$-bisimulation equivalence on it. 
Notice that, applying $\RQFun$ to the 1-arrow $F$, we get a 1-arrow \oneAr{\QR{F}}{\QR\RDoc}{\QR\RDoc}. 
Then, we can construct the doctrine $\Bisim{\QR{F}}$ of $\QR{F}$-bisimulations. 
It is easy to check that $\QR{\Bisim{F}}$ is isomorphic to $\Bisim{\QR{F}}$, that is, the costruction of coalgebras commutes with the quotient completion. 
\end{example}

The 2-adjunction of \cref{thm:qc-univ}, being lax, establishes a weak correspondence between \RDtnl and \QRDtnl: 
between their hom-categories there is neither an isomorphism, nor an equivalence, but just an adjunction. 
Moreover, the family of 1-arrows $\QAr\RDoc$ is only a lax natural transformation. 
As already noticed, this is essentially due to the fact that 1-arrows of \RDtnl and \QRDtnl laxly preserve relational operations, in particular, relational identities. 
Hence, a way to recover a stronger correspondence may be to restrict to strict 1-arrows. 

Denote by \QRDtn the 2-full 2-subcategory of \QRDtnl whose 1-arrows are strict. 
Then, it is easy to see that $\RQFun$ applies \RDtn into \QRDtn, obtaining the following result. 

\begin{theorem}\label[thm]{thm:qc-univ-strict}
The lax 2-adjunction $\RQFun \dashv_l \QRFun$ restricts to a (pseudo) 2-adjunction between \QRDtn and \RDtn. 
That is, for every relational doctrine $\RDoc$ and every relational doctrine with quotients $\SDoc$, the functor 
\[ 
 \fun{\QRFun(\blank)\circ\QAr\RDoc}{\Hom{\QRDtn}{\QR\RDoc}{\SDoc}}{\Hom{\RDtn}{\RDoc}{\QRFun(\SDoc)}} 
\] 
is an equivalence of categories. 
\end{theorem} 
\begin{proof}
By \cref{thm:qc-univ}, we know that the functor \fun{\QRFun(\blank)\circ\QAr\RDoc}{\Hom{\QRDtnl}{\\QR\RDoc}{\SDoc}}{\Hom{\RDtnl}{\RDoc}{\SDoc}}  has a reflection left adjoint 
sending a 1-arrow \oneAr{F}{\RDoc}{\SDoc} to a 1-arrow \oneAr{F'}{\QR\RDoc}{\SDoc}, which preserves quotients, 
and the counit of such an adjunction is an isomorphism. 
Moreover, we know that $F' = Q\circ\QR{F}$, where$Q$ is a strict reflection left adjoint of $\QAr\SDoc$, which exists by \cref{lem:quot-adj}. 
Hence, when $F$ is strict, $F'$ is strict as well, 
hence the adjunction restricts to the categories $\Hom{\QRDtn}{\QR\RDoc}{\SDoc}$ and $\Hom{\RDtn}{\RDoc}{\SDoc}$. 
Moreover, when $F$ is strict, the unit at $F$ of this adjunction is an isomorphism becuase 
the naturality square of $\QAr{}$ at $F$ strictly commutes, that is, $\QAr\SDoc\circ F = \QR{F}\circ\QAr\RDoc$, and 
$Q\circ\QAr\SDoc$ is isomorphic to the identity on $\SDoc$ as $Q$ is a reflection left adjoint. 
Therefore, we can conclude that  this adjunction is actually an equivalence. 
\end{proof}

Like any 2-adjunction, also the one in \cref{thm:qc-univ-strict} induces a 2-monad on \RDtn, which is actually a strict 2-monad. 
Indeed, the underlying functor is $\QMnd = \QRFun\circ\RQFun$, which is strict and maps a relational doctrine \RDoc to its quotient completion $\QR\RDoc$; 
the unit consists of the 1-arrows $\QAr\RDoc$, which form a strict 2-natural transformation; and 
the multiplication is given by the 1-arrows $\QMAr\RDoc$, which form a strict 2-natural transformation as noticed in \cref{rem:qc-counit-strict}. 
 
The 2-monad $\QMnd$ induces a 2-category $\psAlg\QMnd$ of pseudoalgebras and pseudomorphisms \cite{BlackwellKP89}. 
Our goal now is to compare the 2-category $\QRDtn$ to the 2-category $\psAlg\QMnd$. 
As usual, there is a comparison 2-functor mapping a relational doctrine with quotients $\SDoc$ to a pseudoalgebra on $\SDoc$ whose structure 1-arrow is the component at $\SDoc$ of the counit of the 2-adjunction in \cref{thm:qc-univ-strict}. 
Hence, as observed in \cref{rem:qc-counit-pseudo}, this structure map is a reflection left adjoint in $\RDtn$ of the 1-arrow 
\oneAr{\QAr\SDoc}{\SDoc}{\QR\SDoc}. 
This shows that relational doctrines with quotients correspond to quite special pseudoalgebras in $\psAlg\QMnd$. 
However, we will shortly prove hat actually all objects in $\psAlg\QMnd$ are actually of this kind, that is, 
all pseudoalgebras have a structure map which is a reflection left adjoint of the unit of the 2-monad. 
This result is achieved by proving that the 2-monad is \emph{lax idempotent} \cite{Kock95,KellyL97}. 
This amounts to show that, for every relational doctrine $\RDoc$ in \RDtn, 
there is a 2-arrow 
\[ \twoAr{\lambda^\RDoc}{\QMnd(\QAr\RDoc)}{\QAr{\QMnd(\RDoc)}} \] 
such that 
$\lambda^\RDoc$ is natural in $\RDoc$ and the following equations hold:
\begin{description}
\item [\kzd{2}] $\lambda^\RDoc\QAr\RDoc = \id_{\QMnd(\QAr\RDoc)\circ\QAr\RDoc}$; 
\item [\kzd{3}] $M^\RDoc\lambda^\RDoc = \id_{\id_\RDoc}$. 
\end{description} 

\begin{lemma}\label[lem]{lem:qc-mnd}
The 2-monad $\QMnd$ is lax idempotent. 
\end{lemma} 
\begin{proof} 
Let $\RDoc$ be a relational doctrine in $\RDtn$. 
The 2-arrow $\lambda^\RDoc$ is given by the family of arrows \fun{\lambda^\RDoc_{\ple{X,\eqrelr}}}{\ple{\ple{X,\rid_X},\eqrelr}}{\ple{\ple{X,\eqrelr},\eqrelr}} indexed by objects \ple{X,\eqrelr} in \QC\RDoc where $\lambda^\RDoc_{\ple{X,\eqrelr}} = \id_X$. 
It is not difficult to see that $\lambda^\RDoc$ is natural in $\RDoc$, mainly because 1-arrow in \RDtn strictly preserve relational identities. 
Hence, we only check conditions \kzd{2} and \kzd{3}.
\begin{description}
\item [\kzd{2}]
Let $X$ be an object in the base of $\RDoc$. 
Then, the component of $\lambda^\RDoc$ at $\fn{\QAr\RDoc}X = \ple{X,\rid_X}$, is 
\fun{\lambda^\RDoc_{\ple{X,\rid_X}} = \id_X}{\ple{\ple{X,\rid_X}, \rid_X}}{\ple{\ple{X,\rid_X}, \rid_X}}, which is the identity on \ple{\ple{X,\rid_X},\rid_X} in the base of $\QR{\QR\RDoc}$, as needed. 
\item [\kzd{3}]
Let \ple{X,\eqrelr} be an object in \QC\RDoc. 
The component of $\lambda^\RDoc$ at  \ple{X,\eqrelr} is the arrow 
\fun{\lambda^\RDoc_{\ple{X,\eqrelr}} = \id_X}{\ple{\ple{X,\rid_X},\eqrelr}}{\ple{\ple{X,\eqrelr},\eqrelr}}. 
Applying $\fn{\QMAr\RDoc}$, we get  the arrow 
\fun{\fn{\QMAr\RDoc}\lambda^\RDoc_{\ple{X,\eqrelr}}}{\ple{X,\eqrelr}}{\ple{X,\eqrelr}}, 
which is defined by the universal property of the quotient arrow \fun{\id_X}{\ple{X,\rid_X}}{\ple{X,\eqrelr}} in $\QC\RDoc$ as depicted in the following diagram: 
\[\xymatrix@C=5ex{
  \ple{X,\rid_X} 
  \ar[r]^-{\id_X} 
  \ar[d]_-{\id_X} 
& \ple{X,\eqrelr}
  \ar[d]^-{\id_{\ple{X,\eqrelr}}} 
\\
  \ple{X,\eqrelr}
  \ar@{..>}[r]_-{\QMAr\RDoc\lambda^\RDoc_{\ple{X,\eqrelr}}} 
& \ple{X,\eqrelr} 
}\]
Therefore, this arrow is the identity on \ple{X,\eqrelr}, as needed. 
\end{description}
\end{proof}

As proved in \cite{Kock95,KellyL97}, since $\QMnd$ is a lax idempotent monad, we know 
that any pseudoalgebra for $\QMnd$ is a reflection left adjoint in \RDtn of the unit $\QAr{}$ of $\QMnd$. 
Then, since left adjoints are unique up to isomorphism, we know that any relational doctrine carries \emph{at most one} pseudoalgebra structure for $\QMnd$, that is, the monad describes a property of relational doctrines rather than a structure on them. 
Moreover, by \cref{lem:quot-adj}, this also implies that the carrier of a pseudoalgebra for $\QMnd$ is actually a relational doctrine with quotients. 
The following result strengthens this observation proving that the pseudo 2-adjunction $\RQFun\dashv\QRFun$ of \cref{thm:qc-univ-strict} is 2-monadic. 

\begin{theorem}\label[thm]{thm:qc-mnd}
The 2-categories $\QRDtn$ and $\psAlg\QMnd$ are equivalent. 
\end{theorem}
\begin{proof}
We know that there is a comparison 2-functor \fun{\mathrm{K}}{\QRDtn}{\psAlg\QMnd} mapping a relational doctrine with quotients $\SDoc$ to the pseudoalgebra 
\oneAr{Q^\SDoc}{\QMnd(\SDoc)}{\SDoc} 
which is the component at $\SDoc$ of the counit of the pseudo 2-adjunction $\RQFun\dashv\QRFun$ of \cref{thm:qc-univ-strict}.  

By \cref{lem:qc-mnd} and results in \cite{Kock95,KellyL97}, we also know that every pseudoalgebra
\oneAr{L}{\QMnd(\RDoc)}{\RDoc} is a reflection left adjoint of $\QAr\RDoc$ in \RDtn and every pseudomorphism is also a morphism of adjunctions in \RDtn. 
Hence, if \oneAr{\ple{F,\phi}}{\ple{\RDoc_1,L_1}}{\ple{\RDoc_2,L_2}} is pseudomorphism of $\QMnd$-pseudoalgebras and $\eta_1$ and $\eta_2$ are the units the adjunctions $L_1\dashv\QAr{\RDoc_1}$ and $L_2\dashv\QAr{\RDoc_2}$, respectively, 
we have that the equality of 2-arrows $(\eta_2 \QMnd(F))(\QAr{\RDoc_2}\phi) = \QMnd(F)\eta_1$ holds in \RDtn. 
Then, if \oneAr{L}{\QMnd(\RDoc)}{\RDoc} is a pseudoalgebra and $\eta$ the unit of the adjunction $L\dashv\QAr\RDoc$, 
by \cref{lem:quot-adj}, we deduce that $\RDoc$ has quotients and $L$ provides a choice of them: 
for every $\RDoc$-equivalence $\eqrelr$ on an object $X$, the underlying arrow of \fun{\eta_{\ple{X,\eqrelr}}}{\ple{X,\eqrelr}}{\QAr\RDoc L\ple{X,\eqrelr}} is a quotient arrow of $\eqrelr$ in $\RDoc$. 
Therefore, since pseudomorphisms in $\psAlg{\QMnd}$ preserve the units (up to iso), we derive that they also preserve quotient arrows. 
This proves that the forgetful 2-functor from $\psAlg\QMnd$ to $\RDtn$ factors through the 2-category $\QRDtn$, providing us with a 2-functor 
\fun{\mathrm{F}}{\psAlg\QMnd}{\QRDtn}
that simply forgets the pseudoalgebra structure. 
It is easy to see that $\mathrm{F}\circ\mathrm{K} = \Id_\QRDtn$ and 
$\mathrm{K}\circ\mathrm{F}\cong\Id_{\psAlg\QMnd}$, because left adjoints are unique up to isomorphism. 
Hence, we conclude that $\mathrm{K}$ is an equivalence of 2-categories, as needed. 
\end{proof}


\section{Extensional equality}
\label{sect:ext-sep}

An important logical principle commonly assumed is the \emph{extensionality} of equality. 
Intuitively, it means that two functions $f$ and $g$ are equal exactly when their outputs coincide on equal inputs, that is, 
whenever $x = y$ we have $f(x) = g(y)$. 
This is the usual notion of equality for set-theoretic functions, however, if we move to more constructive settings such as Type Theory, it is not necessarily the case that extensionality holds. 
Relational doctrines are able to distinguish the two notions of equality of arrows.

\begin{definition} \label[def]{def:ext-eq}
Let \fun{\RDoc}{\bop\CC}{\Pos} be a relational doctrine and \fun{f,g}{X}{Y} two parallel arrows in \CC. 
We say that $f$ and $g$ are \emph{$\RDoc$-equal}, notation $f \exteq g$, if 
$\rid_X \order \RDoc\reidx{f,g}(\rid_Y)$. 
We say that $\RDoc$ is \emph{extensional} if 
for every $f,g$ in \CC, $f \exteq g$ implies $f = g$.  
\end{definition}
That is, $\RDoc$ is extensional if $\RDoc$-equality implies equality of arrows. The other implication always holds, therefore in an extensional relational doctrine  $f \exteq g$ if and only if $f = g$.

\begin{proposition}\label[prop]{prop:ext-eq}
Let \fun{\RDoc}{\bop\CC}{\Pos} be a relational doctrine and \fun{f,g}{X}{Y} two parallel arrows in \CC. 
Then, $f \exteq g$ iff $\gr{f} = \gr{g}$. 
\end{proposition}
\begin{proof} 
By \cref{prop:left-adj}, we have $\RDoc\reidx{f,g}(\rid_Y) = \gr{f}\rcomp\rid_Y\rcomp\gr{g}\rconv$. 
Then, we have 
$f \exteq g$ iff 
$\rid_X \order \gr{f}\rcomp\gr{g}\rconv$ iff 
$\gr{g} \order \gr{f}$ iff
$\gr{g} = \gr{f}$ (by \cref{prop:fun-ord}). 
\end{proof}

\cref{prop:ext-eq} with \cref{prop:left-adj} mean that  $\RDoc$-equal arrows cannot be distinguished by the logic of $\RDoc$ since they behave in the same way w.r.t. reindexing. 
Indeed given \fun{f,f'}{X}{A} and \fun{g,g'}{Y}{B} in the base, 
$f\exteq f'$ and $g\exteq g'$ imply $\RDoc\reidx{f,g} = \RDoc\reidx{f',g'}$.

From a quantitative or topological perspective, extensional equality is related to various notions of \emph{separation}. 
Take for example the doctrine $\QR{\VRel\RPos}$ over the category \QC{\VRel\RPos} of pseudometric spaces and non-expansive maps (\cf \refItem{ex:qcat}{vrel}). Functions 
\fun{f,g}{\ple{X,\eqrelr}}{\ple{Y,\eqrels}} are $\QR{\VRel\RPos}$-equal iff $\eqrels(f(x),g(x)) = 0$
, which implies $f = g$ exactly when \ple{Y,\eqrels} 
satisfies the identity of indiscernibles, \ie the axiom stating that $\eqrels(x,y) = 0$ implies $x = y$. 
This requirement turns a pseudometric space into a usual metric space
and forces a strong separation property: 
the topology associated with the metric space is Hausdorff. 

This observation shows that the intensional quotient completion does not preserve extensionality. 
Indeed the relational doctrine $\VRel\RPos$ on \Set is extensional, while $\QR{\VRel\RPos}$ is not as not all pseudometric spaces are separated. 
This is due to the fact that the intensional quotient completion changes equality, as it modifies identity relations, while the equality between arrows of the base category remains unchanged.

We now introduce a completion enforcing extensionality or, in quantitative terms,  separation. 
As for the intensional quotient completion, it is inspired by the extensional collapse of an elementary doctrines introduced in \cite{MaiettiME:quofcm}.

\begin{proposition}\label[prop]{prop:ext-eq-cong}
Let $\RDoc$ be a relational doctrine and \fun{f,f'}{X}{Y} and \fun{g,g'}{Y}{Z} are arrows in the base \CC. Then 
$f\exteq f'$ and $g\exteq g'$ imply $g\circ f \exteq g'\circ f'$. 
\end{proposition}
\begin{proof} 
$ \rid_X
  \order \RDoc\reidx{f,f'}(\rid_Y) 
  \order \RDoc\reidx{f,f'}(\RDoc\reidx{g,g'}(\rid_Z)) 
  = \RDoc\reidx{g\circ f,g'\circ f'}(\rid_Z)$ 
\end{proof} 

This proposition shows that $\exteq$ is a congruence on  \CC. Let \EC\RDoc be the quotient of \CC modulo $\exteq$, notably, 
objects are those of \CC and arrows are equivalence classes of arrows in \CC modulo $\exteq$, denoted by $\eqc{f}$. 
Define a functor \fun{\ER\RDoc}{\bop{\EC\RDoc}}{\Pos} by 
$\ER\RDoc(X,Y) = \RDoc(X,Y)$ and $\ER\RDoc\reidx{\eqc{f},\eqc{g}}(\relr) = \RDoc\reidx{f,g}(\relr)$. It is well-defined on arrows by \cref{prop:left-adj,prop:ext-eq}. 

\begin{lemma}\label[lem]{lem:ec-doc}
The functor \fun{\ER\RDoc}{\bop{\EC\RDoc}}{\Pos} together with relational operations of $\RDoc$ is an extensional relational doctrine. 
\end{lemma}
\begin{proof}
Immediate by definition of $\ER\RDoc$. 
\end{proof}

Taking terminology from \cite{MaiettiME:quofcm}, the doctrine $\ER\RDoc$ is the \emph{extensional collapse} of $\RDoc$. 
The following examples show some connections between the extensional collapse and notions of separation in metric and topological structures. 

\begin{example}\label[ex]{ex:sep-metric}
\begin{enumerate}
\item\label{ex:sep-metric:vrel} 
Let $\Qtl = \ple{\Car\Qtl,\qord,\qmul,\qone}$ be a commutative quantale. 
Recall from \refItem{ex:qcat}{vrel} that the category \QC{\VRel\Qtl} is the category of $\Qtl$-metric spaces and non-expansive maps. 
It is the base of the doctrine $\QR{\VRel\Qtl}$, whose identity relation is given by $\rid_{\ple{X,\eqrelr}} = \eqrelr$ for every $\Qtl$-metric space  \ple{X,\eqrelr}. 
A $\Qtl$-metric sapce \ple{X,\eqrelr} is \emph{separated} if $\qone \qord \eqrelr(x,y)$ implies $x = y$. 
Notice that a separated $\RPos$-metric space is the usual notion of metric space.
Denote by \SMet\Qtl the full subcategory of \QC{\VRel\Qtl} of separated $\Qtl$-metric spaces. 
Applying the extensional collapse to $\QR{\VRel\Qtl}$ we get $\ER{\QR{\VRel\Qtl}}$ where two arrows \fun{[f],[g]}{\ple{X,\eqrelr}}{\ple{Y,\eqrels}} of its base $\EC{\QR{\VRel\Qtl}}$ are equal when $\eqrelr(x,y) \qord \eqrels(f(x),g(y))$.
There is an obvious inclusion of \SMet\Qtl  into $\EC{\QR{\VRel\Qtl}}$ sending a non-expansive map \fun{f}{\ple{X,\eqrelr}}{\ple{Y,\eqrels}} between separated $\Qtl$-metric spaces to its equivalence class $\eqc{f}$ with respect to $\exteq$. 
Hence, this functor is obviously full. 
To check it is faithful, let \fun{f,g}{\ple{X,\eqrelr}}{\ple{Y,\eqrels}} be arrows in \SMet\Qtl such that $\eqc f = \eqc g$, then, by reflexivity of $\eqrelr$ and definition of $\exteq$, we have 
$\rone \qord \eqrelr(x,x) \qord \eqrels(f(x),g(x))$, for all $x \in X$. 
Since \ple{Y,\eqrels} is separated, we deduce 
$f(x) = g(x)$, for all $x \in X$, and so $f = g$. 
Finally, to verify essential surjectivity, recall that, for any $\Qtl$-metric space \ple{X,\eqrelr}, we can define  an equivalence relation on $X$ by $x\sim y$ iff $\qone\qord\eqrelr(x,y)$, for all $x,y \in X$, and 
the quotient space \ple{X/\sim,\eqrelr_\sim}, where $\eqrelr_\sim(\eqc x, \eqc y) = \eqrelr(x,y)$, is separated. 
The equivalence class of the projection map \fun{\eqc q}{\ple{X,\eqrelr}}{\ple{X/\sim,\eqrelr_\sim}}  is an isomorphism in $\EC{\QR{\VRel\Qtl}}$. 
Indeed, since the function \fun{q}{X}{X/\sim} is surjective, by the axiom of choice, it has a section \fun{s}{X/\sim}{X}, 
hence we have $\eqc x = \eqc{s(\eqc x)}$, which implies $\qone\qord \eqrelr(x,s(\eqc x))$, for all $x \in X$. 
Therefore, we get 
$\eqrelr_\sim(\eqc x, \eqc y) = \eqrelr(x,y) \qord \eqrelr(s(\eqc x),x)\qmul\eqrelr(x,y)\qmul\eqrelr(y,s(\eqc y)) \qord \eqrelr(s(\eqc x),s(\eqc y))$, 
proving that $s$ is non-expansive. 
We know that $\eqc{q}\circ \eqc{s} = \eqc{q\circ s} = \eqc{\id_{\ple{X/\sim,\eqrelr_\sim}}}$, hence, 
to conclude, we have only to check that $\eqc{s}\circ \eqc{q} = \eqc{\id_{\ple{X,\eqrelr}}}$, that follows since $x \sim s(\eqc x) = s(q(x))$. 
The equivalence between \SMet\Qtl and \EC{\QR{\VRel\Qtl}} provides evidence that extensionality is a good way to talk about separation in an abstract and point free setting.
\item\label{ex:sep-metric:vec} 
Recall from \refItem{ex:qcat}{vec} that 
the base \QC{\VecDoc} of the relational doctrine $\QR\VecDoc$ is the category of semi-normed real vector spaces and short linear maps: an object \ple{X,\eqrelr} is a real vector space $X$ with a lax monoidal and homogeneous pseudometric $\eqrelr$, which is equivalent to a semi-norm on $X$ given by $\|\vecx\| = \eqrelr(\veczero,\vecx)$. 
A semi-norm is a norm when $\|\vecx\| = 0 $ implies $\vecx=\veczero$, which is equivalent to $\eqrelr$ being separated. 
The category $\ct{NVec}$ of normed vector spaces is equivalent to the base category $\EC{\QR\VecDoc}$ of the extensional collapse of $\QR\VecDoc$. The proof of the essential surjectivity of the obvious inclusion of $\ct{NVec}$ into $\EC{\QR\VecDoc}$ uses arguments similar to those used in \refItem{ex:sep-metric}{vrel}. In particular it relies on the axiom of choice. There is only a little care in taking sections $\fun{s}{X/\sim}{X}$ of a quotient map $q$ in $\ct{Vec}$ as these have to be linear. But from a section $s$ one cane take its values on the vectors of a chosen base of $X/\sim$ and generate from this assignment a linear map $\fun{s'}{X/\sim}{X}$ which is easily proved to be a section of $q$.
\end{enumerate}
\end{example}

\begin{example}\label[ex]{ex:top} 
Let \Top be the category of topological spaces and continuous functions and \fun {\RTop}{\bop\Top}{\Pos} be the change-of-base $U^\star\Rel$ along the forgetful functor \fun{U}{\Top}{\Set} as in \refItem{ex:rel-doc}{cb}. 
The base \QC\RTop of the intensional quotient completion of $\RTop$ provides an ``intensional'' version of Scott's equilogical spaces\footnote{Applying the extensional collapse we get exactly the category of equilogical spaces.} \cite{ScottD:newcds}. 
Objects of \QC\RTop are pairs \ple{X,\eqrelr} of a topological space $X$ and an equivalence relation $\eqrelr$ on the underlying set of $X$ and arrows are continuous maps preserving the equivalences. 
Any section \fun{S}{\Top}{\QC\RTop} of the forgetful functor \fun{}{\QC\RTop}{\Top} picks an equivalence relation over every space in a way that relations are compatible with continuous maps.
The change-of-base $S^\star\QR\RTop$ provides a new logic on \Top where identity relations are changed according to $S$. 
For a space $X$, the doctrine $S^\star\RTop$ can not distinguish points which are related by $\eqrelr_X$, while such points may differ in the base. 
The extensional collapse makes such points indistinguishable in the base as well. Instances of this construction are the category $\ct{Top_0}$ of $T_0$-spaces and the homotopy category $\ct{hTop}$.
The former is given by defining $\eqrelr_X$ as follows: 
$\ple{x,y}\in\eqrelr_X$ iff $x$ and $y$ are topologically indistinguishable, that is, 
for every open subset $U\subseteq X$, $x\in U$ iff $y \in U$. 
The latter is given by defining $\eqrelr_X$ as follows: 
$\ple{x,y}\in\eqrelr_X$ iff there is a continuous path from $x$ to $y$, that is, 
there is a continuous function \fun{h}{[0,1]}{X} such that $h(0) = x$ and $h(1) = y$. 
\end{example}

The relational doctrine $\ER\RDoc$ comes together with a strict 1-arrow 
\oneAr{\EAr\RDoc}{\RDoc}{\ER\RDoc} where 
\fun{\fn{\EAr\RDoc}}{\CC}{\EC\RDoc} 
is the identity on objects and maps an arrow $f$ to its equivalence class $\eqc{f}$
 and 
$\lift{\EAr\RDoc}_{X,Y}$ is the identity on $\RDoc(X,Y)$. 
 
The extensional collapse is universal if we restrict to strict 1-arrows. 
Let $\ERDtn$ denote the full 2-subcategory of $\RDtn$ whose objects are extensional relational doctrines and  
\fun{\ERFun}{\ERDtn}{\RDtn} the obvious inclusion 2-functor.
The extensional collpase extends to a 2-functor \fun{\REFun}{\RDtn}{\ERDtn} defined as follows. 
For a relational doctrine $\RDoc$, we set $\REFun(\RDoc) = \ER\RDoc$. 
For a strict 1-arrow \oneAr{F}{\RDoc}{\SDoc} we define \oneAr{\REFun(F) = \ER{F}}{\ER\RDoc}{\ER\SDoc} where 
$\fn{\ER{F}}X = \fn{F}X$, $\fn{\ER{F}}\eqc{f} = \eqc{\fn{F}f}$ and $\lift{\ER{F}}_{X,Y} = \lift{F}_{X,Y}$, 
for all objects $X,Y$ and arrow \fun{f}{X}{Y} in the base of $\RDoc$. 

Finally, for a 2-arrow \twoAr{\theta}{F}{G}, we define a 2-arrow \twoAr{\REFun(\theta) = \ER\theta}{\ER{F}}{\ER{G}} as 
$\ER\theta_X = \eqc{\theta_X}$. 
Then, we can prove the following result. 

\begin{theorem}\label[thm]{thm:ec-univ}
The 2-functors $\ERFun$ and $\REFun$ are such that 
$\REFun\dashv\ERFun$ is a (strict) 2-adjunction.  
That is, for every relational doctrine $\RDoc$ and every extensional relational doctrine $\SDoc$, the functor 
\[
\fun{\blank\circ\EAr\RDoc}{\Hom{\ERDtn}{\ER\RDoc}{\SDoc}}{\Hom\RDtn{\RDoc}{\SDoc}}
\]
is an equivalence. 
\end{theorem}
\begin{proof}
Let $\RDoc$ be a relational doctirne and $\SDoc$ an extensional relaitonal doctrine. 
The functor is obviously fully faithful because $\fn{\EAr\RDoc}$ is full and the identity on objects and $\lift{\EAr\RDoc}$ is the identity. 
Towards a proof of essential surjectivity, let us consider a 1-arrow \oneAr{F}{\RDoc}{\SDoc} and define a 1-arrow \oneAr{F'}{\ER\RDoc}{\SDoc} as follows: 
$\fn{F'}X = \fn{F}X$, $\fn{F'}\eqc{f} = \fn{F}f$ and $\lift{F'}_{X,Y} = \lift{F}_{X,Y}$, for all objects $X,Y$ and arrow \fun{f}{X}{Y} in the base of $\RDoc$. 
The action of $\fn{F'}$ on arrow is well defined because $\SDoc$ is extensional $F$ is strict. 
Indeed, given arrows $\fun{f,g}{X}{Y}$ in the base of $\RDoc$ such that $\gr{f} = \gr{g}$, we have 
$\gr{\fn{F}f} = \lift{F}_{X,Y}(\gr{f}) = \lift{F}_{X,Y}(\gr{g}) = \gr{\fn{F}g}$, 
hence, by extensionality, we get $\fn{F}f = \fn{F}g$, as needed. 
Finally, the equality $F = F'\circ\EAr\RDoc$ is immediate by definition of $F'$, thus proving the thesis. 
\end{proof}

Like any (strict) 2-adjunction, the one in \cref{thm:ec-univ} induces a (strict) 2-monad $\EMnd = \ERFun\circ\REFun$ on $\RDtn$. 
Moreover, since the right 2-adjoint $\ERFun$ is fully faithful, the counit of this adjunction is an isomorphism. 
This implies that the 2-monad $\EMnd$ is idempotent, hence every pseudoalgebra for it is actually an equivalence of relational doctrines. 
Moreover, the following corollary immediately holds. 

\begin{corollary}\label[cor]{cor:ec-mnd}
The 2-adjunction $\REFun\dashv\ERFun$ is 2-monadic. 
That is, the 2-categories $\ERDtn$ and $\psAlg\EMnd$ are equivalent. 
\end{corollary}

\section{The extensional quotient completion} 
\label{sect:eqc} 

We now focus on the interaction between the extensional collapse and  quotients in relational doctrines. 
It is easy to see that if $\RDoc$ is a relational doctrine with quotients, then its extensional collapse $\ER\RDoc$ has quotients as well. 
More precisely, let us denote by $\EQRDtn$ the full 2-subcategory of $\QRDtn$ whose objects are extensional relational doctrines with quotients. 
We get two obvious inclusion 2-functors 
\fun{\EQQFun}{\EQRDtn}{\QRDtn} and \fun{\EQEFun}{\EQRDtn}{\ERDtn} which respectively forget extensionality and quotients.\footnote{$\EQRDtn$ can be seen as the pullback of $\ERFun$ against $\QRFun$. } 
Note that $\EQQFun$ is fully faithful, that is, the functors on hom-categories are isomorphisms, while $\EQEFun$ is only locally fully faithful, that is, the functors on hom-categories are only fully faithful. 
Then, we get the following result. 

\begin{theorem} \label[thm]{thm:qc-ec} 
The 2-adjunction $\REFun\dashv\ERFun$ restricts to a 2-adjunction $\QEQFun\dashv\EQQFun$ between $\EQRDtn$ and $\QRDtn$. 
\end{theorem}
\begin{proof}
Is suffices to show that if $\RDoc$ is a relational doctrine with quotients, then  $\ER\RDoc$ has quotients as well and quotient arrows in $\ER\RDoc$ are exactly equivalence classes of quotient arrows in $\RDoc$. 
To this end, consider a $\ER\RDoc$-equivalence relation $\eqrelr$ on $X$, then, since $\ER\RDoc(X,X) = \RDoc(X,X)$ by definition, 
$\eqrelr$ is also a $\RDoc$-equivalence relation on $X$. 
Let \fun{q}{X}{W} be a quotient arrow for $\eqrelr$ in $\RDoc$ and prove that $\eqc{q}$ is a quotient arrow for $\eqrelr$ in $\ER\RDoc$. 
Consider an arrow \fun{\eqc{f}}{X}{Y} such that $\eqrelr \order \gr{\eqc{f}}\rcomp\gr{\eqc{f}}\rconv$ holds in $\ER\RDoc(X,X)$. 
Then, since $\gr{\eqc{f}} = \gr{f}$, we have that $\eqrelr\order\gr{f}\rcomp\gr{f}\rconv$ holds in $\RDoc(X,X)$. 
By the universal property of $q$, we get a unique arrow \fun{h}{W}{Y} such that $f = h\circ q$, hence we get $\eqc{f} = \eqc{h}\circ\eqc{q}$. 
Finally, observe that for any $h'$ such that $\eqc{f} = \eqc{h'}\circ\eqc{q}$, we have 
$\gr{f} = \gr{q}\rcomp\gr{h'}$, which implies $\gr{h'} = \gr{q}\rconv\rcomp\gr{f}$, because $q$ is surjective. 
Hence, all such $h'$ are extensionally equal, proving that $\eqc{h}$ is unique and so $\eqc{q}$ is a quotient arrow as needed. 
\end{proof}

Again the associated 2-monad $\EQMndQ = \EQQFun\circ\QEQFun$ on $\QRDtn$ is an idempotent 2-monad and the adjunction is 2-monadic, since the right adjoint $\EQQFun$ is fully faithful. 
Moreover, it is easy to see that $\EQMndQ$ is a lifting of $\EMnd$ along the 2-functor \fun{\QRFun}{\QRDtn}{\RDtn}, that is the equality 
$\QRFun\circ\EQQFun = \EMnd\circ\QRFun$ holds. 

In summary, by \cref{thm:qc-univ-strict,thm:ec-univ,thm:qc-ec}, we get the following diagram 
\[\xymatrix@R=7ex@C=7ex{
\EQRDtn \ar@/_8pt/[r]_-{\EQQFun} \ar@{}[r]|{\bot} \ar[d]_-{\EQEFun} & \QRDtn \ar@/_8pt/[d]_-{\QRFun} \ar@{}[d]|{\vdash} \ar@/_8pt/[l]_-{\QEQFun} \\ 
\ERDtn  \ar@/^8pt/[r]^-{\ERFun}  \ar@{}[r]|{\top} & \RDtn \ar@/^8pt/[l]^-{\REFun} \ar@/_8pt/[u]_-{\RQFun} 
}\]
where the external square commutes and $\QEQFun$ is a lifting of $\REFun$, that is, $\EQEFun\circ\QEQFun = \REFun\circ\QRFun$. 
Note that, instead, the 2-functor $\RQFun$ does not restrict to a left biadjoint of $\EQEFun$, because the intensional quotient completion does not preserve extensionality.

Composing the 2-adjunctions  in the diagram above, we derive that 
the composite 2-functor $\REQFun = \QEQFun\circ\RQFun$ is a left biadjoint of the forgetful 2-functor $\EQRFun = \QRFun\circ\EQQFun$. 
It provides a universal construction that freely adds (effective descent) quotients and then forces extensionality. 
Furthermore, since both $\RQFun\dashv\QRFun$ and $\QEQFun\dashv\EQQFun$ are 2-monadic and the latter arises by lifting a 2-monadic 2-adjunction along a 2-monadic 2-functor, we obtain the following corollary. 

\begin{corollary}\label[cor]{cor:eqc-mnd}
The 2-adjunction $\REQFun\dashv\EQRFun$ is 2-monadic for the lax idempotent 2-monad $\EQMnd = \EMnd\circ\QMnd$. 
\end{corollary}

We dub this construction the \emph{extensional quotient completion} and we denote the resulting doctrine by $\EQR\RDoc$ and use a similar notation for 1- and 2-arrows. 

\begin{example}\label[ex]{ex:quot-sep-comp}
\begin{enumerate}
\item\label{ex:quot-sep-comp:sets} Recall from \cite{bishop1967foundations,bishop2012constructive} that a Bishop's set, or setoid, is a pair $\ple{A,\eqrelr}$ of a set $A$ and an equivalence relation $\eqrelr\subseteq A\times A$. A Bishop's function from the setoid $\ple{A,\eqrelr}$ to the setoid $\ple{B,\eqrels}$ is an equivalence class of functions $\fun{f}{A}{B}$ preserving the equivalence relations, where $f$ and $g$ belong to the same equivalence class if $f(a)\eqrels g(a)$ for all $a\in A$. A relation from $\ple{A,\eqrelr}$ to $\ple{B,\eqrels}$ is a subset $U\subset A\times B$ such that $(a,b)\in U$, $a\eqrelr a'$,  $b\eqrels b'$ imply  $(a',b')\in U$. Call $\BSet$ the category of Bishop's sets and functions and $\fun {\BRel}{\bop\BSet}{\Pos}$ the relational doctrine that maps two setoids to the collection of relations between them.
The relational doctrine $\ER{\QR{\Rel}}$ (obtained completing  $\fun {\Rel}{\bop\Set}{\Pos}$ first with quotients and then forcing extensionality) is $\BRel$.
\item\label{ex:quot-sep-comp:span} 
One of the most widely used constructions to complete a category with quotients is the exact completion of a weakly lex category presented in $\CC$ \cite{CarboniA:freecl,CarboniA:regec}. 
This is an instance of our constructions. Recall the relational doctrine $\Span\CC$ from \refItem{ex:rel-doc}{span}. Complete it first with quotients and then force extensionality. One get the relational doctrine $\ER{\QR{\Span\CC}}$ whose base is $\CC\exl$. If products of $\CC$ are strong, the construction coincides with the elementary quotient completion of the doctrines of weak subobjects of $\CC$ shown in \cite{MaiettiME:eleqc,MaiettiME:quofcm}. A comparison between these two constructions is in \cref{sect:eed}.
\end{enumerate}
\end{example}

\section{Projective objects and doctrines of algebras} 
\label{sect:proj} 

In this section we provide a characterization of the essential image of the extensional quotient completion. 
This is inspired and generalizes an analogous characterization for the exact completion of a category with (weak) finite limits \cite{CarboniA:freecl,CarboniA:regec} and another one for the elementary quotient completion of existential elementary doctrines \cite{MaiettiPR24}. 
Then, we will apply this result to show that doctrines of algebras for monads on extensional relaitonal doctrines with quotients can be obtained as the extensional quotient completion  of a suitable relational doctrine on (a subcategory of) free algebras for the same monad, provided that the original doctrine has enough structure. 
This again generalizes similar results proved for the exact completion \cite{Vitale94}. 
To achieve this, we have to introduce the notion of \emph{projective object} with respect to quotient arrows in a relational doctrine. 

\begin{definition}\label[def]{def:proj-obj}
Let \fun{\RDoc}{\bop\CC}{\Pos} be a relational doctrine with quotients. 
An object $P$ in \CC is \emph{$\RDoc$-projective} (with respect to quotient arrows)  if 
for every arrow \fun{f}{P}{Y} and every quotient arrow \fun{q}{X}{Y}  \fun{h}{P}{X} such that $f = q\circ h$, as in the following diagram 
\[\xymatrix{
& X \ar[d]^-{q} 
\\
  P \ar@{..>}[ru]^-{h}
    \ar[r]^-{f} 
& Y 
}\]
\end{definition}

Intuitively, a projective object $P$ is such that, whenever we have an arrow $f$ from $P$ into an object $Y$ which is a quotient of another object $X$ for an equivalence relation $\eqrelr$, 
for every ``element'' $p$ in $P$ we can \emph{choose} an ``element'' $h(p)$ in $X$ in the fibre over $f(p)$. 

\begin{example}\label[ex]{ex:proj-obj}
\begin{enumerate}
\item\label{ex:proj-obj:set}
In the doctrine $\Rel$ of set-theoretic relations, if we assume the Axiom of Choice, every object is projective, because every surjective function has a section. 
\item\label{ex:proj-obj:top}
Let \fun{\RTop}{\bop\Top}{\Pos} be the relational doctrine on the category of topological spaces and continuous functions from \cref{ex:top}. 
It is easy to see that $\RTop$ has quotients, given by the usual quotient space construction. 
Assuming the Axiom of Choice, we can show that discrete spaces are projective objects. 
Indeed, if $X$ is a discrete space, \fun{q}{Y}{Z} is a quotient arrow and \fun{f}{X}{Z} is a continuous function, then 
there is a section \fun{s}{Z}{Y} of $q$ in \Set and the function \fun{s\circ f}{X}{Y} is obviously continuous because $X$ is discrete. 
\item\label{ex:proj-obj:jspan}
Let \CC be an exact category and consider the relational doctrine $\JSpan\CC$ of jointly monic spans. 
This has quotients and quotient arrows are exactly regular epimorphisms. 
Then, an object $P$ in \CC is $\JSpan\CC$-projective if and only if it is regular projective. 
\end{enumerate}
\end{example}

A wide range of examples comes from the quotient completion, as the following proposition shows. 

\begin{proposition}\label[prop]{prop:proj-obj-qc}
Let $\RDoc$ be a relational doctrine and $\QR\RDoc$ its intensional quotient completion. 
An object \ple{X,\eqrelr} in $\QC\RDoc$ is $\QR\RDoc$-projective if and only if $\eqrelr = \rid_X$. 
\end{proposition}
\begin{proof}
Towards a proof of the right-to-left implication, 
consider an arrow \fun{f}{\ple{X,\rid_X}}{\ple{Z,\eqrels}} and a quotient arrow \fun{q}{\ple{Y,\eqrelr}}{\ple{Z,\eqrels}} for the $\QR\RDoc$-equivalence $\eqrels' = \gr{q}\rcomp\eqrels\rcomp\gr{q}\rconv$. 
Since \fun{\id_Y}{\ple{Y,\eqrelr}}{\ple{Y,\eqrels'}} is also a quotient arrow, we deduce that there is an isomorphism 
\fun{i}{\ple{Z,\eqrels}}{\ple{Y,\eqrels'}} such that $\id_Y = i\circ q$.
Consider the arrow \fun{h = i\circ f}{X}{Y}. 
We have $\rid_X \order \gr{h}\rcomp\gr{h}\rconv \order \gr{h}\rcomp\eqrelr\rcomp\gr{h}\rconv$, hence \fun{h}{\ple{X,\rid_X}}{\ple{Y,\eqrelr}} is an arrow in $\QC\RDoc$. 
Finally, we have 
$i\circ q \circ h = h = i\circ f$ and then, postcomposing with $i^{-1}$, we conclude 
$q\circ h = f$, as needed. 

To prove the other direction, 
let \ple{X,\eqrelr} be $\QR\RDoc$-projective. 
Sine \fun{\id_X}{\ple{X,\rid_X}}{\ple{X,\eqrelr}} is a quotient arrow, 
by projectivity we get an arrow \fun{f}{\ple{X,\eqrelr}}{\ple{X,\rid_X}} such that $\id_X\circ f = \id_{\ple{X,\eqrelr}}$. 
This implies that $f = \id_X$ and, since $\eqrelr \order \gr{f}\rcomp\gr{f}\rconv$ holds by definition of arrows in $\QC\RDoc$, we conclude that 
$\eqrelr \order \rid_X$ and so $\eqrelr = \rid_X$, as needed. 
\end{proof}

For the extensional quotient completion, we can only prove one of the two implications above, as stated below. 

\begin{corollary}\label[cor]{cor:proj-obj-eqc}
Let $\RDoc$ be a relational doctrine and $\EQR\RDoc$ its extensional quotient completion. 
The objects \ple{X,\rid_X} in $\QC\RDoc$ are $\EQR\RDoc$-projective. 
\end{corollary}
\begin{proof}
It follows from \cref{prop:proj-obj-qc}, observing that 
the 1-arrow \oneAr{\EAr{\QR\RDoc}}{\QR\RDoc}{\EQR\RDoc} is such that 
$\fn{\EAr{\QR\RDoc}}$ is full and $\lift{\EAr{\QR\RDoc}}$ is the identity. 
\end{proof}

Indeed, it is not difficult to find examples of extensional quotient completions with more projective objects. 
For instance, let $\RDoc$ be a relational doctrine and suppose that \fun{q}{X}{W} is a quotient arrow for a $\RDoc$-equivalence relation $\eqrelr$ on $X$ and suppose also that $q$ has a section \fun{s}{W}{X}. 
Then, the arrow \fun{\eqc{q}}{\ple{X,\eqrelr}}{\ple{W,\rid_W}} in the base of $\EQR\RDoc$ is an isomorphism, whose inverse is \fun{\eqc{s}}{\ple{W,\rid_W}}{\ple{X,\eqrelr}}. 
We obviously have $\eqc{q}\circ\eqc{s} = \eqc{\id_W}$ and, on the other hand, we have 
$ \gr{q}\rcomp\gr{s}\rcomp\eqrelr  
  = \gr{q}\rcomp\gr{s}\rcomp\gr{q}\rcomp\gr{q}\rconv 
  = \gr{q}\rcomp\gr{q}\rconv
  = \eqrelr $ 
proving that $\eqc{s}\circ\eqc{q} = \eqc{\id_X}$, as needed. 
More concretely, this shows, again assuming the Axiom of Choice, that all objects in the base of $\QR\Rel$ are projective. 

The following proposition provides a useful result: projective objects are preserved by left adjoints. 

\begin{proposition}\label[prop]{prop:proj-obj-ladj}
Let \oneAr{F}{\RDoc}{\SDoc} be a 1-arrow in \QRDtn such that $\fn{F}$ has a left adjoint $L$. 
Then, if $P$ is $\SDoc$-projective, $LP$ is $\RDoc$-projective. 
\end{proposition}
\begin{proof}
Consider an arrow \fun{f}{LP}{Y} and a quotient arrow \fun{q}{X}{Y} in the base of $\RDoc$. 
Since $F$ preserves quotients, we have that the arrow \fun{\fn{F}q}{\fn{F}X}{\fn{F}Y} is a quotient arrow in the base of $\SDoc$. 
Then, since $P$ is $\SDoc$-projective, we get an arrow \fun{h}{P}{\fn{F}X} such that the following diagram commutes
\[\xymatrix{
& \fn{F}X \ar[d]^-{\fn{F}q} 
\\
  P \ar[ru]^-{h} \ar[r]^-{f^\sharp} 
& \fn{F}Y 
}\]
where $f^\sharp$ is the transpose of $f$ along the adjunction $L \dashv\fn{F}$. 
Hence, transposing the above diagram we get the thesis. 
\end{proof}

For a relational doctrine $\RDoc$ with quotients, 
we denote by $\RProj\RDoc$ the restriction of a relational doctrine $\RDoc$ to the full subcategory $\Proj\RDoc$ of its base spanned by $\RDoc$-projective objects.  
The next definition will provide the central property for our characterization of the extensional quotient completion. 

\begin{definition}\label[def]{def:enough-proj}
Let \fun{\RDoc}{\bop\CC}{\Pos} be a relational doctrine with quotients. 
A full subcategory $\ct{G}$ of $\Proj\RDoc$ is said to be a \emph{$\RDoc$-projective cover} 
if, for every object $X$ in \CC, there is a quotient arrow \fun{q}{P}{X} where $P$ is an object of $\ct{G}$. 
We say that $\RDoc$ \emph{has enough projectives} if it has a $\RDoc$-projective cover. 
\end{definition}

In other words, a relational doctrine $\RDoc$ has enough projectives when every object can be presented as a quotient of a $\RDoc$-projective object. 
In a sense, if $\ct{G}$ is a $\RDoc$-projective cover, we know that objects of $\ct{G}$, that are $\RDoc$-projective, suffices to ``generate'' all objects in the base of $\RDoc$- 
This idea is made precise in the next theorem and the subsequent corollary. 
In the following, 
given a 1-arrow \oneAr{F}{\SDoc}{\RDoc} in \RDtn into an extensional relational doctrine with quotients, we denote by \oneAr{F^\sharp}{\EQR\SDoc}{\RDoc} its transpose along the biadjunction $\REQFun\dashv\EQRFun$, that is, 
$F^\sharp = Q^\RDoc\circ\EQR{F}$, where $Q^\RDoc$ is the component of the counit at $\RDoc$. 
Furthermore, we say that a 1-arrow $F$ is \emph{fully faithful} if $\fn{F}$ is fully faithful and $\lift{F}$ is an isomorphism. 

\begin{theorem}\label[thm]{thm:proj-obj}
Let \fun{\RDoc}{\bop\CC}{\Pos}  be an extensional relational doctrine with quotients and 
\oneAr{F}{\SDoc}{\RDoc} a fully-faithful  1-arrow in \RDtn. 
Then, the following are equivalent:
\begin{enumerate}
\item\label{thm:proj-obj:1} the 1-arrow \oneAr{F^\sharp}{\EQR\SDoc}{\RDoc} is an equivalence in \EQRDtn;
\item\label{thm:proj-obj:2} $F$ factors through $\RProj\RDoc$ and 
for every object $X$ in \CC, there is an object $P_X$ in the base of $\SDoc$ and a quotient arrow \fun{q_X}{\fn{F}P_X}{X} in \CC. 
\end{enumerate}
\end{theorem}
\begin{proof}
Let us first prove the implication from \cref{thm:proj-obj:1} to \cref{thm:proj-obj:2}. 
Let \oneAr{G}{\RDoc}{\EQR\SDoc} be the pseudoinverse of $F^\sharp$ in \EQRDtn, 
hence, by hypothesis, both $F^\sharp$ and $G$ preserve quotients. 
First we show that, for every object $A$ in the base of $\SDoc$, $\fn{F}A$ is $\RDoc$-projective, which immediately implies that $F$ factors through $\RProj\RDoc$. 
Consider a quotient arrow \fun{q}{X}{Y} and an arrow \fun{f}{\fn{F}A}{Y} in \CC.
Since $\fn{F}A \cong \fn{F^\sharp}\ple{A,\rid_A}$ and $G$ is a pseudoinverse of $F^\sharp$, we deduce that 
$\fn{G}\fn{F}A \cong \ple{A,\rid_A}$, hence by \cref{cor:proj-obj-eqc} we get that $\fn{G}\fn{F}A$ is $\EQR\SDoc$-projective. 
Since $G$ preserves quotients, we know that 
\fun{\fn{G}q}{\fn{G}X}{\fn{G}Y} is a quotient arrow, hence by projectivity of $\fn{G}\fn{F}A$, we get an arrow
as in the next commutative diagram:
\[\xymatrix{
& \fn{G}X \ar[d]^-{\fn{G}q}  \\
  \fn{G}\fn{F}A \ar[r]^-{\fn{G}f} \ar[ur]^-{\eqc{h}} 
& \fn{G}Y 
}\]
Finally, since $\fn{G}$ is fully-faiful we conclude that there is an arrow \fun{h}{\fn{F}A}{X} such that 
$f = q\circ h$, as needed. 
 
Consider now an object $X$ in \CC  and suppose that $\fn{G}X = \ple{P_X,\eqrelr_X}$. 
Let us consider the object $P = \fn{F^\sharp}\ple{P_X,\rid_{P_X}}$ in \CC. 
The arrow \fun{\id_{P_X}}{\ple{P_X,\rid_{P_X}}}{\ple{P_X,\eqrelr_X}} is a quotient arrow for the $\SDoc$-equivalence relation $\eqrelr_X$,  
thus the arrow \fun{\fn{F^\sharp}\id_{P_X}}{P}{\fn{F^\sharp}\ple{P_X,\eqrelr_X}} is a quotient arrow in $\RDoc$, because $F^\sharp$ preserves quotients. 
Therefore, because $\fn{F^\sharp}\ple{P_X,\eqrelr_X} = \fn{F^\sharp}\fn{G} X$ is isomorphic to $X$, as $F^\sharp$ is a pseudoinverse of $G$, and $\fn{F^\sharp}P_X = Q^\RDoc\ple{\fn{F}P_X,\rid_{\fn{F}P_X}}$ is isomorphic to $P$, 
we get the thesis. 

\medskip 

Towards a proof of the other direction, 
we build a pseudoinverse of $F^\sharp$. 
By hypothesis, 
for every object $X$ in \CC, we have an object $P_X$ in the base of $\SDoc$ and 
a quotient arrow \fun{p_X}{\fn{F}P_X}{X}, where $\fn{F}P_X$ is $\RDoc$-projective.
Moreover, since $\lift{F}$ is an isomorphism, 
we let $\eqrelr_X$ be the unique $\SDoc$-equivalence relation on $P_X$ such that 
$\lift{F}_{P_X,P_X}(\eqrelr_X) = \gr{p_X}\rcomp\gr{p_X}\rconv$. 
Consider now an arrow \fun{f}{X}{Y} in \CC. 
Since $\fn{F}P_X$ is $\RDoc$-projective and $\fn{F}$ is full, we get an arrow \fun{\hat{f}}{P:X}{P_Y} such that 
$p_Y\circ\fn{F}\hat{f} = f\circ p_X$. 
Moreover, we have that 
\begin{align*} 
\lift{F}_{P_X,P_X}(\eqrelr_X) 
  & = \gr{p_X}\rcomp\gr{p_X}\rconv 
    \order \gr{p_X}\rcomp\gr{f}\rcomp\gr{f}\rconv\rcomp\gr{p_X}\rconv 
    = \gr{\fn{F}\hat f}\rcomp\gr{p_Y}\rcomp\gr{p_Y}\rconv\rcomp\gr{\fn{F}\hat f}\rconv 
\\
  & = \lift{F}_{P_X,P_X}(\gr{\hat f}\rcomp\eqrelr_Y\rcomp\gr{\hat f}\rconv) 
\end{align*} 
Since $\lift{F}$ is an isomorphism, we deduce that 
$\eqrelr_X\order\gr{\hat f}\rcomp\eqrelr_Y\rcomp\gr{\hat f}\rconv$ holds, 
thus proving that \fun{\eqc{\hat f}}{\ple{P_X,\eqrelr_X}}{\ple{P_Y,\eqrelr_Y}} is a well defined arrow in the base of $\EQR\SDoc$. 
Furthermore, if $g,h$ are both arrows such that $p_Y\circ \fn{F}g = f\circ p_X$ and $p_Y\circ \fn{F}h = f\circ p_X$,  we have that 
\begin{align*} 
\lift{F}_{P_X,P_Y}(\gr{g}\rcomp\eqrelr_Y) 
  & = \gr{\fn{F}g}\rcomp\gr{p_Y}\rcomp\gr{p_Y}\rconv 
    = \gr{\fn{F}h}\rcomp\gr{p_Y}\rcomp\gr{p_Y}\rconv 
\\
  & = \lift{F}_{P_X,P_Y}(\gr{h}\rcomp\eqrelr_Y)
\end{align*} 
Again, since $\lift{F}$ is an isomorphism, we deduce that 
$\gr{g}\rcomp\eqrelr_Y = \gr{h}\rcomp\eqrelr_Y$, 
thus showing that $\eqc{g} = \eqc{h}$. 
Therefore, the arrow $\eqc{\hat f}$ is uniquely determined and so  the assignements 
$\fn{G}X = \ple{P_X,\eqrelr_X}$ and $\fn{G}f = \eqc{\hat f}$ determine a functor from \CC to the base of $\EQR\SDoc$. 
Furthre, given a relation $\relr\in\RDoc(X,Y)$, we denote by $\hat\relr$ the unique relation in $\SDoc(P_X,P_Y)$ such that $\lift{F}_{P_X,P_Y}(\hat\relr) = \gr{p_X}\rcomp\relr\rcomp\gr{p_Y}\rconv$, which again exists because $\lift{F}$ is an isomorphism. 
Then, we set $\lift{G}_{X,Y}(\relr) = \hat\relr$ and 
it is easy to check that \oneAr{G}{\RDoc}{\EQR\SDoc} is a well-defined 1-arrow in \RDtn. 
 
We check that $G$ preserves quotients. 
Let \fun{q}{X}{W} be a quotient arrow in $\RDoc$ for $\eqrelr$ and consider an arow \fun{\eqc{f}}{\ple{P_X,\eqrelr_X}}{\ple{Y,\eqrels}} such that 
$\hat\eqrelr \order \gr{f}\rcomp\eqrels\rcomp\gr{f}\rconv$. 
Then, we have the following diagram 
\[\xymatrix{
  \fn{F}P_X 
  \ar[rrrd]^-{\fn{F}f} 
  \ar[rd]_-{\fn{F}\hat q} 
  \ar[dd]_-{p_X} 
\\
& \fn{F}P_W 
  \ar@{..>}[rr]_-{\fn{F} h} 
  \ar[dd]^-{p_W} 
& 
& \fn{F}Y 
  \ar[dd]^-{p} 
\\ 
  X 
  \ar@{..>}[rrrd]^-{f'} 
  \ar[rd]_-{q} 
\\
& W 
  \ar@{..>}[rr]^-{h'} 
&
& Z 
}\]
where \fun{p}{\fn{F}Y}{Z} is a quotient arrow for $\lift{F}_{Y,Y}(\eqrels)$ in $\RDoc$, 
$f'$ and $h'$ are uniquely determined by the universal property of the quotient arrows $p_X$ and $q$, respectively, 
and $h$ is determined by projectivity of $\fn{F}P_W$ and fullness of $\fn{F}$. 
Note that the upper triangle does not commute. 
However, the above diagram uniquely determines the arrow \fun{\eqc{h}}{\ple{P_W,\eqrelr_W}}{\ple{Y,\eqrels}} in the base of $\EQR\SDoc$ and from 
\begin{align*} 
\lift{F}_{P_X,Y}(\gr{\hat q}\rcomp\gr{h}\rcomp\eqrels) 
  & = \gr{\fn{F}\hat q}\rcomp\gr{\fn{F}h}\rcomp\gr{p}\rcomp\gr{p}\rconv 
    = \gr{p_X}\rcomp\gr{q}\rcomp\gr{h'}\rcomp\gr{p}\rconv 
    = \gr{p_X}\rcomp\gr{f'}\rcomp\gr{p}\rconv
\\
  & = \gr{\fn{F}f}\rcomp\gr{p}\rcomp\gr{p}\rconv 
    = \lift{F}_{P_X,Y}(\gr{f}\rcomp\eqrels) 
\end{align*} 
we deduce that $\gr{\hat q}\rcomp\gr{h}\rcomp\eqrels = \gr{f}\eqrels$, because $\lift{F}$ is an isomorphism, 
thus proving that $\eqc{h}\circ\eqc{\hat q} = \eqc{f}$. 
Therefore, we conclude that $\eqc{\hat q}$ is a quotient arrow for $\hat\eqrelr$ and so $G$ preserves quotients. 

We conclude by showing that $F^\sharp$ and $G$ are pseudoinverse to each other. 
Recall that $F^\sharp = Q^\RDoc\circ\EQR{F}$, where 
\oneAr{Q^\RDoc}{\EQR\RDoc}{\RDoc} is the component at $\RDoc$ of the counit of the biadjunction $\REQFun\dashv\EQRFun$ of \cref{cor:eqc-mnd}. 
Hence, for every object \ple{X,\eqrelr} in the base of $\EQR\SDoc$, there is a quotient arrow \fun{q_\ple{X,\eqrelr}}{\fn{F}X}{\fn{Q^\RDoc}\ple{\fn{F}X,\lift{F}_{X,X}(\eqrelr)}}, where $\fn{Q^\RDoc}\ple{\fn{F}X,\lift{F}_{X,X}(\eqrelr)} = \fn{F^\sharp}\ple{X,\eqrelr}$, and, 
for every $\relr \in \EQR\SDoc(\ple{X,\eqrelr},\ple{Y,\eqrels})$, we have 
$\lift{F^\sharp}_{\ple{X,\eqrelr},\ple{Y,\eqrels}}(\relr) = \gr{q_\ple{X,\eqrelr}}\rconv\rcomp\lift{F}_{X,Y}(\relr)\rcomp\gr{q_\ple{Y,\eqrels}}$. 
We have an invertible 2-arrow \twoAr{\theta}{F^\sharp\circ G}{\Id_\RDoc} given by the following diagram, for every object $X$ 
\[\xymatrix{
  \fn{F}P_X 
  \ar[d]_-{q_\ple{P_X,\eqrelr_X}}
  \ar[rd]^-{p_X} 
\\
  \fn{F^\sharp}\ple{P_X,\eqrelr_X} 
  \ar@{..>}[r]^-{\theta_X}_-{\cong} 
& X 
}\]
where the arrow $\theta_X$ exists and is an isomorphism because both $p_X$ and $q_\ple{P_X,\eqrelr_X}$ are quotient arrows for $\lift{F}_{P_X,P_X}(\eqrelr_X)$. 
On the other hand, 
for an object \ple{X,\eqrelr} in the base of $\EQR\SDoc$, we have the arrows $f_\ple{X,\eqrelr}$ and $g_\ple{X,\eqrelr}$ as in the following diagram, where $W = \fn{F^\sharp}\ple{X,\eqrelr}$ 
\[\xymatrix{
  \fn{F}P_W 
  \ar[rd]_-{p_W}
  \ar@{..>}@/^5pt/[rr]^-{\fn{F}f_\ple{X,\eqrelr}} 
& 
& \fn{F}X 
  \ar[ld]^-{q_\ple{X,\eqrelr}} 
  \ar@{..>}@/^5pt/[ll]^-{\fn{F}g_\ple{X,\eqrelr}} 
\\
& W 
}\]
determined by the fact that both $\fn{F}X$ and $\fn{F}P_W$ are $\RDoc$-projective and $\fn{F}$ is full. 
As already observed, these data uniquely determine the arrows 
\fun{\eqc{f_\ple{X,\eqrelr}}}{\ple{P_W,\eqrelr_W}}{\ple{X,\eqrelr}} and \fun{\eqc{g_\ple{X,\eqrelr}}}{\ple{X,\eqrelr}}{\ple{P_W,\eqrelr_W}} in the base of $\EQR\SDoc$, thus they are inverse to each other. 
Therefore, the 2-arrow \twoAr{\phi}{G\circ F}{\ID_{\EQR{\RProj\RDoc}}} given by $\phi_\ple{X,\eqrelr} = \eqc{f_\ple{X,\eqrelr}}$ is well defined and invertible, as needed. 
\end{proof}

Recall from \refItem{ex:rel-doc}{cb} that, if \fun\RDoc{\bop\CC}{\Pos} is a relational doctrine and \fun{F}{\D}{\CC} is a functor, 
we denote by $F^\star\RDoc$ the change of base along $F$, that is, the functor $\RDoc\circ\bop{F}$. 
Furthermore, if \D is a full subcategory of \CC, we denote by $I_\D$ the associated inclusion functor. 

\begin{corollary}\label[cor]{cor:proj-obj}
Let \fun\RDoc{\bop\CC}{\Pos}  be an extensional relational doctrine with quotients and \ct{G} a full subcategory of \CC. 
Then, \ct{G} is an $\RDoc$-projective cover 
if and only if 
$\RDoc$ is equivalent to $\EQR{I_\ct{G}^\star\RDoc}$ in $\EQRDtn$.
\end{corollary} 
\begin{proof}
It follows from \cref{thm:proj-obj} by considering the obvious fully-faithful 1-arrow from $I_\ct{G}^\star\RDoc$ to $\RDoc$. 
\end{proof}

In other words, \cref{cor:proj-obj} ensures that 
relational doctrines obtained by the extensional quotient completion are, up to equivalence, extensional relational doctrines with quotients and enough projectives. 
Moreover, observe that, combining observations in \refItem{ex:quot-sep-comp}{span} and \refItem{ex:proj-obj}{jspan}, we recover as a special case of \cref{cor:proj-obj} the standard characterization of the exact completion of a category with weak finite limits.  

A key result about the exact completion shows that the category of algebras for a monad on an exact category satisfying a form of the Axiom of Choice is equivalent to the exact completion of the full subcategory of free algebras, that is, the Kleisli category \cite{Vitale94}. 
We conclude this section by showing how we can extend this result in the context of relational doctrines using \cref{cor:proj-obj}. 

Let \fun{\RDoc}{\bop\CC}{\Pos}  be an extensional relational doctrine with quotients and $\mnd = \ple{T,\eta,\mu}$ a monad on $\RDoc$ in $\EQRDtn$,  see \cite{Street72} for  the general 2-categorical definition and \cite{DagninoP24,DagninoR21} for the instantiation on (relational) doctrines. 
We can define the Eilenberg-Moore  doctrine \fun{\RDoc^\mnd}{\bop{\CC^{\fn\mnd}}}{\Pos} whose base is the Eilenberg-Moore category for the monad 
$\fn{\mnd} = \ple{\fn{T},\eta,\mu}$ on \CC and relations  from \ple{X,a} to \ple{Y,b} are relations $\relr\in\RDoc(X,Y)$ such that 
$\gr{a}\rconv\rcomp\lift{T}_{X,Y}(\relr)\rcomp\gr{b}\order \relr$, namely, relations that are closed under the operations of the algebras they relate. 
One can prove that $\RDoc^\mnd$ is extensional and has quotients (see \cite{DagninoP24}). 
For a full subcategory $\D$ of $\CC$, we denote by $\D_{\fn{\mnd}}$ the full subcategory of the Eilenberg-Moore category $\CC^{\fn\mnd}$ spanned by free algebras generated by objects in $\D$, i.e., algebras of the form \ple{\fn{T}X,\mu_X} for $X$ and object of \D. 
We also write \fun{K_\D}{\D_{\fn\mnd}}{\CC^{\fn\mnd}} for the inclusion functor.
Note that when \D coincides with \CC, the category $\D_{\fn\mnd}$ is the Kleisli category of the monad $\fn\mnd$. 

There is an obvious forgetful 1-arrow \oneAr{U}{\RDoc^\mnd}{\RDoc} in \EQRDtn which has a left  adjoint \oneAr{F}{\RDoc}{\RDoc^\mnd}, where $\fn{F}X = \ple{\fn{T}X,\mu_X}$ and $\lift{F}_{X,Y}(\relr) = \lift{T}_{X,Y}(\relr)$. 
The component of the counit at an algebra \ple{X,a} is \fun{\epsilon_\ple{X,a}}{\ple{\fn{T}X,\mu_X}}{\ple{X,a}} given by $\epsilon_\ple{X,a} = a$. 
The arrow $\fn{U}\epsilon_\ple{X,a} = a$ is a split epimorphism, hence a split $\RDoc$-surjection and so, by the next proposition, it is a quotient arrow in $\RDoc$. 

\begin{proposition}\label[prop]{prop:surj-quot-split}
Let $\RDoc$ be an extensional relational doctrine. 
An arrow is a split effective descent quotient arrow if and only if it is a split $\RDoc$-surjective arrow. 
\end{proposition}
\begin{proof}
The left-to-right implication is trivial. 
For the other direction, let \fun{q}{X}{Y} be a split $\RDoc$-surjective arrow with section \fun{s}{Y}{X}. 
Consider an arrow \fun{f}{X}{Z} such that $\gr{q}\rcomp\gr{q}\rconv \order \gr{f}\rcomp\gr{f}\rconv$, hence we get 
$\gr{q}\rcomp\gr{s}\rcomp\gr{f} \order \gr{f}$, which implies $\gr{q}\rcomp\gr{s}\rcomp\gr{f} = \gr{f}$ by \cref{prop:fun-ord}. 
Therefore, by extensionality, we derive $f\circ s \circ q = f$ and, since $q$ splits,  it is an epimorphism and so $f\circ s$ is unique, proving that $q$ is a quotient arrow. 
\end{proof} 

\begin{proposition}\label[prop]{prop:mnd-quot-reflect}
Let $\RDoc$ be an extensional relational doctrine with quotients and $\mnd$ a monad on it in \EQRDtn. 
If \fun{q}{\ple{X,a}}{\ple{Y,b}} is an algebra homomorphism such that \fun{q}{X}{Y} is a quotient arrow in $\RDoc$, 
then $q$ is a quotient arrow in $\RDoc^\mnd$ as well. 
\end{proposition}
\begin{proof}
Suppose $\mnd = \ple{T,\eta,\mu}$, hence $T$ preserves quotients. 
Then, we get the thesis by the following commutative diagram in the base of $\RDoc$. 
\[\xymatrix{
  \fn{T}X 
  \ar[rrd]^-{\fn{T}f} 
  \ar[rd]_-{\fn{T}q}
  \ar[dd]_-{a} 
\\ 
& \fn{T}Y 
  \ar[r]_-{\fn{T}h}
  \ar[dd]^-{b} 
& \fn{T}W 
  \ar[dd]^-{c} 
\\
  X 
  \ar[rrd]^-{f} 
  \ar[rd]_-{q}
\\
& Y 
  \ar[r]^-{h} 
& W 
}\]
where $h$ uniquely exists because $q$ is a quotient arrow in $\RDoc$. 
\end{proof}

In other words, \cref{prop:mnd-quot-reflect} shows that the forgetful 1-arrow \oneAr{U}{\RDoc^\mnd}{\RDoc} reflects quotient arrows. 
As a consequence, 
\cref{prop:surj-quot-split,prop:mnd-quot-reflect} ensure that the components of the counit \fun{\epsilon_\ple{X,a}}{\ple{\fn{T}X,\mu_X}}{\ple{X,a}} are quotient arrows in $\RDoc^\mnd$, thus showing that every algebra is a quotient of a free algebra. 
We are now ready to prove  our main result. 

\begin{theorem}\label[thm]{thm:eqc-mnd} 
Let \fun\RDoc{\bop\CC}{\Pos} be an extensional relational doctrine with quotients and $\ct{G}$ a subcategory of \CC. 
Then, the following are equivalent. 
\begin{enumerate}
\item\label{thm:eqc-mnd:1}
\ct{G} is an $\RDoc$-projective cover.
\item\label{thm:eqc-mnd:2}
For every monad $\mnd = \ple{T,\eta,\mu}$ on $\RDoc$ in \EQRDtn, $\ct{G}_{\fn\mnd}$ is an $\RDoc^\mnd$-projective cover.
\end{enumerate}
\end{theorem} 
\begin{proof}
The fact that \cref{thm:eqc-mnd:2} implies \cref{thm:eqc-mnd:1} is immediate by considering the identity monad on $\RDoc$. 
Towards a proof of the other direction, 
consider a monad $\mnd = \ple{T,\eta,\mu}$ on $\RDoc$ in \EQRDtn. 
Since every object in $\ct{G}_{\fn\mnd}$ is the image through a left adjoint of an object in \ct{G}, which is  an $\RDoc$-projective cover by hypothesis, 
by \cref{prop:proj-obj-ladj}, we get that all objects in $\ct{G}_{\fn{\mnd}}$ are $\RDoc^\mnd$-projective. 
Consider then an algebra \ple{X,a}. 
Since \ct{G} is an $\RDoc$-projective cover by hypothesis, we know there is a quotient arrow \fun{q_X}{P_X}{X} from an object $P_X$ in \ct{G}. 
Then, the composition 
$\xymatrix{\ple{\fn{T}P_X,\mu_{P_X}}\ar[r]^-{\fn{T}q} & \ple{\fn{T}X,\mu_X}\ar[r]^-{\epsilon_{X,a}} & \ple{X,a}}$
is a quotient arrow by \cref{prop:mnd-quot-reflect} and the fact that quotient arrows compose. 
\end{proof}

Combining \cref{cor:proj-obj,thm:eqc-mnd}, we deduce that 
the relational doctrine of algebras for a quotient preserving monad on   a relational doctrine with a projective cover is the extensional quotient completion  of its restriction to free algebras with projective generators. 
Since relational doctrines with a projective cover are exactly those obtained through the extensional quotient completion, 
this result applies to all quotient preserving monads  on such doctrines, thus providing a wide range of examples. 

The analogous result for the exact completion does not require the monad to preserve quotients. 
It assumes instead that the underlying category enjoys a form of the Axiom of Choice. 
We conclude this section by showing how we can achieve a similar result in our general context. 
A way of expressing the Axiom of Choice is by requiring that all surjection splits. 
The next proposition shows how this is related to balancing and projectivity.

\begin{proposition} \label[prop]{prop:proj-obj-choice}
Let \fun{\RDoc}{\bop\CC}{\Pos} be  a relational doctrine with quotients. 
Then, the following hold. 
\begin{enumerate}
\item\label{prop:proj-obj-choice:1}
All objects in \CC are $\RDoc$-projective if and only if every quotient arrow splits. 
\item\label{prop:proj-obj-choice:3}
If $\RDoc$ is extensional, then 
$\RDoc$ is balanced and all objects of \CC are $\RDoc$-projective if and only if every $\RDoc$-surjective arrow in \CC splits. 
\end{enumerate}
\end{proposition}
\begin{proof}
\cref{prop:proj-obj-choice:1}. 
Suppose all objects in \CC are $\RDoc$-projective and consider a quotient arrow \fun{q}{X}{W}. 
Since $W$ is $\RDoc$-projective, there is an arrow \fun{s}{W}{X} such that $q\circ s = \id_W$ as neded. 
On the other hand, if all quotient arrows split, given an arrow \fun{f}{X}{Z} and a quotient arrow \fun{q}{Y}{Z}, we know that $q$ has a section \fun{s}{Z}{Y} and so we have $q\circ s \circ f = f$, proving that $X$ is $\RDoc$-projective. 

\cref{prop:proj-obj-choice:3}. 
Immediate using \cref{prop:proj-obj-choice:1,prop:surj-quot-split,prop:quot-surj-balance}. 
\end{proof}

Therefore, we can finally prove the following stronger form of \cref{thm:eqc-mnd}. 

\begin{theorem}\label[thm]{thm:eqc-mnd-choice} 
Let \fun\RDoc{\bop\CC}{\Pos} be an extensional relational doctrine with quotients. 
Then, the following are equivalent
\begin{enumerate}
\item\label{thm:eqc-mnd-choice:1}
Every $\RDoc$-surjective arrow in \CC splits. 
\item\label{thm:eqc-mnd-choice:2}
For every monad $\mnd = \ple{T,\eta,\mu}$ on $\RDoc$ in $\RDtn$, $\RDoc^\mnd$ is a balanced and extensional relational doctrine with quotients and $\CC_{\fn\mnd}$ is a $\RDoc^\mnd$-projective cover.
\end{enumerate}
\end{theorem}
\begin{proof}
We first prove that \cref{thm:eqc-mnd-choice:1} implies \cref{thm:eqc-mnd-choice:2}. 
By \refItem{prop:proj-obj-choice}{3} we know that $\RDoc$ is balanced and all objects in its base are $\RDoc$-projective, hence \CC is a $\RDoc$-projective cover. 
Consider a monad $\mnd = \ple{T,\eta,\mu}$ on $\RDoc$ in \RDtn. 
We show that $T$ preserves quotient arrows, hence that $\mnd$ is a monad in \EQRDtn. 
Let \fun{q}{X}{W} be a quotient arrow, then, since it is $\RDoc$-surjective, by hypothesis we know that it has a section \fun{s}{W}{X}. 
Therefore, \fun{\fn{T}q}{\fn{T}X}{\fn{T}W} has a section \fun{\fn{T}s}{\fn{T}W}{\fn{T}X} and so it is a split surjection. 
Then, by \cref{prop:surj-quot-split} , we conclude that $\fn{T}q$ is a quotient arrow. 
Therefore, we deduce that $\RDoc^\mnd$ is an extensional relational doctrine with quotients and, by \cref{thm:eqc-mnd},  that $\CC_{\fn\mnd}$ is an $\RDoc^\mnd$-projective cover. 
Finally, the thesis follows by observing that, 
since $\RDoc$ is balanced, by \cref{prop:quot-surj-balance,prop:mnd-quot-reflect}, $\RDoc^\mnd$ is balanced as well. 

Towards a proof of the other direction, let us consider the identity monad on $\RDoc$, whose Eilenberg-Moore doctrine is $\RDoc$ itself. 
Then, from the hypothesis we derive that 
$\RDoc$ is balanced and all the objects in its base are $\RDoc$-projective, because they are the free algebras  for the identity monad. 
Therefore, the thesis follows by \refItem{prop:proj-obj-choice}{3}. 
\end{proof}

Notice that in the above theorem we have also proved that every monad in \RDtn on an extensional relational doctrine with quotients where surjections split is actually a monad in $\EQRDtn$, that is, its underlying 1-arrow preserves quotients. 

\cref{thm:eqc-mnd-choice} almost resembles the analogous result for the exact completion \cite{Vitale94}, but it is not a proper generalization. 
Indeed, \cref{thm:eqc-mnd-choice} applies to any relational doctrine and not just to that of jointly monic spans, which is the one behind the exact completion, 
however, it considers monads compatible with relational operations, while the result in \cite{Vitale94} applies to arbitrary monads (not necessarily preserving pullbacks and factorizations). 
Nevertheless, we can still recover a result for arbitrary monads on an exact category as the following example shows. 

\begin{example}\label[ex]{ex:eqc-mnd-exact} 
Let \CC be an exact category where regular epimorphisms split
and consider a monad $\mnd = \ple{T,\eta,\mu}$ on \CC. 
The Eilenberg-Moore category $\CC^\mnd$ is exact as well. 
The forgetful functor \fun{U}{\CC^\mnd}{\CC}, being a right adjoint, preserves monomorphisms. 
Moreover, it preserves coequalizers of equivalence relations. 
Hence, it extends to a 1-arrow \oneAr{\JSpan{U}}{\JSpan{\CC^\mnd}}{\JSpan\CC} in \EQRDtn, where $\fn{\JSpan{U}} = U$. 
The functor $U$ has a left adjoint \fun{F}{\CC}{\CC^\mnd} and the component at the algebra \ple{X,a} counit of this adjunction is the algebra homomorphism \fun{\epsilon_\ple{X,a}}{\ple{TX,\mu_X}}{\ple{X,a}} given by $\epsilon_\ple{X,a} = a$ is a coeqalizer and so a regular pimorphism, that is, a quotient arrow in $\JSpan{\CC^\mnd}$. 
Since every regular epimorphism in \CC splits, every object of \CC is $\JSpan\CC$-projective. 
Therefore, by \cref{prop:proj-obj-ladj}, we get that the Kleisli category $\CC_\mnd$ is a $\JSpan{\CC^\mnd}$-projective cover, 
hence, by \cref{cor:proj-obj}, 
we get that $\JSpan{\CC^\mnd}$ is equivalent to the extensional quotient completion of its restriction to $\CC_\mnd$. 
In other words, if \fun{K}{\CC_\mnd}{\CC^\mnd} is the inclusion functor of $\CC_\mnd$ into $\CC^\mnd$, we have that 
$\JSpan{\CC^\mnd}$ is equivalent to $\EQR{K^\star\JSpan{\CC^\mnd}}$. 
Finally, it is not difficult to observe that $K^\star\JSpan{\CC^\mnd}$ is equivalent to $\Span{\CC_\mnd}$, thus recovering the known result in \cite{Vitale94}. 
\end{example}


\section{Related Structures}
\label{sect:eed} 

There are many categorical models abstracting the essence of the calculus of relations, such as
cartesian bicategories \cite{CarboniW87} or allegories \cite{FreydS90} which are both special cases of ordered categories with involution \cite{Lambek99}. Also existential and elementary doctrines, \ie those doctrines that model $(\exists, \Land, \top, =)$-fragment of first order logic, encode a calculus of relations. A natural question is how relational doctrines differ from these models.

We show that when working with an ordered category, one implicitly accepts two logical principles, which are not necessarily there in a relational doctrine, and we show that when working with existential elementary doctrines, one implicitly accepts to work with variables, which are not necessarily there in relational doctrines. These comparison are carried out restricting to the 2-category $\RDtn$ where 1-arrows are strict.

\subsection{Ordered categories with involution.}
An ordered category with involution \cite{Lambek99}
is a \Pos-enriched category \CC together with an identity-on-objects and self inverse \Pos-functor \fun{(\blank)\rconv}{\CC\op}{\CC}. 
Intuitively, arrows can be seen as relations whose converse is given by the involution.

A relational doctrine \fun{\RDoc}{\bop\CC}{\Pos}
defines an ordered category with involution $\Ord\RDoc$ as follows: 
objects are those of \CC, the poset of arrows between $X$ and $Y$ is the fibre $\RDoc(X,Y)$, 
composition and identities are given by relational ones and the involution is given by the converse operation. The assignment extends to a 2-functor $\fun{\ROFun}{\RDtn}{\OCI}$, where $\OCI$ is the 2-category of ordered categories with involution 
whose 1-arrows \oneAr{F}{\CC}{\D} are ordered functors preserving involution and 
a 2-arrows \twoAr{\theta}{F}{G} are lax natural transformations.
 that is, 
$\theta_Y \circ F(f)  \leq G(f) \circ \theta_X$.
 
To see how to obtain a relational doctrine from an ordered category, first note that any ordered category with involution \CC induces 
a category $\Map\CC$, called the category of maps in \CC, whose  objects are those of \CC and an arrow \fun{f}{X}{Y} is an arrow in \CC such that \fun{f\rconv}{Y}{X} is its right adjoint, that is $f\circ f\rconv\order \id_Y$ and $\id_X \order f\rconv \circ f$. 
We define a relational doctrine 
\fun{\RMap\CC}{\bop{\Map\CC}}{\Pos} where  $\RMap\CC(X,Y) = \Hom\CC{X}{Y}$ is the poset of all arrows in \CC from $X$ to $Y$ and, 
for \fun{f}{A}{X} and \fun{g}{B}{Y} arrows in \Map\CC, the map \fun{\RMap\CC\reidx{f,g}}{\RMap\CC(X,Y)}{\RMap\CC(A,B)} sends $\relr$ to the composition $g\rconv \circ\ \relr \circ f$. 
Relational composition and identities are composition and identities of \CC and the relational converse is given by the involution $(\blank)\rconv$. 

The assignment easily extends to a fully faithful 2-functor $\fun{\ORFun}{\OCI}{\RDtn}$. 
Our goal is to characterize the essential image of $\ORFun$ by identifying some logical principles a relational doctrine has to satisfy in order to belong to it. 
These are extensionality and the rule of unique choice as defined below. 

\begin{definition}\label[def]{def:ruc}
Let \fun{\RDoc}{\bop\CC}{\Pos} be a relational doctrine. 
We say that $\RDoc$ satisfies the \emph{rule of unique choice}, \RUC for short, if, 
for every functional and total relation $\relr\in\RDoc(X,Y)$, there is an arrow \fun{f}{X}{Y} in \CC such that $\gr{f}\order\relr$. 
\end{definition}
Note that the inequality $\gr{f}\order\relr$ is actually an equality by \cref{prop:fun-ord}. 

Intuitively, this choice rule states that, whenever a relation $\relr$ from $X$ to $Y$ is functional and total, there is an arrow in the base that, for every element $x$ in $X$, picks the unique $y$ in $Y$ related to $x$ by $\relr$. 
We denote by $\ERUCRDtn$ the 2-full 2-subcategory of $\RDtn$ on those relational doctrines which are extensional and satisfy the rule of unique choice. 
The following proposition  shows that the 2-functor $\ORFun$ corestricts to $\ERUCRDtn$. 

\begin{proposition}\label[prop]{prop:oci-ruc-ext}
Let \CC be an ordered category with involution. 
Then, the relational doctrine $\RMap\CC$ is extensional and satisfies the rule of unique choice. 
\end{proposition}
\begin{proof}
Extensionality trivially holds because for any arrow \fun{f}{X}{Y} in $\Map\CC$ we have $\gr{f} = f$. 
The rule of unique choice trivially holds as well because funcional and total relations in $\RMap\CC$ are exactly arrows of $\Map\CC$. 
\end{proof}

The next theorem provides the characterization we are looking for, 
proving that ordered categories with involution correspond to extensional relational doctrines  satisfying the rule o unique choice. 

\begin{theorem}\label[thm]{thm:oci-rdtn-equiv} 
The 2-categories \OCI and \ERUCRDtn are 2-equivalent. 
\end{theorem}
\begin{proof}
The corestriction of $\ORFun$ to $\ERUCRDtn$ is fully-faithful. 
Hence, to conclude it suffices to show that it is essentially surjective. 
To this end, we show that 
if $\RDoc$ is an extensional relational doctrine satisfying \RUC, then it is isomorphic to $\RMap{\Ord\RDoc}$. 
Consider the 1-arrow \oneAr{G}{\RDoc}{\RMap{\Ord\RDoc}} defined by 
$\fn{G}X = X$, $\fn{G}f = \gr{f}$ and $\lift{G}_{X,Y}(\relr) = \relr$, which is well defined because arrows in the base of $\RMap{\Ord\RDoc}$ are functional and total relations in $\RDoc$ and $\RMap{\Ord\RDoc}(X,Y) = \Hom{\Ord\RDoc}{X}{Y} = \RDoc(X,Y)$. 
We have that 
$\fn{G}$ is faithful because $\RDoc$ is extensional and 
$\fn{G}$ is full because $\RDoc$ satisfies \RUC. 
Hence, because $\fn{G}$ is the identity on objects by definition, it is an isomorphism of categories. 
Then, the thesis follows because $\lift{G}$ is an isomorphism by definition. 
\end{proof}

The equivalence stated in \cref{thm:oci-rdtn-equiv} generalises a similar result proved in \cite{BonchiSSS21}, which compares cartesian bicategories and existential elementary doctrines. 
Examples of relational doctrines that are not in \OCI because they are not extensional were given in \cref{sect:ext-sep}. The following example presents a relational doctrine outside \OCI because it does not satisfy \RUC. 

\begin{example}\label[ex]{ex:rel-doc-non-oc}
Take a set $A$ with more than one element. The set $\PP(A)$ of subsets of $A$ is a complete Heyting algebra, therefore a commutative quantale. Recall from \refItem{ex:rel-doc}{vrel} that in the relational doctrine   \fun{\VRel{\PP(A)}}{\bop\Set}{\Pos} of
 $\PP(A)$-relations, for every set $X$ the relation $\rid_X$ maps $(x,x')$ to $A$ if $x=x'$ and to $\emptyset$ if $x\not=x'$. This relational doctrine does not satisfy \RUC. 
Consider $\relr\in \VRel{\PP(A)}(1,A)$ given by $\relr(*,a)=\{a\}$, it holds
\[\rid_1 = A=\bigcup_{a\in A}\{a\} = \relr\rcomp\relr\rconv\quad\text{and}\quad(\relr\rconv\rcomp\relr)(a,a') = \{a\}\cap \{a'\}\subseteq \rid_A\]
Suppose $\fun{f}{1}{A}$ is such that $\gr{f}\subseteq\relr$, \ie $\rid_A(f(*),a)\subseteq\relr(*,a)=\{a\}$. Then $A=\rid_A(f(*),f(*))\subseteq\{f(*)\}$, but this inclusion is contradictory with the assumption that $A$ has more than one element.
\end{example}

\subsection{Existential elementary doctrines.} 
Doctrines $\fun{\PDoc}{\CC\op}{\Pos}$ can be regarded as functorial representations of theories in (fragments of) first order predicate logic: 
objects and arrows of  $\ct{C}$ model contexts and terms, while fibres $\PDoc(X)$ collect  the predicates with free variables over  $X$ ordered by logical entailment. 
To sustain this intuition, the base category $\CC$ must have finite products to model context concatenation (see also \cite{PittsCL}).
In this setting, we can rephrase the well-known encoding of the calculus of relations into first order logic as a construction extracting 
a relational doctrine out of $\PDoc$, provided it has enough structure. 
Indeed, we can consider the functor \fun{\DRel\PDoc}{\bop\CC}{\Pos} mapping $\ple{X,Y}$ to $\PDoc(X\times Y)$ and $\ple{f,g}$ to $\PDoc\reidx{f\times g}$, that is, a relation from $X$ to $Y$ is just a binary predicate over $X\times Y$. 
In order to define relational composition and identities in the standard way, 
we have to consider doctrines modelling at least and the  $(\exists, \Land, \top, =)$-fragment of first order logic. 
These  are called elementary existential doctrines.
Let us recall from \cite{MaiettiME:eleqc,MaiettiME:quofcm} their definition. 
If \ct{C} is a category with finite products and $X,Y$ are two objects of \ct{C}, we write 
\fun{\fpr[X,Y]}{X\times Y}{X} and \fun{spr[X,Y]}{X\times Y}{Y} for the first and second projection, respectively, 
\fun{\diagar_X}{X}{X\times X} for the diagonal arrow, and 
\fun{\terar_X}{X}{1} for the unique arrow into the terminal object. 

\begin{definition}\label[def]{def:eed} 
A doctrine $\fun{\PDoc}{\CC\op}{\Pos}$ is existential elementary if all the following hold: 
\begin{itemize}
\item $\ct{C}$ has finite products; 
\item every fibre has finite meets and these are preserved by reindexing; 
\item for every $\fun{f}{X}{Y}$ in $\CC$ the reindexing $\PDoc\reidx{f}$ has a left adjoint $\fun{\Ex\reidx{f}}{\PDoc(X)}{\PDoc(Y)}$ such that for every $\phi\in \PDoc(X)$ and every $\psi\in\PDoc(Y)$ \emph{Frobenius reciprocity} holds 
\[ \Ex\reidx{f}(\phi)\wedge \psi = \Ex\reidx{f}(\phi\wedge \PDoc\reidx{f}\psi) \] 
\item for every arrow $\fun{f}{A}{B}$ in $\CC$ and every object $X$ in $\CC$ the \emph{Beck-Chevalley condition} holds 
\[ \PDoc\reidx{f}\Ex\reidx{\fpr[B,X]} = \Ex\reidx{\fpr[A,X]}\PDoc\reidx{f\times\id_X} \] 
\end{itemize}
\end{definition} 

\begin{example} \label[ex]{ex:eed} 
An archetypal example of existential elementary doctrine is the contravariant powerset functor  \fun{\PP}{\Set\op}{\Pos}. For a function $\fun{f}{X}{Y}$, the left adjoint $\Ex\reidx{f}$ is the direct image mapping. Two instances are of interest. The first is when  $f$ is the diagonal $\fun{\Delta_X}{X}{X\times X}$. In this case the direct image evaluated on the the top element (\ie the whole $X$) is the diagonal relation, that is $\Ex\reidx{\Delta_X}(\top_X)=\{(x,x')\in X\times X\mid x=x'\}$. The other is when $f$ is a projection $\fun{\pi_2}{X\times Y}{Y}$. In this case $\Ex\reidx{\pi_2}(\phi)=\{y\in Y\mid \exists_{x\in X}\, (x,y)\in \phi\}$.

\end{example}

The previous example shows the key idea that, in an existential elementary doctrine, left adjoints along diagonals compute diagonal relations, lefts adjoints along projections compute existential quantifications.

Then, for every existential elementary  doctrine 
 \fun{\PDoc}{\CC\op}{\Pos}, 
the functor 
\fun{\DRel\PDoc}{\bop\CC}{\Pos} 
defined above is a relational doctrine where 
\[
\rid_X = \Ex\reidx{\diagar_X}(\top)
\qquad 
\relr \rcomp \rels = \Ex\reidx{\ple{\fpr[X,Y,Z],\tpr[X,Y,Z]}} (\PDoc\reidx{\ple{\fpr[X,Y,Z],\spr[X,Y,Z]}}(\relr) \wedge \PDoc\reidx{\ple{\spr[X,Y,Z],\tpr[X,Y,Z]}}(\rels)) 
\qquad 
\relr\rconv = \PDoc\reidx{\ple{\spr[X,Y],\fpr[X,Y]}}(\relr)
\]
for $\relr\in\PDoc(X\times Y)$ and $\rels\in\PDoc(Y\times Z)$. 

This assignment easily extends to a 2-functor  \fun{\EEDRFun}{\EED}{\RDtn}  where  \EED denotes the 2-category whose objects are existential elementary doctrines, 1-arrows \oneAr{F}{\PDoc}{\QDoc} are pairs \ple{\fn{F},\lift{F}} where tha functor \fun{\fn{F}}{\CC}{\D} preserves finite products and 
\nt{\lift{F}}{\PDoc}{\QDoc\circ \fn{F}} preserves finite meets and commutes with left adjoints, and 
2-arrows \twoAr{\theta}{F}{G} are natural transformations \nt{\theta}{\fn{F}}{\fn{G}} such that $\lift{F}_X(\phi)\order\QDoc\reidx{\theta_X}(\lift{G}_X(\phi))$, for all $\phi \in \PDoc(X)$. 

\begin{example}\label[ex]{ex:eed-to-rel} 
Consider the powerset functor as an existential and elementary doctrine as in \cref{ex:eed}. It is immediate to see that $\DRel\PP$ is $\Rel$.
\end{example}

We devote the rest of this section to the characterization of the essential image of the 2-functor $\EEDRFun$. 
To achieve this, the key observation is the following: 
from a relational doctrine of the form $\DRel\PDoc$, we can recover $\PDoc$ mapping $A$ to $\DRel\PDoc(A,1)=\PDoc(A\times 1)\cong \PDoc(A)$. 

First of all note that existential elementary doctrines have finite products in the base, finite meets on all fibres preserved by reindexing, while relational doctrines need not have. 
Hence, our first step is to add this structure to relational doctrines, defining \emph{cartesian} relational doctrines. 
However, this is not as straightforward as one may imagine, because we have to come up with the right interaction laws between finite meets in the fibres and relational operations. 
To achieve this, we proceed in two steps: 
first we give a \emph{global} description of cartesian relational doctrines in terms of adjunctions in \RDtn and then we derive a local characterisation closer to the definition of existential elementary doctrines. 

We start by observing that the 2-category \RDtn has finite products. 
The terminal relational doctrine \fun{\TerDoc}{\bop\One}{\Pos} is defined on the terminal category $\One$, with a single object $\star$ and only the identity on it, and its only fibre $\TerDoc(\star,\star)$ is the singleton poset. 
Given relational doctrines \fun\RDoc{\bop\CC}{\Pos} and \fun\SDoc{\bop\D}{\Pos}, their product is the relational doctrine 
\fun{\RDoc\times\SDoc}{\bop{\CC\times\D}}{\Pos} whose fibres are $(\RDoc\times\SDoc)(\ple{X,A},\ple{Y,B}) = \RDoc(X,Y)\times\SDoc(A,B)$ and all relational operations are defined componentwise. 
Then, for every relational doctrine $\RDoc$, we can define in the obvious way a 
terminal 1-arrow \oneAr{\TerAr_\RDoc}{\RDoc}{\TerDoc} and 
a diagonal 1-arrow \oneAr{\DiagAr_\RDoc}{\RDoc}{\RDoc\times\RDoc}, 
obtaining the following definition. 

\begin{definition}\label[def]{def:rdoc-fin-prod} 
We say that a relational doctrine $\RDoc$ is \emph{cartesian} if 
the 1-arrows $\TerAr_\RDoc$ and $\DiagAr_\RDoc$ have right adjoints in \RDtn. 
\end{definition} 

Spelling out the definition, we have that \fun\RDoc{\bop\CC}{\Pos} is a relatioanl doctrine if and only if 
\begin{itemize}
\item \CC has finite products
\item for all objects $X,Y,A,B$ in \CC, we have a monotone function 
\fun\reltimes{\RDoc(X,Y)\times\RDoc(A,B)}{\RDoc(X\times A,Y\times B)}
natural in $X,Y,A,$ and $B$
\end{itemize}
and the following (in)equations holds, for all $\relr\in\RDoc(X,Y)$, $\rels\in\RDoc(Y,Z)$, $\relr'\in\RDoc(X',Y')$ and $\rels'\in\RDoc(Y',Z')$
\begin{align*} 
\rid_{X\times Y} &= \rid_X\reltimes\rid_Y  & 
(\relr\reltimes\relr')\rcomp(\rels\reltimes\rels') &= (\relr\rcomp\rels)\reltimes(\relr'\reltimes\rels') \\ 
(\relr\reltimes\relr')\rconv &= \relr\rconv\reltimes{\relr'}\rconv  & 
\relr &\order \gr{\terar_X}\rcomp\gr{\terar_Y}\rconv \\ 
\relr\reltimes\relr' &\order \gr{\fpr[X,X']}\rcomp\relr\rcomp\gr{\fpr[Y,Y']}\rconv  & 
\relr\reltimes\relr' &\order \gr{\spr[X,X']}\rcomp\relr'\rcomp\gr{\spr[Y,Y']}\rconv  \\ 
\relr &\order \gr{\diagar_X}\rcomp (\relr\reltimes\relr)\rcomp\gr{\diagar_Y}\rconv 
\end{align*} 
Note that these conditions together with naturality of $\reltimes$ imply that 
$\gr{f\times g} = \gr{f}\reltimes\gr{g}$. 

\begin{example}\label[ex]{ex:cartesian-rel} 
The doctrine $\DRel\PP=\fun{\Rel}{\bop\Set}{\Pos}$ is cartesian. The right adjoint to $\Delta_{\Rel}$ is given using products. Indeed for $\ple{\ple{A,B},\ple{X,Y}}$ the base of $\Rel\times \Rel$ the natural transformation $\Rel(A, B)\times \Rel(X,Y)\stackrel.\to\Rel(A\times X,B\times Y)$ maps $\relr\in \Rel(A,B)$ and $\rels\in\Rel(X,Y)$ to $\{\ple{\ple{a,x},\ple{b,y}}\mid \ple{a,b}\in \relr \text{ and } \ple{x,y}\in\rels\}$.
\end{example}

It is not difficult to see that, 
if $\RDoc$ is a cartesian relational doctrine, 
then every fibre of $\RDoc$ has finite meets given by 
\[
\relr\land \rels = \gr{\diagar_X}\rcomp(\relr\reltimes\rels)\rcomp\gr{\diagar_Y}\rconv 
\qquad 
\top = \gr{\terar_X}\rcomp\gr{\terar_Y}\rconv 
\]
where $\relr,\rels\in\RDoc(X,Y)$. 
Indeed, the following proposition holds. 

\begin{proposition}\label[prop]{prop:cartesian}
A relational doctrine \fun\RDoc{\bop\CC}{\Pos} is cartesian if and only if the following properties hold:
\begin{enumerate} 
\item\label{prop:cartesian:1}  \CC has finite products
\item\label{prop:cartesian:2}  every fibre of $\RDoc$ has finite meets and reindexing preserves them 
\item\label{prop:cartesian:3}  for all objects $X,Y$ in \CC, we have 
\[ \rid_1 = \top_1 \qquad \rid_{X\times Y} = (\gr{\fpr[X,Y]}\rcomp\gr{\fpr[X,Y]}\rconv)\land(\gr{\spr[X,Y]}\rcomp\gr{\spr[X,Y]}\rconv) \] 
\item\label{prop:cartesian:4}  for all relations $\relr\in\RDoc(A,X)$, $\rels\in\RDoc(X,B)$, $\relr'\in\RDoc(A,Y)$ and $\rels'\in\RDoc(Y,B)$ we have 
\[ (\relr\rcomp\rels)\land(\relr'\rcomp\rels') = ((\relr\rcomp\gr{\fpr[X,Y]}\rconv)\land(\relr'\rcomp\gr{\spr[X,Y]}\rconv)) \rcomp ((\gr{\fpr[X,Y]}\rcomp\rels)\land(\gr{\spr[X,Y]}\rcomp\rels')) \] 
\item\label{prop:cartesian:5}  for all relations $\relr,\rels\in\RDoc(X,Y)$ we have 
\[ (\relr\land_{X,Y}\rels)\rconv = \relr\rconv\land_{Y,X}\rels\rconv  \qquad \top_{X,Y}\rconv = \top_{Y,X} \] 
\end{enumerate} 
\end{proposition}
\begin{proof}
Suppose that $\RDoc$ is cartesian. 
All items follow easily from \cref{def:rdoc-fin-prod} and the definition of $\land$ and $\top$ given above. 
We focus only on \cref{prop:cartesian:4}, which is slightly less trivial. 
We have the following equations
\begin{align*}
(\relr\rcomp\rels)\land(\relr'\rcomp\rels') 
  &= \gr{\diagar_A}\rcomp((\relr\rcomp\rels)\reltimes(\relr'\rcomp\rels'))\rcomp\gr{\diagar_B}\rconv \\
  &= \gr{\diagar_A}\rcomp((\relr\reltimes\relr')\rcomp(\rels\reltimes\rels'))\rcomp\gr{\diagar_B}\rconv \\
  &= \gr{\diagar_A}\rcomp((\relr\rcomp\gr{\fpr}\rconv)\reltimes(\relr'\rcomp\gr{\spr}\rconv))\rcomp\gr{\diagar_{X\times Y}}\rconv\rcomp\gr{\diagar_{X\times Y}}\rcomp((\gr{\fpr}\rcomp\rels)\reltimes(\gr{\spr}\rcomp\rels'))\rcomp\gr{\diagar_B}\rconv \\ 
  &= ((\relr\rcomp\gr{\fpr}\rconv)\land(\relr'\rcomp\gr\spr))\rcomp((\gr\fpr\rcomp\rels)\land(\gr\spr\rcomp\rels')) 
\end{align*}
using the following commutative diagram
\[\xymatrix{
& X\times Y 
  \ar[ld]_-{\diagar_{X\times Y}}
  \ar[rd]^-{\diagar_X\times\diagar_Y} 
\\
  X\times Y \times X \times Y 
  \ar[rr]_-{\id_X\times\ple{\spr,\fpr}\times\id_Y}
  \ar[d]_-{\fpr[X,Y]\times\spr[X,Y]} 
&
& X\times X \times Y \times Y 
  \ar[d]^-{\fpr[X,X]\times\spr[Y,Y]}
\\
  X\times Y 
  \ar[rr]_-{\id_{X\times Y}} 
& 
& X\times Y 
}\]

Let us now suppose that $\RDoc$ satisfies all the items above and prove it is cartesian. 
For relations $\relr\in\RDoc(X,Y)$ and $\rels\in\RDoc(A,B)$ we define 
\[
\relr\reltimes\rels = (\gr{\fpr[X,A]}\rcomp\relr\rcomp\gr{\fpr[Y,B]}\rconv) \land (\gr{\spr[X,A]}\rcomp\rels\rcomp\gr{\spr[Y,B]}\rconv)  
\]
The only non-trivial equation is the compatibility between $\reltimes$ and relational composition. 
Indeed, we have two following equalities
\begin{align*}
(\relr_1\rcomp\relr_2)\reltimes(\rels_1\rcomp\rels_2)
  &= (\gr\fpr\rcomp\relr_1\rcomp\relr_2\rcomp\gr\fpr\rconv)\land(\gr\spr\rcomp\rels_1\rcomp\rels_2\rcomp\gr\spr\rconv) \\
  &= ((\gr\fpr\rcomp\relr_1\gr\fpr\rconv)\land(\gr\spr\rcomp\rels_1\rcomp\gr\spr\rconv))\rcomp((\gr\fpr\rcomp\relr_2\rcomp\gr\fpr\rconv)\land(\gr\spr\rcomp\rels_2\rcomp\gr\spr\rconv)) \\ 
  &= (\relr_1\reltimes\rels_1)\rcomp(\relr_2\reltimes\rels_2) 
\end{align*}
\end{proof}

\begin{corollary}\label[cor]{cor:cartesian}
Let $\RDoc$ be a cartesian relational doctrine. 
The following equalities hold: 
\[
\gr{\terar_X} = \top_{X,1} 
\qquad
\gr{\fpr[X,Y]}\rconv\rcomp\gr{\spr[X,Y]} = \top_{X,Y} 
\qquad 
\gr{\diagar_X} = \gr{\fpr[X,X]}\rconv\land\gr{\spr[X,X]}\rconv 
\]
\end{corollary} 

For a cartesian relational doctrine $\RDoc$ over the category \CC, consider the doctrine \fun{\RelD\RDoc}{\CC\op}{\Pos} 
obtained by precomposing $\RDoc$ with the functor \fun{\ple{-1}}{\CC\op}{\bop\CC}. 
By \cref{prop:cartesian:1,prop:cartesian:2} of \cref{prop:cartesian}, we know that \CC has finite products and $\RelD\RDoc$ factors through the category of meet-semilattices. 
Moreover, \cref{prop:left-adj} ensures that $\RelD\RDoc$ has left adjoints of all reindexing maps. 
Therefore, the missing ingredients are the Beck-Chevalley condition and Frobenius reciprocity. 
Let us start by the latter, considering its relational version. 

\begin{definition}\label[def]{def:frobenius}
Let \fun\RDoc{\bop\CC}{\Pos} be a relational doctrine. 
We say that $\RDoc$ satisfies \emph{Frobenius reciprocity} if, 
for every arrow \fun{f}{X}{Y} in \CC and every relation $\relr\in\RDoc{A,X}$ and $\rels \in \RDoc(A, Y)$, the following equality holds: 
\[ \relr\rcomp\gr{f} \land \rels = (\relr \land \rels\rcomp\gr{f}\rconv)\gr{f} \] 
\end{definition}

The nice fact is that Frobenius reciprocity implies the Beck-Chevalley condition in the relational setting. 

\begin{proposition}\label[prop]{prop:frobenius-bc}
Let $\RDoc$ be a cartesian relational doctrine satisfying Frobenius reciprocity. 
Then, the Beck-Chevalley condition holds: 
for every relation $\relr\in\RDoc(X,Y)$ we have 
\[ \relr\rcomp\gr{\fpr[Y,A]}\rconv = \gr{\fpr[X,A]}\rconv\rcomp (\relr\reltimes\rid_A) \] 
\end{proposition}
\begin{proof}
Using \cref{prop:cartesian,cor:cartesian} and Frobenius reciprocity, we have the following equalities: 
\begin{align*}
\gr{\fpr[XA]}\rconv \rcomp (\relr\reltimes\rid_A)  
  &= \gr{\fpr[X,A]}\rconv \rcomp ((\gr{\fpr[X,A]}\rcomp\relr\rcomp\gr{\fpr[Y,A]}\rconv) \land (\gr{\spr[X,A]}\rcomp\gr{\spr[Y,A]}\rconv) \\
  &= \relr\rcomp\gr{\fpr[Y,A]}\rconv \land (\gr{\fpr[X,A]}\rconv\rcomp \gr{\spr[X,A]}\rcomp\gr{\spr[Y,A]}\rconv) \\ 
  &= \relr\rcomp\gr{\fpr[Y,A]}\rconv \land \top\rcomp\gr{\spr[Y,A]}\rconv \\
  &= \relr\rcomp\gr{\fpr[Y,A]}\rconv 
\end{align*}
\end{proof}

Hence, in particular \cref{prop:frobenius-bc} implies that 
for every arrow \fun{f}{X}{Y} in the base of $\RDoc$ we have 
$ \gr{f} \rcomp \gr{\fpr[Y,A]}\rconv = \gr{\fpr[X,A]}\rconv \rcomp \gr{f\times \id_A} $. 
This shows that if $\RDoc$ is a cartesian relational doctrine satisfying Frobenius reciprocity, then 
$\RelD\RDoc$ is an existential elementary doctrine. 
However, this is not the end of the story because, in order to get a proper charaterization, we still have to show that the relational doctrine extracted from $\RelD\RDoc$ is isomorphic to $\RDoc$, that is, 
$\DRel{\RelD\RDoc} \cong \RDoc$ in an appropriate 2-category. 
In particular this means that we have to show that 
$\RDoc(X,Y)$ and $\RDoc(X\times Y,1)$ are naturally isomorphic  and all relational operations on $\RDoc$ can be defined through the logic of $\RelD\RDoc$. 
To achieve this, we  prove the following lemmas.

\begin{lemma}\label[lem]{lem:frob}
Let \fun\RDoc{\bop\CC}{\Pos}  be a cartesian relational doctrine satisfying Frobenius reciprocity. 
Then, for every object $X$ in \CC, the following equation holds:
\[ (\rid_X\reltimes\gr{\diagar_X})\rcomp(\gr{\diagar_X}\rconv\reltimes\rid_X) = \gr{\diagar_X}\rconv\rcomp\gr{\diagar_X} \]
\end{lemma}
\begin{proof}
Using \cref{prop:cartesian,cor:cartesian} and Frobenius reciprocity we derive 
\begin{align*}
(\rid_X\reltimes\gr{\diagar_X})\rcomp(\gr{\diagar_X}\rconv\reltimes\rid_X) 
  &= (\gr\fpr\rcomp\gr\fpr\rconv) \land (\gr\spr\rcomp\gr\fpr\rconv) \land (\gr\spr\rcomp\gr\spr\rconv) \\
  &= (\gr\fpr\rcomp\gr\fpr\rconv) \land (\gr\spr\rcomp(\gr\fpr\rconv\land\gr\spr\rconv) \\ 
  &= (\gr\fpr\rcomp\gr\fpr\rconv) \land (\gr\spr\rcomp\gr{\diagar_X}) \\ 
  &= ((\gr\fpr\rcomp\gr\fpr\rconv\rcomp\gr{\diagar_X}\rconv) \land \gr\spr)\rcomp\gr{\diagar_X} \\ 
  &= (\gr\fpr\land\gr\spr)\rcomp\gr{\diagar_X} \\ 
  &= \gr{\diagar_X}\rconv\rcomp\gr{\diagar_X} 
\end{align*}
\end{proof}

\begin{lemma}\label[lem]{lem:frob-iso}
Let $\RDoc$ be a cartesian relational doctrine satisfying Frobenius reciprocity. 
Then, the natural transformations 
\nt{\phi}{\RDoc}{\DRel{\RelD\RDoc}} and \nt{\psi}{\DRel{\RelD\RDoc}}{\RDoc} defined by 
\[\begin{split} 
\phi_{X,Y}(\relr) &= (\relr\reltimes\rid_Y)\rcomp\gr{\diagar_Y}\rconv\gr{\terar_Y} \\ 
\psi_{X,Y}(\rels) &= \gr{\fpr[X,Y]}\rconv\rcomp(\rid_X\reltimes\gr{\diagar_Y})\rcomp(\rels\reltimes\rid_Y)\rcomp\gr{\spr[1,Y]} 
\end{split}\]
are inverse to each other 
and preserve relational composition, identities and fibred finite meets. 
\end{lemma}
\begin{proof}
Using compatibility of $\reltimes$ with relational composition, \cref{lem:frob}, and the fact that $\spr[1,Y]\circ(\id_Y\times\terar_Y) = \spr[Y,Y]$, 
we derive that 
\[
\psi_{X,Y}(\phi_{X,Y}(\relr)) 
  = \gr{\fpr[X,Y]}\rconv\rcomp(\relr\reltimes\rid_Y)\rcomp\gr{\diagar_Y}\rconv\rcomp\gr{\diagar_Y}\rcomp\gr{\spr[Y,Y]} 
\]
Then, using \cref{prop:frobenius-bc}, we conclude 
\[
\psi_{X,Y}(\phi_{X,Y}(\relr)) 
  = \relr\rcomp\gr{\fpr[X,Y]}\rconv\rcomp\gr{\diagar_Y}\rconv\rcomp\gr{\diagar_Y}\rcomp\gr{\spr[Y,Y]} 
  = \relr 
\]
On the other hand, first observe that, by \cref{prop:frobenius-bc}, we have 
\[
(\gr{\spr[1,Y]}\reltimes\rid_Y)\rcomp\gr{\diagar_Y}\rconv 
  = \gr{\spr[1,Y\times Y]}\rcomp\gr{\diagar_Y}\rconv 
  = (\rid_1\reltimes\gr{\diagar_Y}\rconv)\rcomp\gr{\spr[1,Y]} 
\]
Then, using this fact and again 
compatibility of $\reltimes$ with relational composition and \cref{lem:frob}, 
we derive that 
\[
\phi_{X,Y}(\psi_{Y,Y}(\rels) ) 
  = (\gr{\fpr[X,Y]}\rconv\reltimes\rid_Y)\rcomp(\rid_X\reltimes\gr{\diagar_Y}\rconv)\rcomp(\rid_X\reltimes\gr{\diagar_Y})\rcomp(\rels\reltimes\rid_Y)\rcomp\gr{\spr[1,Y]} \rcomp \gr{\terar_Y} 
\]
Observing that 
$\terar_Y\circ\spr[1,Y] = \fpr[1,Y]$, 
$\fpr[X,Y]\times\id_Y = \id_X\times\spr[Y,Y]$, and 
$\fpr[X\times Y,Y] = \id_X\times\fpr[Y,Y]$, 
and using \cref{prop:frobenius-bc}, we conclude
\[
\phi_{X,Y}(\psi_{Y,Y}(\rels) ) 
  = (\rid_X\reltimes\gr{\spr[Y,Y]}\rconv)\rcomp(\rid_X\reltimes\gr{\diagar_Y}\rconv)\rcomp(\rid_X\reltimes\gr{\diagar_Y})\rcomp(\rid_X\reltimes\gr{\fpr[Y,Y]})\rcomp\rels 
  = \rels 
\]
To conclude the proof we have to check that $\phi$ preserves relational composition, identities and fibred finite meets. Analogous facts for $\psi$ will then follow because it is a natural inverse of $\phi$. 
\begin{itemize}
\item $\phi$ preserves relational identities because we have 
$\phi_{X,X}(\rid_X) = \gr{\diagar_X}\rconv\rcomp\gr{\terar_X} = \gr{\diagar_X}\rconv\rcomp\top_{X,1}$, 
which is the equality predicate in $\RelD\RDoc$ and so the identity relation in $\DRel{\RelD\RDoc}$. 
\item The fact that $\phi$ preserves relational composition follows from the equalities below, where we use \cref{prop:frobenius-bc,lem:frob}. 
\begin{align*}
\relr\rcomp\rels 
  &= \relr\rcomp\gr{\fpr}\rconv\rcomp\gr{\diagar_Y}\rconv\rcomp\gr{\diagar_Y}\rcomp\gr{\spr}\rcomp\rels \\ 
  &= \gr{\fpr}\rconv\rcomp(\relr\reltimes\rid_Y)\rcomp(\rid_Y\reltimes\gr{\diagar_Y}\rconv)\rcomp(\gr{\diagar_Y}\reltimes\rid_Y)\rcomp(\rid_Y\reltimes\rels)\rcomp\gr{\spr} \\ 
  &= \gr{\fpr}\rconv\rcomp(\rid_X\reltimes\gr{\diagar_Y})\rcomp(\relr\reltimes\rid_Y\reltimes\rels)\rcomp(\gr{\diagar_Y}\rconv\reltimes\rid_Y)\rcomp\gr{\spr} \\ 
  &= \gr{\fpr}\rconv\rcomp(\rid_X\reltimes\gr{\diagar_Y})\rcomp(\phi_{X,Y}(\relr) \reltimes \rels) \rcomp \gr{\spr} 
\end{align*}
Indeed, applying $\phi_{X,Z}$ to $\relr\rcomp\rels$ and using again \cref{prop:frobenius-bc} we get the following 
\begin{align*}
\phi_{X,Z}(\relr\rcomp\rels) 
  &= (\gr\fpr\rconv \reltimes \rid_Z)\rcomp(\rid_X\reltimes\gr{\diagar_Y}\reltimes\rid_Z)\rcomp(\phi_{X,Y}(\relr)\reltimes\phi_{Y,Z}(\rels))\rcomp\gr{\diagar_1}\rconv \\ 
  &= \gr{\ple{\fpr,\tpr}}\rconv\rcomp( (\gr{\ple{\fpr,\spr}}\rcomp\phi_{X,Y}(\relr))\land(\gr{\ple{\spr,\tpr}}\rcomp\phi_{Y,Z}(\rels)))
\end{align*}
which is exactly the composition of $\phi_{X,Y}(\relr)$ and $\phi_{Y,Z}(\rels)$ in $\DRel{\RelD\RDoc}$. 
\item To check that $\phi$ preserves fibred finite meets, it suffices to show that it preserves the relational product$\reltimes$. 
This follows by noting that 
\begin{align*}
\phi_{X\times A, Y\times B}(\relr\reltimes\rels) 
  &= (\relr\reltimes\rels\reltimes\rid_{Y\times B})\rcomp\gr{\diagar_{Y\times B}}\rconv\rcomp\gr{\terar_{Y\times B}} \\ 
  &= (\rid_X\reltimes\gr{\ple{\spr,\fpr}}\reltimes\rid_B)\rcomp(\relr\reltimes\rid_Y\reltimes\rels\reltimes\rid_B)\rcomp(\gr{\diagar_Y}\reltimes\gr{\diagar_B}) \rcomp (\gr{\terar_Y}\reltimes\gr{\terar_B})\rcomp\gr{\diagar_1}\rconv \\ 
  &= (\rid_X\reltimes\gr{\ple{\spr,\fpr}}\reltimes\rid_B) \rcomp (\phi_{X,Y}(\relr)\reltimes\phi_{A,B}(\rels)) \rcomp \gr{\diagar_1}\rconv 
\end{align*}
which is exactly the relational product of $\phi_{X,Y}(\relr)$ and $\phi_{A,B}(\rels)$. 
\end{itemize}
\end{proof}

The isomorphism in \cref{lem:frob-iso} generalizes and analogous  correspondence proved in \cite{BonchiSSS21} for cartesian bicategories. 
However, it is not enough to prove the characterization we are looking for, because we still have to show that it preserves the relational converse. 
To achieve this result, we need to strengthen Frobenius reciprocity. 
We take inspiration from allegories \cite{FreydS90} and give the following definition. 

\begin{definition}\label[def]{def:modular}
A cartesian relational doctrine $\RDoc$ is said \emph{modular} if the \emph{the modular law} holds: 
for all relations $\relr\in\RDoc(A,X)$, $\rels\in\RDoc(A,Y)$ and $\relt\in\RDoc(X,Y)$ we have 
\[ \relr\rcomp\relt \land \rels \order (\relr \land \rels\rcomp\relt\rconv)\rcomp\relt \] 
\end{definition}

It is immediate to see that the modular law implies Frobenius reciprocity: it suffices to take $\relt$ to be the graph of an arrow. 
The converse instead is not true, as the following example shows. 

\begin{example}\label[ex]{ex:not-modular}
Consider the set $H = \{00,01,10,11\}$ with the partial order  generated by 
$00\leq 01\leq 11$ and $00\leq 10\leq 11$. 
It is easy to see that this is a complete lattice (it is the product of the two-element lattice with itself). 
Consider also the monotone involution $(\blank)\rconv:H\to H$ defined by 
\[ 00\rconv = 00 \qquad 01\rconv = 10 \qquad 10\rconv = 01 \qquad 11\rconv = 11 \]
This poset can be regarded as a relational doctrine on the terminal category $\One$, 
where relational composition is given by meet $x\rcomp y = x\land y$, 
relational identity by top $\rid_1 = 11$ and 
the converse operation is given by the involution defined above. 
Using \cref{prop:cartesian} it is easy to see that this doctrine is cartesian. 
Moreover, it obviously satisfies Frobenius reciprocity, because the only arrow in the base is an identity. 
However, it does not satisfy the modular law: 
we have $11\land (01\rcomp 01) = 01$ while 
$((11\rcomp01\rconv)\land 01)\rcomp 01 = (10\land 01)\rcomp 01 = 00$

Note that, 
if we denote by $\RDoc$ the relational doctrine we have just described, 
then $\DRel{\RelD\RDoc}$ is not isomorphic to $\RDoc$ because the converse operation on $\DRel{\RelD\RDoc}$ is the identity. 
\end{example}

The next lemma shows a key property of cartesian and modular relational doctrine. 

\begin{lemma}\label[lem]{lem:mod-converse}
Let $\RDoc$ be a cartesian and modular relational doctrine. 
For every relation $\relr\in\RDoc(X,Y)$ the following equation holds: 
\[ (\relr\reltimes\rid_Y)\rcomp\gr{\diagar_Y}\rconv\rcomp\gr{\terar_Y} = (\rid_X\reltimes\relr\rconv)\rcomp\gr{\diagar_X}\rconv\rcomp\gr{\terar_X} \]
\end{lemma}
\begin{proof}
Using \cref{prop:cartesian} and the modular law, we have 
\begin{align*}
(\relr\reltimes\rid_Y)\rcomp\gr{\diagar_Y}\rconv\rcomp\gr{\terar_Y} 
  &= ((\gr{\fpr[X,Y]}\rcomp\relr)\land\gr{\spr[X,Y]})\rcomp\gr{\terar_Y} \\
  &\order (\gr{\fpr[X,Y]} \land (\gr{\spr[X,Y]}\rcomp\relr\rconv))\rcomp\relr\rcomp\gr{\terar_Y}  \\ 
  &\order (\gr{\fpr[X,Y]} \land (\gr{\spr[X,Y]}\rcomp\relr\rconv)) \rcomp \gr{\terar_X} \\ 
  &= (\rid_X\reltimes\relr\rconv)\rcomp\gr{\diagar_X}\rconv\rcomp\gr{\terar_X} 
\end{align*}
The other inequality follows in a similar way. 
\end{proof}

Let us denote by \MCRDtn the 2-full 2-subcategory of \RDtn whose objects are cartesian and modular relational doctrines and 1-arrows are those preserving the cartesian structure, that is, finite products in the base and finite meets in the fibres. 
As an immediate consequence of the following lemma, we get that the 2-functor $\EEDRFun$ corestricts to the 2-category \MCRDtn. 

\begin{lemma}\label[lem]{lem:drel-mod}
Let $\PDoc$ be an existential elementary doctrine. 
Then, $\DRel\PDoc$ is a cartesian and modular relational doctrine. 
\end{lemma}
\begin{proof}
We only have to check the modular law, but this is the case because it can be proven in regular logic \cite{FreydS90}. 
\end{proof}

We are finally able to prove the following theorem. 

\begin{theorem}\label[thm]{thm:cara} 
The 2-categories \MCRDtn and \EED are 2-equivalent. 
\end{theorem}
\begin{proof}
By \cref{lem:drel-mod}, we know that the 2-functor $\EEDRFun$ has type $\EED\to\MCRDtn$. 
It is also easy to see that it is fully faithful. 
Consider now a cartesian and modular relational doctrine $\RDoc$. 
Since every cartesian and modular relational doctrine satisfies Frobenius reciprocity, by \cref{lem:frob-iso}, we have a natural isomorphism 
\nt{\phi}{\RDoc}{\DRel{\RelD\RDoc}}, 
which preserves relational composition, identities and fibred finite meets. 
Then, to conclude, it suffices to show that $\phi$ preserves the relational converse. 
Indeed, by \cref{lem:mod-converse} we have 
\begin{align*} 
\phi_{Y,X}(\relr\rconv) 
  &= (\relr\reltimes\rid_X)\rcomp\gr{\diagar_X}\rconv\rcomp\gr{\terar_X} 
   = (\rid_Y\reltimes\relr)\rcomp\gr{\diagar_Y}\rconv\rcomp\gr{\terar_Y}   \\ 
  &= \gr{\ple{\spr,\fpr}}\rcomp(\relr\reltimes\rid_Y)\rcomp\gr{\diagar_Y}\rconv\rcomp\gr{\terar_Y} 
   = \gr{\ple{\spr,\fpr}}\rcomp\phi_{X,Y}(\relr) 
\end{align*} 
\end{proof}

It is not difficult to see that both the intensional quotient completion and the extensional collapse (hence, also the extensional quotient completion) of relational doctrines preserve the cartesian modular structure. 
More precisely, the 2-monades $\QMnd$, $\EMnd$, and $\EQMnd$ restrict to the 2-category \MCRDtn. 
Moreover, 
relying on the 2-equivalence of \cref{thm:cara}, 
we can observe that 
the completion of an existential elementary doctrine with quotients, introduced in \cite{MaiettiME:quofcm}, is equivalent to the intensional quotient completion of the corresponding (cartesian and modular) relational doctrine;
similarly, 
. The elementary quotient completion of an existential elementary doctrine, introduced in \cite{MaiettiME:eleqc}, is equivalent to the extensional quotient completion of the corresponding (cartesian and modular) relational doctrine. 
This results in a wide range of examples of relational doctrines and their completions such as realisability doctrines, doctrines of (strong/weak) subobjects and syntactic doctrines \cite{JacobsB:catltt, PittsCL,OostenJ:reaait}. Also dependent Types Theories give rise to existential elementary doctrines whose elementary quotient completion is the category of setoids  \cite{MaiettiME:eleqc,MaiettiME:quofcm}.

In summary, 
we have observed that cartesian and modular relational doctrines and existential elementary doctrines are equivalent and the completions presented in this paper, when restricted to this class of relational doctrines, coincide with those introduced by Maietti and Rosolini. 
More precisely, we also note hat
both  of them work on larger classes of doctrines. 
Indeed, the completions proposed by Maietti and Rosolini can be applied to doctrines that need not be existential in the sense that they need not have left adjoints to all the reindexing maps, but they need finite products in the base and finite meets in the fibres.
On the other hand, relational doctrines intrinsically have left adjoints to all reindexing maps, but the completions described in this paper work also on relational doctrines that need not be cartesian or modular.
In particular, relational doctrines are not require to have products in the base category or to satisfy the the modular law, which is a crucial fact to cover the quantitative examples as well as to propertly states the results involving projective objects (e.g., \cref{thm:proj-obj,thm:eqc-mnd}), because projective objects are rairly closed under products. 

Finally, let us notice that we can consider relational doctrines of the form $\DRel\PDoc$ also when 
$\PDoc$ is either an existential elementary linear doctrine \cite{DagninoP22} or an existential biased  elementary doctrine\footnote{At least when we have a choice of weak finite products.} \cite{Cioffo24}. 
For instance, relational doctrines of the form $\VRel\Qtl$ of \refItem{ex:rel-doc}{vrel} are obtained in this way starting from an existential elementary linear doctrine. 
However, the characterization of relational doctrines arising in these ways is left for future work.

\section{Conclusions}
\label{sect:conclu} 

We introduced relational doctrines as a functorial description of the essence of the calculus of relations, 
Relying on this structure, we defined quotients and a universal construction adding them, dubbed intensional quotient completion, capable to cover quantitative examples as well. 
Then, we studied extensional equality in relational doctrines, showing it captures various notion of separation in metric and topological structures. 
Moreover, we described a universal construction forcing extensionality, thus separation, analysing how it interacts with quotients. 
this led us to the definition of the extensional quotient completion as a combination of the two previously introduced constructions. 
We characterized its essential image through the notion of projective cover and applied this result to show that, under suitable conditions, doctrines of algebras are obtained as the extensional quotient completion of their restriction to free algebras. 
Finally, we compared relational doctrines with two important classes of examples: 
ordered categories with involution, proving these correspond to relational doctrines having both extensional equality and the rule of unique choice, 
and existential elementary doctrines, showing they correspond to cartesian relational doctrines satisfying the modular law 

There are many directions for future work. 
The first one 
is the study of choice rules in the framework of relational doctrines, extending known results for doctrines \cite{xyz, Maietti-Rosolini16}, giving them a quantitative interpretation, for instance in terms of completeness, 
following the connection between \RUC and Cauchy completeness pointed out in \cite{Lawvere73}. 
We have already made some steps in this direction \cite{DagninoP24}. 
Moreover, this could lead us to the definition of a quantitative counterpart of the tripos-to-topos construction, generalising known results 
\cite{Jonas,MaiettiME:exacf}, which could generate categories of complete (partial) metric spaces. 

We also plan to bring the study of relations to the proof-relevant setting of type theories. 
Algebraically this can be done moving from doctrines to arbitrary fibrations as it is common practice, 
which in this case means working with framed bicategories (a.k.a. equipements) \cite{Shulman08}, endowed with a suitable involution to model the relational converse. 
These are a special kind of double categories and equivalence relation can be regarded as (symmetric) monads inside them. 
Then, it would be interesting studying the connections between our results and those for monads in such double categories 
\cite{Shulman08,CrutwellS10,FioreGK11}. 
On the syntactic side, instead, things are much less clear: developing a proper syntax and rules for a ``relational type theory'' is something interesting per se.
Actually, we do not even have a syntactic calculus behind relational doctrines. 
Then, another interesting direction is to design it, possibly in a diagrammatic way, for instance in the style of string diagrams. 

Finally, a promising direction would be the use of relational doctrines as an abstract framework where to formulate and develop relational techniques  used in the study of programming languages and software systems, 
such as (coalgebraic) bisimulation, program equivalence or operational semantics, as well as, quantitative equational theories and rewriting.

\bibliographystyle{elsarticle-num} 
\bibliography{biblio}

\end{document}